\documentclass[a4paper,10pt]{article}
\usepackage[T1]{fontenc}
\usepackage[english]{babel}
\usepackage[left= 2.8cm, bottom = 3.7cm, top = 3.5cm, right= 2.8cm]{geometry}
\usepackage[citecolor=blue,colorlinks=true]{hyperref}
\usepackage{authblk}

\usepackage{amsmath,amssymb,xcolor,latexsym,theorem}

\newcommand{\assign}{:=}
\newcommand{\backassign}{=:}
\newcommand{\cdummy}{\cdot}

\newcommand{\infixand}{\text{ and }}
\newcommand{\nobracket}{}

\newcommand{\tmcolor}[2]{{\color{#1}{#2}}}

\newcommand{\tmop}[1]{\ensuremath{\operatorname{#1}}}
\newcommand{\tmtextbf}[1]{\text{{\bfseries{#1}}}}
\newcommand{\tmtextit}[1]{\text{{\itshape{#1}}}}
\newenvironment{proof}{\noindent\textbf{Proof\ }}{\hspace*{\fill}$\Box$\medskip}

\newcommand{\nonconverted}[1]{\mbox{}}

\newtheorem{theorem}{Theorem}[section]
\newtheorem{corollary}[theorem]{Corollary}
\newtheorem{lemma}[theorem]{Lemma}

\newtheorem{proposition}[theorem]{Proposition}
{\theorembodyfont{\rmfamily}\newtheorem{remark}[theorem]{Remark}}

\numberwithin{equation}{section}

\begin{document}

\title{Fluctuation dynamics in randomly advected Navier--Stokes equations
  below critical scaling}

\author[1,2]{Arnaud Debussche\thanks{arnaud.debussche (at) ens-rennes.fr}}
\author[3]{Martina Hofmanov{\'a}\thanks{hofmanova (at) math.uni-bielefeld.de}}
\affil[1]{\small Univ Rennes, CNRS, IRMAR - UMR 6625, F-35000 Rennes, France.}
\affil[2]{\small Institut universitaire de France (IUF).
}
  \affil[3]{\small Fakult\"at f\"ur Mathematik, Universit\"at Bielefeld, Postfach 10 01 31, D-33501 Bielefeld, Germany.}

\date{}

\maketitle

{\let\thefootnote\relax
}

\begin{abstract}
  We study randomly advected incompressible Navier--Stokes equations, where
  the advecting field is a mean-zero, divergence-free, space-time stationary
  velocity field with smooth order-one correlations. We introduce a
  two-parameter family of models in which the advection is accelerated on a
  fast temporal scale $\varepsilon^2$ and has spatial correlation length
  $\delta$; the critical regime $\varepsilon = \delta$ corresponds to the
  natural parabolic scaling of the Navier--Stokes equation. In the full
  subcritical regime $\varepsilon = o (\delta)$, we prove a law of large
  numbers in dimensions $d = 2, 3$: the solutions converge to a deterministic
  Navier--Stokes system with an enhanced diffusion coefficient given by a
  Green--Kubo formula. In two space dimensions, under the slightly stronger
  assumption $\varepsilon = o (\delta^{1 + \iota})$ for some $\iota > 0$, we
  identify the leading-order fluctuations: after subtracting deterministic
  macroscopic corrections satisfying a nonlinear system of Navier--Stokes
  type, the rescaled fluctuations converge to a Gaussian field solving a
  linearized Navier--Stokes equation driven by multiplicative space-time white
  noise.
  
  \tmtextbf{Keywords:} Navier--Stokes equations, random advection, enhanced
  diffusion, Gaussian fluctuations, multiscale analysis
\end{abstract}

\tableofcontents

\section{Introduction}

The large-scale behavior of fluid flows subject to random advection by a
turbulent environment is a fundamental problem at the intersection of
mathematical analysis, stochastic homogenization, and the theory of
turbulence. From a physical standpoint, the effective enhancement of diffusion
by turbulent fluctuations -- referred to as eddy diffusivity or turbulent
diffusion in the physics literature -- has been recognized as a central
mechanism in macroscopic fluid dynamics since the foundational works of Taylor
{\cite{MR1577363}}, Richardson {\cite{Rich26}}, and Kolmogorov {\cite{K41}}.
From a mathematical standpoint, a rigorous derivation of effective macroscopic
behavior from a microscopic stochastic model for the Navier--Stokes equations
requires controlling the interplay between the nonlinearity and the fast
random environment, and identifying the precise structure of the randomness
that survives at the macroscopic scale. The present work establishes both a
law of large numbers and, in two space dimensions, an invariance principle for
the randomly advected Navier--Stokes equations in a natural multiscale regime,
and determines the effective macroscopic dynamics at both levels.

The mathematical analysis of random advection problems originates in the
Kraichnan model {\cite{kraichnan1968small}}, which describes passive scalar
transport by a Gaussian velocity field that is delta-correlated in time.
Despite its linear structure, the model exhibits a wide range of nontrivial
phenomena, including anomalous scaling of structure functions in the
long-range regime {\cite{Kraichnan1994}}, {\cite{GawedzkiKupiainen1995}}. The
survey {\cite{FalkKupVerg2001}} provides a comprehensive account of these
developments from a physical perspective. The present work addresses random
advection in a structurally different regime: the white-in-time Gaussian model
of Kraichnan is replaced by a velocity field with nontrivial temporal
correlations, and the transport mechanism is coupled to the full
Navier--Stokes dynamics rather than to a linear scalar equation.

Consider the Navier--Stokes equations advected by a divergence-free,
mean-zero, stationary random field $m$ with smooth correlations at unit
space-time scales,
\[ \partial_t u + u \cdummy \nabla u + \nabla p = \Delta u + m \cdummy \nabla
   u, \quad \tmop{div} u = 0. \]
To probe large scales, we perform the Navier--Stokes parabolic rescaling
$u_{\lambda} (t, x) \assign \lambda u (\lambda^2 t, \lambda x)$, $\lambda \to
\infty$, which yields the rescaled equation
\[ \partial_t u_{\lambda} + u_{\lambda} \cdummy \nabla u_{\lambda} + \nabla
   p_{\lambda} = \Delta u_{\lambda} + \lambda m_{\lambda} \cdummy \nabla
   u_{\lambda}, \quad \tmop{div} u_{\lambda} = 0, \]
where $m_{\lambda} (t, x) \assign m (\lambda^2 t, \lambda x)$. This scaling
accelerates both temporal and spatial oscillations at the parabolic rate and
is therefore the natural regime for studying macroscopic limits. Unlike
singular SPDEs driven by a scale-invariant noise, the random field $m$ has a
fixed correlation length and is not invariant under this rescaling. In
particular, the correct central limit normalization is $\lambda^{1 + d / 2}
m_{\lambda}$, so the factor $\lambda m_{\lambda}$ vanishes in the limit. This
already suggests that randomness does not survive at the law of large numbers
level: the leading macroscopic effect of the transport is deterministic, i.e.
enhanced diffusion, while stochastic contributions can only emerge at the
level of fluctuations, after the appropriate $\lambda^{d / 2}$ normalization.

This separation between a deterministic leading order and Gaussian
fluctuations is far from obvious a priori. At the nonlinear level, the
interaction between the random advection and the Navier--Stokes nonlinearity
could in principle generate persistent correlations that distort the Gaussian
character of the limit, or produce non-Gaussian fluctuations of the kind
observed in the Kraichnan model at critical scaling. The fact that the
fluctuation limit is itself a solution of a linearized Navier--Stokes system
-- driven by a multiplicative space-time white noise whose intensity is
explicitly determined by the covariance of the original random field --
reflects a remarkable simplification. The nonlinearity, rather than entangling
the two levels, acts only through the deterministic background flow $u$ around
which the fluctuations are linearized.

From a multiscale perspective, the regime described above is a threshold case:
temporal and spatial oscillations are accelerated at the same rate, so no a
priori separation of scales is available. To analyze how macroscopic behavior
emerges from such random advection, it is therefore natural to consider a
family of microscopic models in which the temporal and spatial mixing
properties of the environment can be varied independently.

Guided by this perspective, we introduce a two-parameter family of models in
which the random advection evolves on a fast temporal scale $\varepsilon^2$
and has spatial correlation length $\delta$, $m^{\varepsilon, \delta} (t, x)
\assign m (\varepsilon^{- 2} t, \delta^{- 1} x)$. The case $\varepsilon =
\delta$ corresponds precisely to the parabolic rescaling above and represents
a scaling-critical regime in which temporal and spatial mixing occur at
comparable rates. By contrast, the full subcritical regime $\varepsilon = o
(\delta)$ introduces a genuine separation between temporal and spatial mixing,
placing the problem in a setting in which averaging and fluctuation effects
can be rigorously analyzed. In this regime, we establish a law of large
numbers with an effective enhanced diffusion and, in two dimensions, identify
the corresponding Gaussian fluctuation dynamics.

The role of the interplay between temporal and spatial mixing rates in
determining macroscopic transport properties was emphasized in the physics
literature by Avellaneda and Majda {\cite{AM92}}, who demonstrated that the
form of the eddy-diffusivity equation depends sensitively on the relative
strength of temporal and spatial decorrelation and identified a phase diagram
in the parameter space of velocity statistics. In particular, the
Kolmogorov--Obukhov regime lies on a boundary separating qualitatively
distinct macroscopic behaviors. Our two-parameter family $(\varepsilon,
\delta)$ is designed to reflect this dependence in a nonlinear setting by
varying temporal and spatial mixing scales independently; the present work
analyzes the strictly subcritical regime in which temporal mixing occurs on a
faster scale than spatial mixing.

Let $\varepsilon, \delta \in (0, 1]$. We study the following randomly advected
Navier--Stokes equations on the torus $\mathbb{T}^d$, $d = 2, 3,$
\begin{equation}
  \partial_t u^{\varepsilon, \delta} + u^{\varepsilon, \delta} \cdummy \nabla
  u^{\varepsilon, \delta} + \nabla p^{\varepsilon, \delta} = \Delta
  u^{\varepsilon, \delta} + \varepsilon^{- 1} m^{\varepsilon, \delta} \cdummy
  \nabla u^{\varepsilon, \delta}, \quad \tmop{div} u^{\varepsilon, \delta} =
  0, \quad u^{\varepsilon, \delta} (0) = u_0 . \label{eq:u}
\end{equation}
Here $m^{\varepsilon, \delta} (t, x) \assign m (\varepsilon^{- 2} t, \delta^{-
1} x)$, where $m$ is a divergence-free and mean-zero, centered, space-time
stationary Ornstein--Uhlenbeck process on $L^2 (\mathbb{T}^d ; \mathbb{R}^d)$.
Its covariance is given by
\[ \mathbb{E} [m (t, x) \otimes m (s, y)] = \frac{1}{2} e^{- | t - s |} Q (x -
   y), \]
where the spatial covariance $Q : \mathbb{T}^d \rightarrow \mathbb{R}^{d
\times d}$ is of the form $Q = q \tmop{Id}_{\mathbb{R}^d} = (K \ast K)
\tmop{Id}_{\mathbb{R}^d}$. The kernel $K : \mathbb{R}^d \rightarrow
\mathbb{R}$ is assumed to be smooth, compactly supported and rotation
invariant and we denote $q = K \ast K$. Throughout the paper, we focus on the
full subcritical regime which corresponds to the setting $\varepsilon = o
(\delta)$. The critical case $\varepsilon = \delta$ lies at the threshold of
our approach and would require substantially different techniques.

For every fixed $\varepsilon, \delta \in (0, 1]$, the existence of a
probabilistically and analytically weak solution to \eqref{eq:u} in $d = 2, 3$
is obtained by Galerkin approximation and stochastic compactness argument.
Additionally, due to the divergence-free constraint on $m$ leading to the
cancellation in the $L^2$-inner product
\[ \varepsilon^{- 1} \langle u^{\varepsilon, \delta}, m^{\varepsilon, \delta}
   \cdummy \nabla u^{\varepsilon, \delta} \rangle = - \varepsilon^{- 1}
   \langle m^{\varepsilon, \delta} \cdummy \nabla u^{\varepsilon, \delta},
   u^{\varepsilon, \delta} \rangle = 0, \]
the solutions satisfy the energy inequality
\begin{equation}
  \| u_t^{\varepsilon, \delta} \|_{L^2}^2 + 2 \int_0^t \| \nabla
  u_s^{\varepsilon, \delta} \|_{L^2}^2 d s \leqslant \| u_0 \|_{L^2}^2,
  \label{eq:energy}
\end{equation}
which directly provides a pathwise uniform energy estimate for
$u^{\varepsilon, \delta}$ in $L^{\infty} (0, T ; L^2) \cap L^2 (0, T ; H^1)$.

In $d = 2$, the solutions $u^{\varepsilon, \delta}$ are unique, hence
probabilistically strong by the Gy\"ongy--Krylov type argument, and the energy
inequality becomes an equality.

Our first main result establishes a law of large numbers.

\begin{theorem}
  \label{thm:1}Let $d = 2, 3$, assume $\varepsilon = o (\delta)$ and $u_0 \in
  L^2 (\mathbb{T}^d ; \mathbb{R}^d)$ is mean- and divergence-free. Then any
  family of probabilistically and analytically weak solutions $u^{\varepsilon,
  \delta}$, $\varepsilon, \delta \in (0, 1]$, to \eqref{eq:u} satisfying the
  energy inequality \eqref{eq:energy}, converges in law along a subsequence in
  $L^2 (0, T ; L^2)$, as $\varepsilon, \delta \rightarrow 0$, to a solution of
  the deterministic Navier--Stokes system
  \begin{equation}
    \partial_t u + u \cdummy \nabla u + \nabla p = (1 + \nu) \Delta u, \quad
    \tmop{div} u = 0, \quad u (0) = u_0, \label{eq:ulim}
  \end{equation}
  where the enhanced viscosity is given by $\nu = \frac{q (0)}{16} =
  \frac{1}{16} \| K \|_{L^2}^2$ if $d = 2$ and $\nu = \frac{q (0)}{5} =
  \frac{1}{5} \| K \|_{L^2}^2$ if $d = 3$.
\end{theorem}

The enhanced diffusion coefficient $\nu$ is explicitly computable in terms of
the spatial covariance $Q$ of the advection field: it is proportional to $q
(0) = \| K \|_{L^2}^2$, the common diagonal entry of $Q (0)$, reflecting the
isotropy of the environment. The precise numerical factor arises from the
Leray projection onto divergence-free fields. This formula is an instance of
the Green--Kubo relation, which expresses macroscopic transport coefficients
as time-integrated autocorrelations of microscopic currents; see, for example,
Green {\cite{Green1954}}, Kubo {\cite{MR149885}}, 
  and the review by Majda and Kramer
{\cite{MajdaKramer1999}}.

In the linear setting of passive scalar transport, integral representations
and sharp bounds for the effective diffusivity were established by Avellaneda
and Majda {\cite{AM91}}. In the present isotropic Ornstein--Uhlenbeck setting,
the enhanced viscosity admits an explicit closed-form expression, whose
precise value emerges from the interaction between the covariance structure
and the Leray projection.

\begin{remark}
  Since the limit \eqref{eq:ulim} is deterministic, the convergence in
  Theorem~\ref{thm:1} holds in probability whenever the solutions
  $u^{\varepsilon, \delta}$ are defined on a common probability space. In $d =
  2$, uniqueness of weak solutions implies that the full sequence converges in
  probability. In $d = 3$, such solutions can be constructed for a countable
  subset of vanishing parameters $\varepsilon, \delta \in (0, 1]$ by applying
  the Galerkin approximation and stochastic compactness argument jointly for
  all such parameters.
\end{remark}

\begin{theorem}
  \label{thm:2}Let $d = 2$, assume $\varepsilon = o (\delta^{1 + \iota})$ for
  some $\iota > 0$ and $u_0 \in L^2 (\mathbb{T}^2 ; \mathbb{R}^2)$ is mean-
  and divergence-free. There exist deterministic macroscopic corrections
  $v^{\varepsilon, \delta}$, $\varepsilon, \delta \in (0, 1]$, satisfying a
  Navier--Stokes type system with forcing and full quadratic self-interaction,
  such that
  \[ \frac{u^{\varepsilon, \delta} - v^{\varepsilon, \delta} - u}{\delta^{d /
     2}} \rightarrow z \quad \text{in probability in}\quad L^2 (0, T
     ; H^{- \beta}) \quad\text{for any}\quad\beta > 0. \]
  Here the limit $z$ is a Gaussian field solving the stochastic linearized
  Navier--Stokes system
  \[ d z + [z \cdummy \nabla u + u \cdummy \nabla z+ \nabla p_z] d t  = (1 +
     \nu) \Delta z d t + \chi d W \cdummy \nabla u, \quad \tmop{div} z = 0,
     \quad z (0) = 0, \]
  where $W$ is a space-time white noise on the divergence-free, mean-zero
  subspace of $L^2 (\mathbb{T}^2 ; \mathbb{R}^2)$, and $\chi =
  (\mathcal{F}_{\mathbb{R}^2} K) (0)$ denotes the noise intensity.
\end{theorem}

The macroscopic corrections $v^{\varepsilon, \delta}$ account for
contributions that persist after rescaling but do not contribute to the
limiting fluctuation; their deterministic character is essential for
Theorem~\ref{thm:2} to constitute a genuine central limit theorem, in which
the fluctuations of $u^{\varepsilon, \delta}$ are measured around a
deterministic reference. As in related homogenization results
{\cite{MR4198718}}, deterministic intermediate-scale corrections must be
subtracted before the genuine Gaussian fluctuations become visible. Whereas
in~{\cite{MR4198718}} these corrections satisfy a linear equation, a
distinctive feature of the present nonlinear setting is that they satisfy a
Navier--Stokes type equation with full quadratic self-interaction.

The fluctuation theorem reveals the precise mechanism by which randomness
re-enters the macroscopic description after being averaged out at leading
order. The multiplicative structure of the limit noise is not a modeling
choice but an output of the analysis: it arises from the interaction of the
$\delta^{d / 2}$-normalized correctors with the nonlinear background. Its
intensity $\chi = (\mathcal{F}_{\mathbb{R}^2} K) (0)$ is determined solely by
the zero-frequency component of the covariance kernel, reflecting the
large-scale amplitude of the environment.

The strict subcriticality assumption $\varepsilon = o (\delta^{1 + \iota})$
with $\iota > 0$ is essential in our method: the quantitative power gap
guarantees that the multiscale expansion terminates at finite order. Whether
the result extends to the weaker regime $\varepsilon = o (\delta)$ or beyond
the subcritical regime altogether are natural open questions. We expect that in the critical case the enhanced diffusion
has a different form and that its expression combines the correlation function of the noise and the Green kernel
of the Stokes operator. The restriction
to $d = 2$ is structural, see Remark~\ref{r:34}.

Enhanced dissipation and regularization by noise for fluid equations driven by
transport-type stochastic perturbations have been studied extensively by
Flandoli and coauthors (see e.g. {\cite{MR4304694}}, {\cite{MR4238216}},
{\cite{MR4714774}}, {\cite{MR4265023}}, {\cite{MR4174071}},
{\cite{MR4709553}}, among others). In that literature, the stochastic forcing
is white in time -- reflecting a model in which the limit $\varepsilon
\rightarrow 0$ has already been performed, leaving only $\delta \rightarrow 0$
to be taken -- and supported on finitely many prescribed Fourier modes; both
the Fourier support and the amplitude normalization are tuned to the
approximation parameter rather than inherited from a fixed physical model.
These assumptions and the specific form of the noise considerably simplify the
analysis. Quantitative results are commonly obtained for linear equations or
at the vorticity level.

By contrast, the present work starts from a single fixed velocity field $m$
with compact spatial correlations and exponentially decaying temporal
correlations, and studies its effect across scales through the natural
two-parameter family $m^{\varepsilon, \delta} (t, x) = m (\varepsilon^{- 2} t,
\delta^{- 1} x)$. Enhanced diffusion then emerges as an output of the
multiscale analysis rather than being engineered into the model. Working at
the level of the full Navier--Stokes velocity formulation, we moreover
identify the Gaussian fluctuation dynamics, a level of description not
addressed in the works cited above. Randomness is not eliminated but
reorganized across scales, reappearing at the next order as Gaussian
fluctuations -- a perspective closer in spirit to statistical mechanics and
homogenization theory than to the noise-by-design approach of the
regularization by noise literature.

The framework of Hairer and Pardoux {\cite{MR3417505}}, {\cite{MR4198718}}
addresses homogenization for SPDEs where the random perturbation enters
without transport structure and is subcritical in the sense of regularity
structures even at parabolic scaling. By contrast, the transport term
$\varepsilon^{- 1} m^{\varepsilon, \delta} \cdot \nabla u^{\varepsilon,
\delta}$ at $\varepsilon = \delta$ has the same homogeneity as the Laplacian,
placing our equation at a critical threshold for regularity structures. This
is precisely the regime our approach does not reach. Whether the macroscopic
limit exists at critical scaling, and what form it takes, remain open.

Cannizzaro and Kiedrowski {\cite{MR4709543}} and Cannizzaro, Gubinelli and
Toninelli {\cite{MR4719971}} study critical stochastic fluid-type equations
under an explicitly invariant Gaussian measure, which permits generator-based
Wiener chaos analysis and gives access to scaling-critical and weak-coupling
regimes. No comparable invariant Gaussian structure is available for the full
randomly advected Navier--Stokes system: while the fast Ornstein--Uhlenbeck
field has a Gaussian invariant measure, the coupled Navier--Stokes evolution
does not preserve it.

\subsection{Strategy of the proof}

The proof is shaped by four structural challenges: the near-critical scaling
$\varepsilon = o (\delta^{1 + \iota})$ with small $\iota > 0$, the persistent
expectations of the corrector observables, the Navier--Stokes nonlinearity at
energy-level regularity, and the presence of the Leray projection
$\mathbb{P}$. Each creates distinct difficulties, which we describe in turn
before outlining how they are overcome.

\tmtextbf{Near-criticality and the corrector hierarchy.} To remove the large
oscillatory contributions we invert the Ornstein--Uhlenbeck generator
$-\mathcal{M}^{\delta}$ of the environment and build a hierarchy of
correctors. In the law-of-large-numbers regime the expansion terminates after
two correctors and produces a deterministic second-order correction to the
Laplacian. Its coefficient is an instance of the Green--Kubo relation,
expressing the effective diffusion coefficient $\nu$ as a time-integrated
autocorrelation of the microscopic current; the precise numerical prefactor is
determined entirely by the action of $\mathbb{P}$ on the isotropic covariance,
with the divergence-free constraint encoded in the prefactor.

For the fluctuation analysis in $d = 2$, writing $z^{\varepsilon, \delta} :=
(u^{\varepsilon, \delta} - v^{\varepsilon, \delta} - u) / \delta^{d / 2}$ for
the rescaled fluctuation variable, additional layers of the expansion become
visible. We construct a hierarchy of correctors $\varphi_{\sigma}$ indexed by
finite binary sequences $\sigma$, where the digit $0$ encodes an application
of the transport term $\varepsilon^{- 1} \mathbb{P} (m^{\varepsilon, \delta}
\cdot \nabla \cdummy)$ and the digit $1$ an application of the
Laplacian; $M_{\sigma}$ and $L_{\sigma}$ denote the respective counts, so that
$\varphi_{\sigma}$ carries the weight $\varepsilon^{M_{\sigma} + 2
L_{\sigma}}$. The hierarchy is built by iterating the inversion of
$-\mathcal{M}^{\delta}$ along branches of the expansion. Each corrector is a
polynomial functional of the Gaussian field, and the action of
$(-\mathcal{M}^{\delta})^{- 1}$ is explicit in Wiener chaos: it produces a
finite sum of terms, each obtained by pairing tensor slots against the
covariance $\frac{1}{2} \mathcal{Q}_{\delta}$, with each contraction reducing
the polynomial degree by two. The hierarchy is therefore algebraically
explicit; see Section~\ref{s:high}. Its length is governed by the truncation
condition $\iota > d / (2 N)$, which ensures that terms of generation $N$
vanish after the fluctuation rescaling. The assumption $\iota > 0$ guarantees
that such a finite $N$ exists and the expansion terminates.

\tmtextbf{Absorbing the effective operators into the semigroup.} At each level
of the hierarchy, certain expectations with respect to the invariant measure
of $m^{\varepsilon, \delta}$ do not vanish and cannot be cancelled by further
correctors. Each such expectation decomposes as $S_{\sigma} u^{\varepsilon,
\delta} + R_{\sigma} u^{\varepsilon, \delta}$, where $S_{\sigma}$ is a
second-order divergence-form operator and $R_{\sigma}$ is higher-order. The
leading part $\varepsilon^{M_{\sigma} + 2 L_{\sigma} - 1} S_{\sigma}
u^{\varepsilon, \delta}$ does not vanish after division by $\delta^{d / 2}$
and cannot be treated as an error. A natural attempt would be to include these
terms directly in the macroscopic correction $v^{\varepsilon, \delta}$, but
since they depend on $u^{\varepsilon, \delta}$ this would make
$v^{\varepsilon, \delta}$ random and invalidate the central limit theorem. The
resolution is to include the operators $\varepsilon^{M_{\sigma} + 2 L_{\sigma}
- 1} S_{\sigma}$ as perturbations of $(1 + \nu) \mathbb{P} \Delta$ into the
semigroup generator. This is viable because translation invariance of the
stationary environment makes each $S_{\sigma}$ a Fourier multiplier, and
subcriticality ensures the perturbation is small relative to the Laplacian.
The perturbed generator defines a $C_0$-semigroup via a direct Fourier
multiplier construction, with all estimates following from an elementary
energy argument at the symbol level.

\tmtextbf{The nonlinearity and critical regularity.} At the level of the
approximate equation the only available control is the pathwise uniform energy
bound $u^{\varepsilon, \delta} \in L^{\infty} (0, T ; L^2) \cap L^2 (0, T ;
H^1)$. After linearizing around the deterministic limit $u$, the convective
term must be controlled at scale-invariant regularity, with neither time
smallness nor derivative surplus available, see Section~\ref{s:convective}.
This critical structure is intrinsic to the velocity formulation: it is the
combination of the Navier--Stokes nonlinearity with the energy-level
regularity of $u^{\varepsilon, \delta}$ that forces the use of a critical
endpoint estimate for the convective term.

Higher-order correctors introduce progressively more spatial derivatives which
cannot be absorbed at energy-level regularity. We introduce a frequency
cut-off $R > 1$: the high-frequency component $P_{> R} (z^{\varepsilon,
\delta} - z)$ is controlled by uniform a priori bounds, while the
low-frequency component $P_{\leqslant R} (z^{\varepsilon, \delta} - z)$ is
estimated through the mild formulation, with each derivative either hitting
$m^{\varepsilon, \delta}$ at cost $\delta^{- 1}$ or being transferred to the
test function at cost $R$; the choice $R = \delta^{- 1}$ is the natural one
balancing the frequency truncation against the homogenization scale and
verified to be sufficient in Section~\ref{s:removal}.

To close the fluctuation estimate without any smallness assumption on the
initial data, we incorporate an exponential weight $e^{- \lambda t}$ into the
semigroup. The critical terms are then absorbed into the left-hand side by
decomposing the background velocity into a part that vanishes in probability
as $\varepsilon, \delta \rightarrow 0$ and a smooth Galerkin truncation for
which the exponential decay of the semigroup yields smallness for $\lambda$
large; see Section~\ref{s:concl}.

\tmtextbf{Macroscopic corrections.} A striking consequence of absorbing the
operators $\varepsilon^{M_{\sigma} + 2 L_{\sigma} - 1} S_{\sigma}$ into the
semigroup generator is that all intermediate-scale contributions -- those too
large to vanish after the $\delta^{d / 2}$ rescaling but too small to survive
at leading order -- can be collected into a purely deterministic correction
$v^{\varepsilon, \delta}$, with forcing terms $\varepsilon^{M_{\sigma} + 2
L_{\sigma} - 1} S_{\sigma} u$, evaluated at the deterministic limit. Since the
corrections can be of order $\varepsilon^2 \delta^{- 2}$, the quadratic
nonlinearity $v^{\varepsilon, \delta} \cdot \nabla v^{\varepsilon, \delta}$
cannot be controlled as a perturbation and must be retained: $v^{\varepsilon,
\delta}$ solves a Navier--Stokes type system with full quadratic
self-interaction; see Section~\ref{s:v}.

\tmtextbf{The Leray projection.} The Leray projection $\mathbb{P}$ complicates
the analysis in two independent ways. First, it obstructs the direct
contraction argument that would otherwise immediately yield the enhanced
diffusion coefficient: without $\mathbb{P}$, integration by parts reduces the
approximate enhanced diffusion operator $S_{\delta}$ to a plain multiple of
$\Delta$. In the presence of $\mathbb{P}$, this immediate contraction argument
breaks down: $S_{\delta}$ involves a kernel of Calder{\'o}n--Zygmund type in
which $\mathbb{P}$ is intertwined with $Q_{\delta}$, and identifying its limit
$\nu$ requires the careful Fourier analysis carried out in
Section~\ref{s:quali} and upgraded to a quantitative error estimate in
Section~\ref{s:quanti}. Second, at higher levels of the expansion, contraction
terms in the Wiener chaos produce two-variable kernels in which $\mathbb{P}$
becomes intertwined with derivatives of $Q_{\delta}$, placing them outside the
Calder{\'o}n--Zygmund framework. We bypass this by estimating inner products
before taking expectations, so that at most one copy of $m^{\varepsilon,
\delta}$ appears in any bound, replacing the composition of two singular
operators by a product of separate $L^p$-estimates; see Section~\ref{s:gen}.

\tmtextbf{Summary.} The proof combines inversion of the Ornstein--Uhlenbeck
generator, Wiener chaos contraction identities, Fourier multiplier semigroup
theory, and sharp derivative counting under minimal regularity.
Near-criticality is handled by a finite corrector hierarchy of length dictated
by $\iota$, with the persistent expectations absorbed into a perturbed
semigroup generator controlled by elementary symbol estimates. The macroscopic
corrections solve a full Navier--Stokes type system to accommodate their
potentially large size. The critical nonlinearity is tamed by frequency
localization and an exponential weight argument, closed by a final absorption
argument at $R = \delta^{- 1}$. The Leray projection is handled throughout by
a probabilistic approach to kernel estimates that avoids the composition of
singular operators.

\subsection{Organization of the paper}

Section~\ref{s:noise} introduces the setup, notation, and the
Ornstein--Uhlenbeck process $m^{\varepsilon, \delta}$. Section~\ref{s:2}
carries out the second-order expansion producing the enhanced diffusion
operator. Section~\ref{s:quali} identifies its limit via Fourier analysis.
Section~\ref{s:LLN} completes the proof of Theorem~\ref{thm:1} via stochastic
compactness. Section~\ref{s:high} develops the higher-order corrector
hierarchy via Wiener chaos. Section~\ref{s:setup} sets up the decomposition of
the fluctuation dynamics and introduces the frequency cut-off.
Section~\ref{s:semigroup} constructs the semigroup. Section~\ref{s:convective}
establishes the endpoint convective estimate. Section~\ref{s:10} constructs
the limiting fluctuation $z$ and the macroscopic correction $v^{\varepsilon,
\delta}$. Section~\ref{s:quanti} establishes the quantitative enhanced
diffusion estimates. Section~\ref{s:error} and Section~\ref{s:concl} treat the
remainder terms and carry out the final absorption argument, completing the
proof of Theorem~\ref{thm:2}. Appendix~\ref{s:a} collects auxiliary technical
results.

\subsection*{Acknowledgment} The research of M.H. was funded by the European Research Council
  (ERC) under the European Union's Horizon 2020 research and innovation
  programme (grant agreement No. 949981). A.D. benefits from the support of the French government ``Investissements d'Avenir'' program integrated to France 2030, bearing the following reference ANR-11-LABX-0020-01.

\section{Setup and assumptions}\label{s:noise}

\subsection{Notations}\label{s:notation}

We consider the $d$-dimensional torus $\mathbb{T}^d = [- \pi, \pi]^d$ and
denote by $H$ the subspace of $L^2$-integrable mean- and divergence-free
vector fields $f : \mathbb{T}^d \rightarrow \mathbb{R}^d$. We
denote by $\langle \cdummy, \cdummy \rangle$ the associated $L^2$-inner
product. By $\mathbb{P}$ we denote the Leray projection onto $H$ and we
recall that it is continuous on the Lebesgue space $L^p (\mathbb{T}^d ;
\mathbb{R}^d)$ for all $p \in (1, \infty)$. For $k \in \mathbb{Z}_0^d \assign
\mathbb{Z}^d \setminus \{ 0 \}$, let $\{ a_{k, 1}, \ldots, a_{k, d - 1}, k / |
k | \}$ be an orthonormal basis of $\mathbb{R}^d$ so that $a_{k, \alpha} =
a_{- k, \alpha}$ and define $\sigma_{k, \alpha} \assign (2 \pi)^{- d / 2}
a_{k, \alpha} e^{i k \cdummy}$. Then $(\sigma_{k, \alpha})_{k \in
\mathbb{Z}_0^d, \alpha = 1, \ldots ,d - 1}$ is an orthonormal basis of $H$. The
Fourier transform and its inverse are defined as
\[ \mathcal{F}f (\ell) \assign \hat{f} (\ell) \assign \int_{\mathbb{T}^d} f
   (x) e^{- i \ell \cdummy x} d x, \qquad \mathcal{F}^{- 1} f (x) \assign (2
   \pi)^{- d} \sum_{\ell \in \mathbb{Z}^d} f (\ell) e^{i \ell \cdummy x} . \]
On the full space $\mathbb{R}^d$ we define similarly
\[ \mathcal{F}_{\mathbb{R}^d} f (\ell) \assign \int_{\mathbb{R}^d} f (x) e^{-
   i \ell \cdummy x} d x, \qquad \mathcal{F}_{\mathbb{R}^d}^{- 1} f (x)
   \assign (2 \pi)^{- d} \int_{\mathbb{R}^d} f (\ell) e^{i \ell \cdummy x} d
   \ell \]
and we recall the Fourier--Plancherel formula $\| \mathcal{F}_{\mathbb{R}^d} f
\|_{L^2} = (2 \pi)^{d / 2} \| f \|_{L^2}$.

\tmcolor{black}{For $p \in [1, \infty]$, $L^p = L^p (\mathbb{T}^d)$ denotes the
usual Lebesgue space, with the same notation for vector-valued functions. For
$k \in \mathbb{N}_0$ and $p \in [1, \infty]$, $W^{k, p} = W^{k, p}
(\mathbb{T}^d)$ denotes the classical Sobolev space with norm $\|f\|_{W^{k,
p}} \assign \sum_{| \beta | \leqslant k} \| \partial^{\beta} f\|_{L^p}$. For
$s \in \mathbb{R}$, the Sobolev space $H^s = H^s (\mathbb{T}^d)$ is defined by
the norm
\[ \|f\|_{H^s}^2 \assign (2 \pi)^{- d} \sum_{k \in \mathbb{Z}^d} (1 + |k|^2)^s
   \hspace{0.17em} | \hat{f} (k) |^2, \]
and for vector fields the norm is taken componentwise; with a slight abuse of
notation we write $H^s$ also for the subspace of mean- and divergence-free
vector fields, so that $H^0 = H$. Duality pairings between $H^s$ and $H^{- s}$
extend $\langle \cdummy, \cdummy \rangle$.

Let $(\Delta_j)_{j \geqslant - 1}$ denote the Littlewood--Paley blocks
associated with a standard dyadic partition of unity on $\mathbb{T}^d$ (see,
e.g., {\cite[Chapter~2]{MR2768550}}). For $s \in \mathbb{R}$, $p, q \in [1,
\infty]$, the Besov space $B^s_{p, q} = B^s_{p, q} (\mathbb{T}^d)$ is defined
by the norm
\[ \|f\|_{B^s_{p, q}} \assign \| (2^{js} \| \Delta_j f\|_{L^p})_{j \geqslant -
   1} \|_{\ell^q}, \]
so that $B^s_{2, 2} = H^s$ with equivalent norms.

For $p \in (1, \infty)$, $q \in [1, \infty]$, the Lorentz space $L^{p, q} =
L^{p, q} (\mathbb{T}^d)$ is defined via the decreasing rearrangement
$f^{\ast}$ by the quasi-norm
\[ \|f\|_{L^{p, q}} \assign \| t^{1 / p} f^{\ast} (t) \|_{L^q \left( (0,
   \infty), \frac{d t}{t} \right)}, \]
so that $L^{p, p} = L^p$; we refer to {\cite[Section~1.4]{MR2445437}} for its
basic properties, in particular H{\"o}lder's inequality and O'Neil's
convolution inequality {\cite{ONeil1963}} used in Appendix~\ref{s:a4}.

For a Banach space $X$, $L^q (0, T ; X)$ and $C ([0, T] ; X)$ denote the usual
spaces of Bochner integrable, respectively continuous, functions $[0, T] \to
X$, and for $\gamma \in (0, 1)$, $W^{\gamma, 2} (0, T ; X)$ denotes the
fractional Sobolev space of $f \in L^2 (0, T ; X)$ with
\[ [f]_{W^{\gamma, 2} (0, T ; X)}^2 \assign \int_0^T \hspace{-0.17em}
   \hspace{-0.17em} \int_0^T \frac{\|f (t) - f (s)\|_X^2}{|t - s|^{1 + 2
   \gamma}} d s d t < \infty, \]
normed by $\|f\|_{W^{\gamma, 2} (0, T ; X)}^2 \assign \|f\|_{L^2 (0, T ; X)}^2
+ [f]_{W^{\gamma, 2} (0, T ; X)}^2$.

Throughout, $a \lesssim b$ means $a \leqslant C b$ for a constant $C$
independent of the relevant parameters, in particular of $\varepsilon$ and
$\delta$; dependence on other quantities is indicated by subscripts if
necessary.} We use the notation $p+$ and $p-$ for exponents slightly larger and smaller than $p$, respectively. Their values may vary from line to line and are chosen whenever needed so that H\"older's inequality and related interpolation estimates apply.

\subsection{Noise construction}\label{s:m}

\tmcolor{black}{Throughout, $(\Omega, \mathcal{F}, (\mathcal{F}_t)_{t \geqslant
0}, \mathbf{P})$ denotes a filtered probability space satisfying the usual
conditions, and $\mathbb{E}$ the corresponding expectation.} Let $m_t (n)$,
for $t \geqslant 0$, with initial condition $m_0 = n$, be a centered ergodic
Markov process on a suitable subspace $E$ of $H$ with invariant measure $\nu^{1}$ (the superscript anticipates the rescaled family $\nu^\delta$
 introduced below)
and generator $\mathcal{M}$.

We write $m_t$, $t \geqslant 0$, to denote the stationary process (i.e.,
initialized with $n \sim \nu^{1}$) and distinguish $m_t (n)$ for the process
started from deterministic initial data $n$ from $m_t (x)$ for the stationary
process evaluated at $x \in \mathbb{T}^d$.

We additionally assume that the process $m$ is stationary with respect to
spatial translations, i.e., the law of $m_t (\cdummy + y)$ is the same as that
of $m_t$ for all $y \in \mathbb{T}^d$. Under these assumptions, we define the
effective spatial covariance kernel $Q^{i j} : \mathbb{T}^d \rightarrow
\mathbb{R}$ by
\begin{equation}
  Q^{i j} (x - y) = \int_0^{\infty} \mathbb{E} [m_0^i (x) m^j_t (y) + m^j_t
  (x) m_0^i (y)] d t, \qquad x, y \in \mathbb{T}^d . \label{eq:Q1}
\end{equation}
We further assume for simplicity that the components $m^i$, $m^j$ are
independent unless $i = j$, and that all components have the same variance.
Consequently, the matrix-valued kernel $Q  : \mathbb{T}^d
\rightarrow \mathbb{R}^{d \times d}$ is a multiple of the identity matrix,
symmetric and $Q (x) = q (x) \tmop{Id}_{\mathbb{R}^d}$ for some scalar
function $q  : \mathbb{T}^d \rightarrow \mathbb{R}$ with $q
(x) = q (- x)$. We assume that $q = K \ast K$ where the convolution kernel $K
: \mathbb{R}^d \rightarrow \mathbb{R}$ is smooth, compactly supported in a
small ball around the origin, symmetric (i.e. $K (x) = K (- x)$), and
isotropic (i.e. rotation invariant $K (x) = K (R x)$ for every rotation matrix
$R \in \tmop{SO} (d)$).

For $\delta \in (0, 1]$, we define the rescaled process $m^{\delta}$ as
$m^{\delta}_t (x) \assign m_t (\delta^{- 1} x)$, where $m_t$ denotes the
stationary process introduced above. The process $m$ is defined on
$\mathbb{T}^d = [- \pi, \pi]^d$ so $m^{\delta}$ is defined on $(\delta
\mathbb{T})^d$ and assuming $\delta^{- 1} \in \mathbb{N}$ we can understand it
also to be defined on $\mathbb{T}^d$: indeed, for $k \in \mathbb{Z}^d$
\[ m^{\delta} (x + 2 \pi k) = m (\delta^{- 1} x + \delta^{- 1} 2 \pi k) = m
   (\delta^{- 1} x) = m^{\delta} (x), \]
where for the second equality we require $\delta^{- 1} 2 \pi k \in 2 \pi
\mathbb{Z}^d$ hence $\delta^{- 1} \in \mathbb{N}$. Let us keep this assumption
from now on and consider all processes on the fixed torus $\mathbb{T}^d$. Similarly, since we are concerned with the limit $\varepsilon,\delta\to 0$, we assume throughout and without loss of generality that $\varepsilon\leqslant 1/2$.

Since~$m$~is stationary both in time and in space, the rescaled
process~$m^{\delta}$~remains stationary in time and space. We denote its law
at each time by~$\nu^{\delta}$, which is the pushforward measure
of~$\nu^{1}$~under the rescaling map~$n \mapsto (\Theta_{\delta} n) (\cdummy) = n
(\delta^{- 1} \cdot)$. That is: $m^{\delta}_t \sim \nu^{\delta}$ with
$\nu^{\delta}$ being the law of $\Theta_{\delta} n$ where $n \sim \nu^{1}$. Since the
process $m$ is $E$-valued, $m^{\delta}$ takes values in $\Theta_{\delta} E$ and for
any observable $\varphi : \Theta_{\delta} E \rightarrow V$ with a separable Banach
space $V$, it holds $\mathbb{E}_{\nu^{\delta}} [\varphi (n)] =\mathbb{E}_{\nu^{1}}
[\varphi (n (\delta^{- 1} \cdummy))]$. Consequently, by definition
of~$\nu^{\delta}$, we have for $s \in \mathbb{N}_0$ and $p \in [1, \infty]$
\begin{equation}
  \mathbb{E}_{\nu^{\delta}} [\| \nabla^s n \|_{L^p}] =\mathbb{E}_{\nu^{1}} [\|
  \nabla^s (n (\delta^{- 1} \cdummy)) \|_{L^p}] = \delta^{- s}
  \mathbb{E}_{\nu^{1}} [\| (\nabla^s n) (\delta^{- 1} \cdummy) \|_{L^p}] =
  \delta^{- s} \mathbb{E}_{\nu^{1}} [\| \nabla^s n \|_{L^p}] . \label{eq:7ddd}
\end{equation}
Since $\nu^{1}$ is centered, $\nu^{\delta}$ is centered as well. For an observable
$\varphi : \Theta_{\delta} E \rightarrow V$ we write $[\varphi]^{\bullet} :
=\mathbb{E}_{\nu^{\delta}} [\varphi] = \int \varphi \ 
d \nu^{\delta}$ for its expectation under $\nu^{\delta}$ and
$[\varphi]^{\circ} \assign \varphi - [\varphi]^{\bullet}$ for its mean-free
part.

Furthermore, the associated covariance kernel $Q_{\delta}$ satisfies
$Q_{\delta} = q_{\delta} \tmop{Id}_{\mathbb{R}^d}$ with $q_{\delta} (\cdummy)
= q (\delta^{- 1} \cdummy)$. We now define the rescaled kernel $K_{\delta} :
\mathbb{R}^d \rightarrow \mathbb{R}$ by $K_{\delta} (\cdummy) = \delta^{- d /
2} K (\delta^{- 1} \cdummy)$. Its support lies in a small ball around the
origin, and we extend $K_{\delta}$ periodically to the torus $\mathbb{T}^d$
without changing notation. Hence $q_{\delta} = K_{\delta} \ast K_{\delta}$.
Indeed,
\[ \int_{\mathbb{T}^d} K_{\delta} (x - z) K_{\delta} (z) d z = \delta^{- d}
   \int_{\mathbb{T}^d} K (\delta^{- 1} x - \delta^{- 1} z) K (\delta^{- 1} z)
   d z \]
\[ = \int_{\mathbb{R}^d} K (\delta^{- 1} x - z) K (z) d z = q (\delta^{- 1} x)
   = q_{\delta} (x) . \]
Finally, we define for $f : \mathbb{T}^d \rightarrow \mathbb{R}^d$ the
operator $\mathcal{Q}^{1 / 2}_{\delta} f \assign K_{\delta} \ast f$, in
coordinates
\[ (\mathcal{Q}^{1 / 2}_{\delta} f)^j (x) = \int_{\mathbb{T}^d} K_{\delta} (x
   - y) f^j (y) d y, \]
and it follows that $\mathcal{Q}_{\delta} f = Q_{\delta} \ast f$.

Using the basis $(\sigma_{k, \alpha})_{k \in \mathbb{Z}^d \setminus \{ 0 \},
\alpha \in \{ 1, \ldots, d - 1 \}}$ introduced in Section~\ref{s:notation},
\[ (\mathcal{Q}^{1 / 2}_{\delta} \sigma_{k, \alpha}) (x) = (K_{\delta} \ast
   \sigma_{k, \alpha}) (x) = (2 \pi)^{- d / 2} a_{k, \alpha}
   \int_{\mathbb{T}^d} K_{\delta} (y) e^{i k \cdummy (x - y)} d y \]
and since $K_{\delta}$ was obtained by first rescaling and then periodizing, a
change of variables yields
\begin{equation}
  = (2 \pi)^{- d / 2} \delta^{d / 2} \int_{\mathbb{R}^d} K (y) e^{- i (\delta
  k) \cdummy y} d y a_{k, \alpha} e^{i k \cdummy x} = \delta^{d / 2}
  (\mathcal{F}_{\mathbb{R}^d} K) (\delta k) \sigma_{k, \alpha} (x),
  \label{eq:eigen}
\end{equation}
where $\mathcal{F}_{\mathbb{R}^d}$ denotes the Fourier transform on
$\mathbb{R}^d$. In particular, $\mathcal{Q}^{1 / 2}_{\delta}$ acts diagonally
on this basis with eigenvalues $\delta^{d / 2} (\mathcal{F}_{\mathbb{R}^d} K)
(\delta k)$.

Let $V$ be a separable Banach space and $f : E \rightarrow V$ be a mean-zero
observable, i.e., $\int f d \nu^{1} = 0$. The Poisson equation $-\mathcal{M}
\varphi = f$ is solved by $\varphi = (-\mathcal{M})^{- 1} f = \int_0^{\infty}
P_t f d t$, where $P_t \assign e^{t\mathcal{M}}$ is the Markov semigroup. The
analogous statement holds for $\mathcal{M}^{\delta}$ with semigroup
$P^{\delta}_t = e^{t\mathcal{M}^{\delta}}$.

\subsection{Ornstein--Uhlenbeck process}\label{s:OU}

For the remainder of the paper, the processes $m$ and $m^{\delta}$ are
stationary solutions to
\begin{equation}
  d m = - m d t +\mathcal{Q}^{1 / 2} d W, \qquad d m^{\delta} = - m^{\delta} d
  t +\mathcal{Q}^{1 / 2}_{\delta} d W^{\delta}, \label{eq:m}
\end{equation}
for some two-sided cylindrical Wiener processes $W$ and $W^{\delta}$ on $H$, given by
the formulas
\[ m_t = \int_{- \infty}^t e^{- (t - r)} \mathcal{Q}^{1 / 2} d W_{r}, \qquad
   m^{\delta}_t = \int_{- \infty}^t e^{- (t - r)} \mathcal{Q}_{\delta}^{1 / 2}
   d W_{r}^{\delta} . \]
We define the Cameron--Martin space $\mathcal{H}=\mathcal{Q}^{1 / 2} H$ and
let $(\sigma_k)_{k \in \mathbb{N}}$ be a complete orthonormal system in
$\mathcal{H}$. Then $Q (x - y) = \sum_{k \in \mathbb{N}} \sigma_k (x) \otimes
\sigma_k (y)$, where the series converges absolutely, uniformly on compact
sets. A standard computation gives
\[ \mathbb{E} [m_t (x) \otimes m_s (y)] = \frac{1}{2} e^{- | t - s |} Q (x -
   y) . \]
Recall that in Section \ref{s:m} we defined $Q = (K \ast K)
\tmop{Id}_{\mathbb{R}^d} = q \tmop{Id}_{\mathbb{R}^d}$. Since the components
of the Ornstein--Uhlenbeck process $m$ are independent, the right hand side of
\eqref{eq:Q1} reduces to
\[ \int_0^{\infty} \mathbb{E} [m^i_0 (x) m_t^i (y) + m^i_t (x) m^i_0 (y)] d t
   = \int_0^{\infty} \frac{1}{2} e^{- t} q (x - y) + \frac{1}{2} e^{- t} q (x
   - y) d t = q (x - y), \]
and $m$ satisfies the assumptions of Section~\ref{s:m}.

We proceed similarly with the rescaled Ornstein--Uhlenbeck process
$m^{\delta}$. Here, we denote by $(\sigma_{\delta, k})_{k \in \mathbb{N}}$ a
complete orthonormal system in the Cameron--Martin space $\mathcal{H}_{\delta}
\assign \mathcal{Q}^{1 / 2}_{\delta} H$. So that again
\[ \sum_{k \in \mathbb{N}} \sigma_{\delta, k} (x) \otimes \sigma_{\delta, k}
   (y) = Q_{\delta} (x - y) = Q (\delta^{- 1} (x - y)) \]
and
\[ \mathbb{E} [m^{\delta}_t (x) \otimes m^{\delta}_s (y)] = \frac{1}{2} e^{- |
   s - t |} Q (\delta^{- 1} (x - y)) . \]
The invariant measure of the Ornstein--Uhlenbeck process $m$ is $\nu^{1}
=\mathcal{N} \left( 0, \frac{1}{2} \mathcal{Q} \right)$, and that of the
rescaled process $m^{\delta}$ is $\nu^{\delta} = \mathcal{N} \left( 0, \frac{1}{2}
\mathcal{Q}_{\delta} \right)$. This Gaussian structure underlies the explicit
Wiener chaos analysis of $(-\mathcal{M}^{\delta})^{- 1}$ carried out in
Section~\ref{s:chaos}.

Since $\mathcal{Q}^{1 / 2}$ acts diagonally on the Fourier basis $(\sigma_{k,
\alpha})_{k \in \mathbb{Z}^d_0, \alpha \in \{ 1, \ldots, d - 1 \}}$ with
eigenvalues $(\mathcal{F}_{\mathbb{R}^d} K) (k)$ by \eqref{eq:eigen}, the
process $m$ admits the Fourier representation in the sense that in law
\begin{equation}
  m \sim \frac{1}{\sqrt{2}} \sum_{k \in \mathbb{Z}^d_0, \alpha \in \{ 1,
  \ldots, d - 1 \}} (\mathcal{F}_{\mathbb{R}^d} K) (k) \sigma_{k, \alpha}
  \xi_{k, \alpha}, \label{eq:fourier}
\end{equation}
where $(\xi_{k, \alpha})_{k, \alpha}$ are mutually independent $\mathcal{N}
(0, 1)$-distributed random variables with $\xi_{- k, \alpha} =
\overline{\xi_{k, \alpha}}$. The same diagonalization gives the time-dependent
process as
\begin{equation}
  m_t \sim \frac{1}{\sqrt{2}} \sum_{k \in \mathbb{Z}^d_0, \alpha \in \{ 1,
  \ldots, d - 1 \}} (\mathcal{F}_{\mathbb{R}^d} K) (k) \sigma_{k, \alpha}
  \zeta_{k, \alpha} (t), \label{eq:9999}
\end{equation}
where $(\zeta_{k, \alpha})_{k, \alpha}$ are independent stationary
Ornstein--Uhlenbeck processes with $\zeta_{- k, \alpha} = \overline{\zeta_{k,
\alpha}}$ and $\mathbb{E} [\zeta_{k, \alpha} (t) \zeta_{k, \alpha} (s)] = e^{-
| t - s |}$.

The following estimates are standard consequences of the Borell--TIS inequality
\cite{MR2319516}, Chapter 2; see Appendix~\ref{s:a1}.

\begin{proposition}
  \label{c:sup}Let $m^{\varepsilon, \delta}$ be the rescaled
  Ornstein--Uhlenbeck process with $m^{\varepsilon, \delta}_t =
  m^{\delta}_{\varepsilon^{- 2} t}$. For all $s \in \mathbb{N}_0$ and $q \in
  [1, \infty)$ and every $t \in [0, T]$,
  \begin{equation}
    (\mathbb{E} [\| \nabla^s m^{\varepsilon, \delta}_t \|_{L^{\infty}}^q])^{1
    / q} \lesssim \delta^{- s}, \label{eq:m13}
  \end{equation}
  \begin{equation}
    (\mathbb{E} [\sup_{t \in [0, T]} \| \nabla^s m^{\varepsilon, \delta}_t
    \|_{L^{\infty}}^q])^{1 / q} \lesssim \delta^{- s} (\log \varepsilon^{-
    1})^{1 / 2} . \label{eq:supm}
  \end{equation}
\end{proposition}

\begin{proof}
  Since $m^{\varepsilon, \delta}_t (x) = m_{\varepsilon^{- 2} t}^{\delta} (x)
  = m (\varepsilon^{- 2} t, \delta^{- 1} x)$, Sobolev embedding implies for
  some $p > d$
  \[ \| \nabla^s m^{\varepsilon, \delta}_t \|_{L^{\infty}} = \delta^{- s} \|
     (\nabla^s m_{\varepsilon^{- 2} t}) (\delta^{- 1} \cdummy) \|_{L^{\infty}}
     = \delta^{- s} \| \nabla^s m_{\varepsilon^{- 2} t} \|_{L^{\infty}}
     \lesssim \delta^{- s} \| (- \Delta)^{(s + 1) / 2} m_{\varepsilon^{- 2} t}
     \|_{L^p} . \]
  Since $\mathcal{F}_{\mathbb{R}^d} K$ is a Schwartz function, the pointwise
  variance satisfies $\mathbb{E} [| (- \Delta)^{(s + 1) / 2} m_t (x) |^2]
  \lesssim 1$ uniformly in $t, x$, and \eqref{eq:m13} follows by
  hypercontractivity.
  
  The estimate \eqref{eq:supm} follows from Proposition~\ref{p:cont} applied
  to the centered stationary Gaussian field
  \[ X_{t, x} \assign \nabla^s m_t (x), \qquad (t, x) \in [0, \varepsilon^{-
     2} T] \times \mathbb{T}^d, \]
  once we verify its assumptions with $\sigma \lesssim 1$ and $\kappa
  \lesssim 1$, uniformly in $\varepsilon$. The pointwise variance bound
  $\mathbb{E} [X_{t, x}^2] \lesssim 1$ was established above. To verify
  \eqref{eq:block}, note first that by stationarity of $m$ in time, all
  blocks have the same law and it suffices to consider $j = 0$. Since
  $\mathcal{F}_{\mathbb{R}^d} K$ is a Schwartz function, the spectral
  representation \eqref{eq:9999} yields, for any fixed $\beta\in (0,1)$,
  \[ \mathbb{E} [| X_{t, x} - X_{t, y} |^2] \lesssim | x - y |^{2 \beta}
     \sum_{k \in \mathbb{Z}^d_0} | (\mathcal{F}_{\mathbb{R}^d} K) (k) |^2 | k
     |^{2 s + 2 \beta} \lesssim | x - y |^{2 \beta}, \]
  \[ \mathbb{E} [| X_{t, x} - X_{s, x} |^2] \lesssim | t - s | \sum_{k \in
     \mathbb{Z}^d_0} | (\mathcal{F}_{\mathbb{R}^d} K) (k) |^2 | k |^{2 s}
     \lesssim | t - s |, \]
  using $| \sigma_{k, \alpha} (x) - \sigma_{k, \alpha} (y) | \lesssim | k
  |^{\beta} | x - y |^{\beta}$ and $\mathbb{E} [| \zeta_{k, \alpha} (t) -
  \zeta_{k, \alpha} (s) |^2] \leqslant 2 \min (| t - s |, 1)$ respectively.
  By the triangle inequality and hypercontractivity (as $X_{t, x} - X_{s, y}$
  lies in the first Wiener chaos), for every $n \in \mathbb{N}$,
  \[ \mathbb{E} [| X_{t, x} - X_{s, y} |^{2 n}] \lesssim_n (| t - s | + | x -
     y |^{2 \beta})^n . \]
  Choosing $n$ sufficiently large, the Kolmogorov continuity theorem applied
  on the fixed parameter set $[0, 1] \times \mathbb{T}^d$ provides a
  continuous modification together with the bound
  \[ \mathbb{E} \Big[ \sup_{(t, x) \in [0, 1] \times \mathbb{T}^d} | X_{t, x}
     | \Big] \lesssim \mathbb{E} [| X_{0, x_0} |] +\mathbb{E} \Big[
     \sup_{(t, x) \neq (s, y)} \frac{| X_{t, x} - X_{s, y} |}{(| t - s | + |
     x - y |^{2 \beta})^{\gamma}} \Big] \lesssim 1 \]
 for some $x_{0}\in \mathbb{T}^{d}$ and a suitable $\gamma > 0$, with a constant depending only on $d$, $\beta$
  and $s$, but not on $\varepsilon$ or $T$. This verifies \eqref{eq:block}
  with $\kappa \lesssim 1$. Consequently, Proposition~\ref{p:cont} yields
  \[ \Big( \mathbb{E} \Big[ \sup_{(t, x) \in [0, \varepsilon^{- 2} T] \times
     \mathbb{T}^d} | X_{t, x} |^q \Big] \Big)^{1 / q} \lesssim_q \big( 1 +
     \log (\varepsilon^{- 2} T) \big)^{1 / 2} \lesssim_T (1 + \log
     \varepsilon^{- 1})^{1 / 2}\lesssim_T (\log
     \varepsilon^{- 1})^{1 / 2}, \]
     using $\varepsilon\leqslant 1/2$.
Hence \eqref{eq:supm} follows in view of $\| \nabla^s m^{\varepsilon,
  \delta}_t \|_{L^{\infty}} = \delta^{- s} \| \nabla^s m_{\varepsilon^{- 2}
  t} \|_{L^{\infty}}$ established above.
\end{proof}

\section{Second order expansion}\label{s:2}

To establish the law of large numbers, Theorem~\ref{thm:1}, the full machinery
of Section~\ref{s:high} is not required. We instead give a direct,
self-contained proof. Let $m$ be the Ornstein--Uhlenbeck process defined in
Section~\ref{s:OU} and set $m^{\varepsilon, \delta} (t, x) \assign m
(\varepsilon^{- 2} t, \delta^{- 1} x)$. The Markov generator of the pair
$(u^{\varepsilon, \delta}, m^{\varepsilon, \delta})$ is given by
\[ \mathcal{L}^{\varepsilon, \delta} \varphi (u, n) = \varepsilon^{- 2}
   \mathcal{M}^{\delta} \varphi (u, n) + \varepsilon^{- 1} \langle \mathbb{P}
   (n \cdummy \nabla u), D_u \varphi \rangle + \langle \Delta u, D_u \varphi
   \rangle - \langle \mathbb{P} (u \cdummy \nabla u), D_u \varphi \rangle, \]
where
\begin{equation}
  \varepsilon^{- 2} \mathcal{M}^{\delta} \varphi \assign - \varepsilon^{- 2}
  \langle n, D_n \varphi \rangle + \frac{\varepsilon^{- 2}}{2} \tmop{Tr}
  (\mathcal{Q}_{\delta} D^2_n \varphi) . \label{eq:Mdelta}
\end{equation}
Our goal is to identify the limit dynamics as $\varepsilon, \delta \rightarrow
0$. To this end, fix a smooth divergence-free test function $h$ and set
$\varphi (u) \assign \langle u, h \rangle$. We seek two correctors
$\varepsilon \varphi_1 (u, n)$ and $\varepsilon^2 \varphi_2 (u, n)$ designed
to cancel the singular terms in the evolution of $\varphi (u^{\varepsilon,
\delta})$. Expanding $\mathcal{L}^{\varepsilon, \delta} (\varphi + \varepsilon
\varphi_1 + \varepsilon^2 \varphi_2)$, one obtains
\[ \mathcal{L}^{\varepsilon, \delta} (\varphi + \varepsilon \varphi_1 +
   \varepsilon^2 \varphi_2) (u, n) = \varepsilon^{- 2} \mathcal{M}^{\delta}
   \varphi (u) \]
\[ + \varepsilon^{- 1} \langle \mathbb{P} (n \cdummy \nabla u), D_u \varphi
   \rangle + \varepsilon^{- 1} \mathcal{M}^{\delta} \varphi_1 (u, n) \]
\[ + \langle \Delta u, D_u \varphi \rangle - \langle \mathbb{P} (u \cdummy
   \nabla u), D_u \varphi \rangle + \langle \mathbb{P} (n \cdummy \nabla u),
   D_u \varphi_1 \rangle +\mathcal{M}^{\delta} \varphi_2 (u, n) \]
\begin{equation}
  + \varepsilon \langle \Delta u, D_u \varphi_1 \rangle - \varepsilon \langle
  \mathbb{P} (u \cdummy \nabla u), D_u \varphi_1 \rangle + \varepsilon \langle
  \mathbb{P}(n \cdummy \nabla u), D_u \varphi_2 \rangle \label{eq:kl1}
\end{equation}
\begin{equation}
  + \varepsilon^2 \langle \Delta u, D_u \varphi_2 \rangle - \varepsilon^2
  \langle \mathbb{P} (u \cdummy \nabla u), D_u \varphi_2 \rangle .
  \label{eq:kl2}
\end{equation}
The first term on the right hand side vanishes since $\varphi$ is independent
of $n$. We choose $\varphi_1$ to cancel the $\varepsilon^{- 1}$ terms, that
is, we require
\[ -\mathcal{M}^{\delta} \varphi_1 (u, n) = \langle \mathbb{P} (n \cdummy
   \nabla u), D_u \varphi \rangle = \langle n \cdummy \nabla u, h \rangle . \]
Since $m^{\varepsilon, \delta}$ satisfies the SPDE
\begin{equation}
  d m^{\varepsilon, \delta} = - \varepsilon^{- 2} m^{\varepsilon, \delta} d t
  + \varepsilon^{- 1} \mathcal{Q}^{1 / 2}_{\delta} d W, \label{eq:med}
\end{equation}
it is natural to take $\varphi_1 (u, n) = \langle n \cdummy \nabla u, h
\rangle$. Indeed, applying It{\^o}'s formula to $\varepsilon \varphi_1
(u^{\varepsilon, \delta}, m^{\varepsilon, \delta})$, the drift term in
\eqref{eq:med} produces the term $- \varepsilon^{- 1} \langle m^{\varepsilon,
\delta} \cdummy \nabla u^{\varepsilon, \delta}, h \rangle$, as required.

Turning to the $\varepsilon^0$ line, we choose $\varphi_2$ to cancel the
$n$-dependent part of $\langle \mathbb{P} (n \cdummy \nabla u), D_u \varphi_1
\rangle +\mathcal{M}^{\delta} \varphi_2 (u, n)$. Since $\langle \mathbb{P} (n
\cdummy \nabla u), D_u \varphi_1 \rangle = \langle n \cdummy \nabla \mathbb{P}
(n \cdummy \nabla u), h \rangle$, taking $\varphi_2 (u, n) = \frac{1}{2}
\langle n \cdummy \nabla \mathbb{P} (n \cdummy \nabla u), h \rangle$ and
applying It{\^o}'s formula to $\varepsilon^2 \varphi_2 (u^{\varepsilon,
\delta}, m^{\varepsilon, \delta})$, the drift in \eqref{eq:med} produces $-
\langle m^{\varepsilon, \delta} \cdummy \nabla \mathbb{P} (m^{\varepsilon,
\delta} \cdummy \nabla u^{\varepsilon, \delta}), h \rangle$, cancelling the
contribution of $\langle \mathbb{P} (n \cdummy \nabla u), D_u \varphi_1
\rangle$. The remaining term, arising from the second summand in
\eqref{eq:Mdelta} acting on $\varphi_2$, that is,
\begin{equation}
  \langle S_{\delta} u^{\varepsilon, \delta}, h \rangle \assign \frac{1}{2}
  \tmop{Tr} (\mathcal{Q}_{\delta} D^2_n \varphi_2 (u^{\varepsilon, \delta},
  m^{\varepsilon, \delta})), \label{eq:Sd}
\end{equation}
is responsible for the enhanced diffusion and is analyzed in
Section~\ref{s:quali}. As $\varphi_2$ is quadratic in $n$, $S_{\delta}$ is
indeed an operator acting on $u$.

\section{Qualitative enhanced diffusion}
\label{s:quali}

Using $\overline{\sigma_{k, \alpha}} = \sigma_{- k, \alpha}$, which holds
since $a_{k, \alpha} = a_{- k, \alpha} \in \mathbb{R}^d$, together with
\eqref{eq:eigen}, we compute
\[ \langle S_{\delta} v, h \rangle = \frac{1}{2} \tmop{Tr} (\mathcal{Q}^{1 /
   2}_{\delta} D^2_n \varphi_2 (v, n) \mathcal{Q}^{1 / 2}_{\delta}) =
   \frac{1}{2} \sum_{k, \alpha} \langle D^2_n \varphi_2 (v, n) \mathcal{Q}^{1
   / 2}_{\delta} \sigma_{k, \alpha}, \mathcal{Q}^{1 / 2}_{\delta} \sigma_{- k,
   \alpha} \rangle \]
\[ = \frac{\delta^d}{4} \sum_{k, \alpha} | (\mathcal{F}_{\mathbb{R}^d} K)
   (\delta k) |^2 (\langle \sigma_{k, \alpha} \cdummy \nabla \mathbb{P}
   (\sigma_{- k, \alpha} \cdummy \nabla v), h \rangle + \langle \sigma_{- k,
   \alpha} \cdummy \nabla \mathbb{P} (\sigma_{k, \alpha} \cdummy \nabla v), h
   \rangle) \]
\[ = \frac{\delta^d}{2} \sum_{k, \alpha} | (\mathcal{F}_{\mathbb{R}^d} K)
   (\delta k) |^2 \langle \mathbb{P} (\sigma_{k, \alpha} \cdummy \nabla
   \mathbb{P} (\sigma_{- k, \alpha} \cdummy \nabla v)), h \rangle, \]
by symmetry in $k \mapsto - k$. In view of the Fourier representation of the
stationary Ornstein--Uhlenbeck process $m^{\delta}$ from Section~\ref{s:OU},
this translates to $S_{\delta} v =\mathbb{P}\mathbb{E}_{\nu^{\delta}} [n
\cdummy \nabla \mathbb{P} (n \cdummy \nabla v)]$, where $\nu^{\delta}$ is the
invariant measure of $m^{\delta}$. From this representation, one immediately
observes that $S_{\delta}$ is translation-invariant, self-adjoint,
non-positive, and of second order. Formally, $S_{\delta}$ may be viewed as a
divergence-form operator with a nonlocal coefficient involving the convolution
kernel $Q_{\delta} (x) \otimes \mathbb{P} (x)$, where $\mathbb{P} (x)$ denotes
the Leray kernel. This kernel exhibits Calder{\'o}n--Zygmund type
singularities that are uniform in $\delta$, in the sense that the spatial
localization by $Q_{\delta}$ does not alter the underlying singular structure.
By construction, $S_{\delta}$ commutes with the Leray projection.

\begin{lemma}
  \label{l:7}$S_{\delta}$ is a Fourier multiplier operator, $\mathcal{F}
  [S_{\delta} v] (\ell) = s_{\delta} (\ell) \hat{v} (\ell)$, $\ell \in
  \mathbb{Z}^d$, with symbol
  \[ s_{\delta} (\ell) = - \frac{(2 \pi)^{- d}}{2} \left( \tmop{Id} -
     \frac{\ell \otimes \ell}{| \ell |^2} \right) \delta^d \sum_{k \in
     \mathbb{Z}_0^d, k\neq\ell, \alpha = 1, \ldots, d - 1} | (\mathcal{F}_{\mathbb{R}^d}
     K) (\delta k) |^2 (a_{k, \alpha} \cdummy \ell)^2 \left( \tmop{Id} -
     \frac{(\ell - k) \otimes (\ell - k)}{| \ell - k |^2} \right) . \]
\end{lemma}

\begin{proof}
  Since $\sigma_{- k, \alpha} = (2 \pi)^{- d / 2} a_{k, \alpha} e^{- i k
  \cdummy}$, multiplication by $\sigma_{- k, \alpha}$ shifts the Fourier
  frequency by $k$:
  \[ \mathcal{F} [\sigma_{- k, \alpha} \cdummy f] (\ell) = (2 \pi)^{- d / 2}
     a_{k, \alpha} \cdummy \hat{f} (\ell + k) . \]
  Since $a_{k, \alpha} \cdummy k = 0$, applying this to $\nabla v$ gives
  \[ \mathcal{F} [\sigma_{- k, \alpha} \cdummy \nabla v] (\ell) = (2 \pi)^{- d
     / 2} a_{k, \alpha} \cdummy \mathcal{F} [\nabla v] (\ell + k) = (2 \pi)^{-
     d / 2} i (a_{k, \alpha} \cdummy \ell) \hat{v} (\ell + k) . \]
  Applying the Leray projection $\mathcal{F} [\mathbb{P}f] (\ell) = \left(
  \tmop{Id} - \frac{\ell \otimes \ell}{| \ell |^2} \right) \hat{f} (\ell)$ and
  then $\nabla$:
  \[ \mathcal{F} [\mathbb{P} (\sigma_{- k, \alpha} \cdummy \nabla v)] (\ell) =
     (2 \pi)^{- d / 2} \left( \tmop{Id} - \frac{\ell \otimes \ell}{| \ell |^2}
     \right) i (a_{k, \alpha} \cdummy \ell) \hat{v} (\ell + k), \]
  \[ \mathcal{F} [\nabla \mathbb{P} (\sigma_{- k, \alpha} \cdummy \nabla v)]
     (\ell) = - (2 \pi)^{- d / 2} \left( \tmop{Id} - \frac{\ell \otimes
     \ell}{| \ell |^2} \right) (a_{k, \alpha} \cdummy \ell) \ell \otimes
     \hat{v} (\ell + k) . \]
  Now, $\sigma_{k, \alpha}$ shifts the frequency by $- k$, so evaluating the
  above at $\ell - k$ for $k\neq \ell$  (the term $k=\ell$ vanishes due to $a_{k,\alpha}\cdot k=0$):
  \[ \mathcal{F} [\sigma_{k, \alpha} \cdummy \nabla \mathbb{P} (\sigma_{- k,
     \alpha} \cdummy \nabla v)] (\ell) = - (2 \pi)^{- d} (a_{k, \alpha}
     \cdummy \ell)^2 \left( \tmop{Id} - \frac{(\ell - k) \otimes (\ell - k)}{|
     \ell - k |^2} \right) \hat{v} (\ell) . \]
  A final application of $\mathbb{P}$ and summation over $k, \alpha$ with
  weights $\frac{\delta^d}{2} | (\mathcal{F}_{\mathbb{R}^d} K) (\delta k) |^2$
  yields the stated formula.
\end{proof}

Before we proceed, note that since $\{ a_{k, 1}, \ldots, a_{k, d - 1}, k / | k
| \}$ is an orthonormal basis of $\mathbb{R}^d$, it holds for every $\ell \in
\mathbb{R}^d$
\[ \ell = \sum_{\alpha = 1, \ldots, d - 1} (a_{k, \alpha} \cdummy \ell) a_{k,
   \alpha} + \left( \frac{k}{| k |} \cdummy \ell \right) \frac{k}{| k |}, \]
hence
\[ | \ell |^2 = \sum_{\alpha = 1, \ldots, d - 1} (a_{k, \alpha} \cdummy
   \ell)^2 + \left( \frac{k}{| k |} \cdummy \ell \right)^2, \]
and finally
\begin{equation}
  \sum_{\alpha = 1, \ldots, d - 1} (a_{k, \alpha} \cdummy \ell)^2 = | \ell |^2
  - \left( \frac{k}{| k |} \cdummy \ell \right)^2 = | \ell |^2 \left( 1 -
  \frac{(k  \cdummy \ell)^2}{| k |^2 | \ell |^2} \right) = | \ell |^2
  \sin^2 (\angle_{k, \ell}), \label{eq:22}
\end{equation}
where $\angle_{k, \ell}$ denotes the angle between the vectors $k, \ell$.

The key step is to replace the Leray projector at $\ell - k$ by the one at
$k$, introducing the simplified symbol $\tilde{s}_{\delta} (\ell)$, which
enables the explicit computation of the limit in Lemma~\ref{lem:5}.

\begin{lemma}
  \label{lem:3}For every $\ell \in \mathbb{Z}^d$ it holds $\lim_{\delta
  \rightarrow 0} (s_{\delta} (\ell) - \tilde{s}_{\delta} (\ell)) = 0$, where
  \begin{equation}
    \tilde{s}_{\delta} (\ell) = - \frac{(2 \pi)^{- d}}{2} \left( \tmop{Id} -
    \frac{\ell \otimes \ell}{| \ell |^2} \right) | \ell |^2 \delta^d \sum_{k
    \in \mathbb{Z}_0^d} | (\mathcal{F}_{\mathbb{R}^d} K) (\delta k) |^2 \sin^2
    (\angle_{k, \ell}) \left( \tmop{Id} - \frac{k \otimes k}{| k |^2} \right)
    . \label{eq:stilde}
  \end{equation}
\end{lemma}

\begin{proof}
  The case $\ell = 0$ is immediate since $s_{\delta} (0) = \tilde{s}_{\delta}
  (0) = 0$. For $\ell \neq 0$, we use the identity \eqref{eq:22} to perform
  the $\alpha$ sum and reduce to controlling $s_{\delta} (\ell) -
  \tilde{s}_{\delta} (\ell)$ via the estimate
  \[ \left| \frac{(\ell - k) \otimes (\ell - k)}{| \ell - k |^2} - \frac{k
     \otimes k}{| k |^2} \right| \leqslant \left\{ \begin{array}{llll}
       2, &  & \tmop{if} & | k | \leqslant 1,\\
       4 | \ell | / | k |, &  & \tmop{if} & | k | > 1,
     \end{array} \right. \]
  from Lemma 5.5 in {\cite{MR4265023}}. Splitting the sum accordingly and
  using $\sin^2 (\angle_{k, \ell}) \leqslant 1$, we bound the difference
  $s_{\delta} (\ell) - \tilde{s}_{\delta} (\ell)$ by a fixed multiple of
  \[ \delta^d \sum_{k \in \mathbb{Z}_0^d,k\neq\ell} | (\mathcal{F}_{\mathbb{R}^d} K)
     (\delta k) |^2 \sin^2 (\angle_{k, \ell}) \left| \frac{(\ell - k) \otimes
     (\ell - k)}{| \ell - k |^2} - \frac{k \otimes k}{| k |^2} \right| \]
  \[ \lesssim \delta^d \sum_{k \in \mathbb{Z}_0^d, | k | \leqslant 1} |
     (\mathcal{F}_{\mathbb{R}^d} K) (\delta k) |^2 + | \ell | \delta^d \sum_{k
     \in \mathbb{Z}_0^d, | k | > 1} | (\mathcal{F}_{\mathbb{R}^d} K) (\delta
     k) |^2 \frac{1}{| k |} \]
  \[ \lesssim \delta^d \sup_k | (\mathcal{F}_{\mathbb{R}^d} K) (\delta k) |^2
     + | \ell | \delta^d \sum_{k \in \mathbb{Z}_0^d, | k | > 1} |
     (\mathcal{F}_{\mathbb{R}^d} K) (\delta k) |^2 \frac{| k |^{d - 1 +
     \kappa}}{| k |^{d + \kappa}} \]
  \[ \lesssim \delta^d + | \ell | \delta^{1 - \kappa} \sup_k \{ |
     (\mathcal{F}_{\mathbb{R}^d} K) (\delta k) |^2 | \delta k |^{d - 1 +
     \kappa} \} \lesssim \delta^d + | \ell | \delta^{1 - \kappa}, \]
  for any $\kappa \in (0, 1)$, where both suprema are finite since
  $\mathcal{F}_{\mathbb{R}^d} K$ is a Schwartz function. Thus, the right hand
  side vanishes as $\delta \rightarrow 0$, proving the claim.
\end{proof}

\begin{lemma}
  \label{lem:5}For every $\ell \in \mathbb{Z}^d$ it holds
  \[ \lim_{\delta \rightarrow 0} \tilde{s}_{\delta} (\ell) = - \nu \left(
     \tmop{Id} - \frac{\ell \otimes \ell}{| \ell |^2} \right) | \ell |^2,
     \qquad \tmop{where} \qquad \nu = \left\{ \begin{array}{llll}
       \frac{1}{16} \| K \|_{L^2}^2, &  &  & d = 2,\\
       \frac{1}{5} \| K \|_{L^2}^2, &  &  & d = 3.
     \end{array} \right. \]
\end{lemma}

\begin{proof}
  Recall from \eqref{eq:stilde} that
  \[ \tilde{s}_{\delta} (\ell) = - \frac{(2 \pi)^{- d}}{2} \left( \tmop{Id} -
     \frac{\ell \otimes \ell}{| \ell |^2} \right) | \ell |^2 \delta^d \sum_{k
     \in \mathbb{Z}_0^d} | (\mathcal{F}_{\mathbb{R}^d} K) (\delta k) |^2
     \sin^2 (\angle_{\delta k, \ell}) \left( \tmop{Id} - \frac{\delta k
     \otimes \delta k}{| \delta k |^2} \right) . \]
  By Lemma~\ref{l:riemann}, it holds
  \[ \delta^d \sum_{k \in \mathbb{Z}_0^d} | (\mathcal{F}_{\mathbb{R}^d} K)
     (\delta k) |^2 \sin^2 (\angle_{\delta k, \ell}) \left( \tmop{Id} -
     \frac{\delta k \otimes \delta k}{| \delta k |^2} \right) \]
  \begin{equation}
    \rightarrow \int_{\mathbb{R}^d} | (\mathcal{F}_{\mathbb{R}^d} K) (k) |^2
    \sin^2 (\angle_{k, \ell}) \left( \tmop{Id} - \frac{k \otimes k}{| k |^2}
    \right) d k \backassign J (\ell), \label{eq:J}
  \end{equation}
  since the function $k \mapsto | (\mathcal{F}_{\mathbb{R}^d} K) (k) |^2
  \sin^2 (\angle_{k, \ell}) \left( \mathrm{Id} - \frac{k \otimes k}{| k |^2} \right)$ is
  continuous away from zero, bounded near zero and rapidly decaying at
  infinity since $\mathcal{F}_{\mathbb{R}^d} K$ is Schwartz.
  
  Step 1: Rotation and diagonal structure. Since the integrand defining $J
  (\ell)$ depends on $\ell$ only through its direction $\ell / | \ell |$, so
  does $J (\ell)$. Let $R = R_{\ell / | \ell |}$ be the rotation matrix such
  that $R \ell / | \ell | = f_d$ where $(f_j)_{j = 1, \ldots, d}$ is the
  canonical orthonormal basis of $\mathbb{R}^d$. The matrix is given by $$R =
  (a_{\ell, 1}, \ldots, a_{\ell, d - 1}, \ell / | \ell |)^{\ast}$$ and
  satisfies $R a_{\ell, \alpha} = f_{\alpha}$ for $\alpha = 1, \ldots, d - 1$.
  By radial symmetry of $\mathcal{F}_{\mathbb{R}^d} K$,
  \[ R J (\ell) R^{\ast} = \int_{\mathbb{R}^d} | (\mathcal{F}_{\mathbb{R}^d}
     K) (R k) |^2 \sin^2 (\angle_{R k, f_d}) \left( \tmop{Id} - \frac{R k
     \otimes R k}{| R k |^2} \right) d k \]
  \[ = \int_{\mathbb{R}^d} | (\mathcal{F}_{\mathbb{R}^d} K) (k) |^2 \sin^2
     (\angle_{k, f_d}) \left( \tmop{Id} - \frac{k \otimes k}{| k |^2} \right)
     d k. \]
  The off-diagonal entries of $R J (\ell) R^{\ast}$ vanish since the integrand
  is odd under the reflection $k_i \mapsto - k_i$ for $i \neq j$. By the
  rotational symmetry of the integrand within the subspace $f_d^{\perp} =
  \tmop{span} \{ f_1, \ldots, f_{d - 1} \}$, the upper-left $(d - 1) \times (d
  - 1)$ block is a scalar multiple of $\tmop{Id}_{\mathbb{R}^{d - 1}}$, so $R
  J (\ell) R^{\ast} = \tmop{diag} (c, \ldots, c, \tilde{c})$, for some
  constants $c$, $\tilde{c}$ depending only on $K$ and $d$.
  
  Step 2: Explicit computation of $c$, case $d = 3$. In spherical coordinates
  $k_1 = r \sin \psi \cos \varphi$, $k_2 = r \sin \psi \sin \varphi$, $k_3 = r
  \cos \psi$, $\angle_{k, f_3} = \psi$ and $(\tmop{Id} - k \otimes k / | k
  |^2)_{11} = 1 - \sin^2 \psi \cos^2 \varphi$. Using radiality of $K$, i.e.
  $(\mathcal{F}_{\mathbb{R}^d} K) (k) =\varkappa (| k |)$, we obtain
  \[ c = (R J (\ell) R^{\ast})_{11} = \int_0^{\infty} d r\,\varkappa^2 (r)
     \int_0^{\pi} d \psi \int_0^{2 \pi} d \varphi \sin^2 \psi (1 - \sin^2 \psi
     \cos^2 \varphi) r^2 \sin \psi \]
  \[ = \int_0^{\infty} d r\,\varkappa^2 (r) r^2 \left( 2 \pi \int_0^{\pi} d
     \psi \sin^3 \psi - \pi \int_0^{\pi} d \psi \sin^5 \psi \right) =
     \int_0^{\infty} d r\,\varkappa^2 (r) r^2 \left( \frac{2 \pi \cdummy 4}{3}
     - \frac{\pi \cdummy 16}{15} \right) \]
  \[ = \frac{8 \pi}{5} \int_0^{\infty} d r\,\varkappa^2 (r) r^2 . \]
  Converting back to Euclidean coordinates
  \[ c = \frac{2}{5} \int_0^{\infty} d r\,\varkappa^2 (r) \int_0^{\pi} d \psi
     \int_0^{2 \pi} d \varphi r^2 \sin \psi = \frac{2}{5} \int_{\mathbb{R}^3}
     | (\mathcal{F}_{\mathbb{R}^3} K) (k) |^2 d k = \frac{2}{5} (2 \pi)^3 \| K
     \|_{L^2}^2 . \]
  Step 3: Explicit computation of $c$, case $d = 2$. In polar coordinates $k_1
  = r \cos \theta$, $k_2 = r \sin \theta$, we have $\angle_{k, f_2} = \pi / 2
  - \theta$, hence $\sin^2 (\angle_{k, f_2}) = \cos^2 (\theta)$, and
  $(\tmop{Id} - k \otimes k / | k |^2)_{11} = 1 - \cos^2 \theta = \sin^2
  \theta$. Using radiality of $K$, we obtain
  \[ c = (R J (\ell) R^{\ast})_{11} = \int_0^{\infty} d r\,\varkappa^2 (r) r
     \int_0^{2 \pi} d \theta \cos^2 \theta \sin^2 \theta = \frac{\pi}{4}
     \int_0^{\infty} d r\,\varkappa^2 (r) r, \]
  hence in Euclidean coordinates,
  \[ c = \frac{1}{8} \int_0^{\infty} d r\,\varkappa^2 (r) r \int_0^{2 \pi} d
     \theta = \frac{1}{8} \int_{\mathbb{R}^2} | (\mathcal{F}_{\mathbb{R}^2} K)
     (k) |^2 d k = \frac{1}{8} (2 \pi)^2 \| K \|_{L^2}^2 . \]
  Step 4: Conclusion. We have established that $R J (\ell) R^{\ast}$ is
  diagonal with $(R J (\ell) R^{\ast})_{j j} = 2 \nu (2 \pi)^d$ for $j = 1,
  \ldots, d - 1$ and some value in the last diagonal entry. Combined with the
  factor $\frac{(2 \pi)^{- d}}{2} \left( \mathrm{Id} - \frac{\ell \otimes \ell}{| \ell
  |^2} \right)$ from \eqref{eq:stilde}, we compute
  \[ \frac{(2 \pi)^{- d}}{2} \left( \mathrm{Id} - \frac{\ell \otimes \ell}{| \ell |^2}
     \right) J = \frac{(2 \pi)^{- d}}{2} R^{\ast} \left[ R \left( \mathrm{Id} -
     \frac{\ell \otimes \ell}{| \ell |^2} \right) R^{\ast} \right] [R J
     R^{\ast}] R \]
  \[ = R^{\ast} (\mathrm{Id} - f_d \otimes f_d) \tmop{diag} \left( \nu, \ldots, \nu,
     \frac{(2 \pi)^{- d}}{2} \tilde{c} \right) R = \nu R^{\ast} \tmop{diag}
     (1, \ldots, 1, 0) R, \]
  where in the last step we used that $\tmop{Id} - f_d \otimes f_d =
  \tmop{diag} (1, \ldots, 1, 0)$ kills the last diagonal entry, so the value
  of $\tilde{c}$ plays no role. The resulting operator acts on the orthonormal
  basis $\{ a_{\ell, 1}, \ldots, a_{\ell, d - 1}, \ell / | \ell | \}$ as
  \[ \nu R^{\ast} \tmop{diag} (1, \ldots, 1, 0) R \frac{\ell}{| \ell |} = \nu
     R^{\ast} \tmop{diag} (1, \ldots, 1, 0) f_d = 0, \]
  \[ \nu R^{\ast} \tmop{diag} (1, \ldots, 1, 0) R a_{\ell, \alpha} = \nu
     R^{\ast} f_{\alpha} = \nu a_{\ell, \alpha}, \qquad \alpha = 1, \ldots, d
     - 1, \]
  which coincides with $\nu \left( \tmop{Id} - \frac{\ell \otimes \ell}{| \ell
  |^2} \right)$ on the orthonormal basis, hence the two operators are equal.
  Therefore, $\lim_{\delta \rightarrow 0} \tilde{s}_{\delta} (\ell) = - \nu
  \left( \tmop{Id} - \frac{\ell \otimes \ell}{| \ell |^2} \right) | \ell |^2$,
  completing the proof.
\end{proof}

\begin{proposition}
  \label{p:11re}Let $K$ be smooth, compactly supported and radially symmetric,
  and let $\nu$ be as in Lemma~\ref{lem:5}. Then for every smooth
  divergence-free $v$ and every $\gamma \in \mathbb{R}$, $S_{\delta} v
  \rightarrow \nu \mathbb{P} \Delta v$ in $H^{\gamma}$ as $\delta \rightarrow
  0$. Moreover, the symbol $s_{\delta} (\ell)$ satisfies the uniform bound $|
  s_{\delta} (\ell) | \lesssim (1 + \| K \|_{L^2}^2) | \ell |^2$ for all
  $\delta \in (0, \delta_0]$ and some $\delta_0 \in (0, 1)$.
\end{proposition}

\begin{proof}
  By Lemma~\ref{l:7}, Lemma~\ref{lem:3} and Lemma~\ref{lem:5}, $s_{\delta}
  (\ell) \rightarrow - \nu \left( \tmop{Id} - \frac{\ell \otimes \ell}{| \ell
  |^2} \right) | \ell |^2$ pointwise for every $\ell \in \mathbb{Z}^d$. For
  the uniform bound, using $| a_{k, \alpha} \cdummy \ell | \leqslant | \ell |$
  and $\left| \tmop{Id} - \frac{(\ell - k) \otimes (\ell - k)}{| \ell - k |^2}
  \right| \leqslant 1$,
  \[ | s_{\delta} (\ell) | \lesssim | \ell |^2 \delta^d \sum_{k \in
     \mathbb{Z}^d_0} | (\mathcal{F}_{\mathbb{R}^d} K) (\delta k) |^2 \lesssim
     (1 + \| K \|_{L^2}^2) | \ell |^2 \]
  for $\delta \in (0, \delta_0]$, where the last step uses
  Lemma~\ref{l:riemann} to bound the Riemann sum uniformly. For the
  convergence in $H^{\gamma}$, write
  \[ \| S_{\delta} v - \nu \mathbb{P} \Delta v \|_{H^{\gamma}}^2 \leqslant
     \sum_{\ell \in \mathbb{Z}^d, | \ell | \leqslant M} (1 + | \ell
     |^2)^{\gamma} \left| \left( s_{\delta} (\ell) + \nu \left( \tmop{Id} -
     \frac{\ell \otimes \ell}{| \ell |^2} \right) | \ell |^2 \right) \right|^2
     | \hat{v} (\ell) |^2 \]
  \[ + \sum_{\ell \in \mathbb{Z}^d, | \ell | > M} (1 + | \ell |^2)^{\gamma}
     \left| \left( s_{\delta} (\ell) + \nu \left( \tmop{Id} - \frac{\ell
     \otimes \ell}{| \ell |^2} \right) | \ell |^2 \right) \right|^2 | \hat{v}
     (\ell) |^2 . \]
  The first term vanishes for every fixed $M > 0$ as $\delta \rightarrow 0$.
  The second term is bounded, for any $\kappa > 0$, by
  \[ \lesssim (1 + \| K \|^2_{L^2})^2 \sum_{\ell \in \mathbb{Z}^d, | \ell | >
     M} (1 + | \ell |^2)^{\gamma} | \ell |^4 | \hat{v} (\ell) |^2 \lesssim
     \frac{1}{M^{2 \kappa}} (1 + \| K \|^2_{L^2})^2 \| v \|_{H^{\gamma + 2 +
     \kappa}}^2 . \]
  Choosing first $M$ large and then $\delta$ small gives the result.
\end{proof}

\section{Law of large numbers}\label{s:LLN}

\subsection{Error terms}\label{s:err}

We estimate the error terms in \eqref{eq:kl1} and \eqref{eq:kl2} as well as
the boundary terms at time 0 and $t$ appearing after the application of
It{\^o}'s formula to $\varphi (u^{\varepsilon, \delta}) + \varepsilon
\varphi_1 (u^{\varepsilon, \delta}, m^{\varepsilon, \delta}) + \varepsilon^2
\varphi_2 (u^{\varepsilon, \delta}, m^{\varepsilon, \delta})$, where $\varphi
(u) = \langle u, h \rangle$, $\varphi_1 (u, n) = \langle n \cdummy \nabla u, h
\rangle$, and $\varphi_2 (u, n) = \frac{1}{2} \langle n \cdummy \nabla
\mathbb{P} (n \cdummy \nabla u), h \rangle$, for a smooth divergence-free test
function $h$.

The boundary terms are controlled using the uniform bound on $m^{\varepsilon,
\delta}$ from \eqref{eq:m13} and the energy estimate \eqref{eq:energy}:
\[ \mathbb{E} [| \varepsilon \varphi_1 (u^{\varepsilon, \delta}_t,
   m^{\varepsilon, \delta}_t) |] \lesssim \varepsilon \mathbb{E} [\|
   m^{\varepsilon, \delta}_t \|_{L^2} \| u^{\varepsilon, \delta}
   \|_{L^{\infty} L^2}] \lesssim \varepsilon, \]
\[ \mathbb{E} [| \varepsilon^2 \varphi_2 (u^{\varepsilon, \delta}_t,
   m^{\varepsilon, \delta}_t) |] \lesssim \varepsilon^2 \mathbb{E} [\|
   m^{\varepsilon, \delta}_t \|_{L^4} (\| \nabla m^{\varepsilon, \delta}_t
   \|_{L^4} + \| m^{\varepsilon, \delta}_t \|_{L^4}) \| u^{\varepsilon,
   \delta} \|_{L^{\infty} L^2}] \lesssim \varepsilon^2 \delta^{- 1} . \]
For the remaining error terms from \eqref{eq:kl1} and \eqref{eq:kl2}, the
estimates follow a common structure: we bound $\nabla u^{\varepsilon, \delta}$
in $L^2 L^{2}$ via the energy estimate, and control $m^{\varepsilon, \delta}$
and its derivatives in appropriate $L^p$-spaces using \eqref{eq:m13}. A key
observation governing the derivative counting is that among all derivatives
appearing in a given term, precisely one is always tied to $u^{\varepsilon,
\delta}$, while each remaining derivative either hits $m^{\varepsilon,
\delta}$, contributing a factor of $\delta^{- 1}$, or is transferred to the
test function $h$ via integration by parts at no $\delta$ cost. Specifically,
we immediately obtain
\[ \mathbb{E} \left[ \left| \int_0^t \varepsilon \langle \Delta
   u^{\varepsilon, \delta}, D_u \varphi_1 (u^{\varepsilon, \delta},
   m^{\varepsilon, \delta}) \rangle d s \right| \right] = \varepsilon
   \mathbb{E} \left[ \left| \int_0^t \langle m^{\varepsilon, \delta} \cdummy
   \nabla (\Delta u^{\varepsilon, \delta}), h \rangle d s \right| \right] \]
\[ = \varepsilon \mathbb{E} \left[ \left| \int_0^t \langle \nabla
   u^{\varepsilon, \delta}, \nabla (m^{\varepsilon, \delta} \cdummy \nabla h)
   \rangle d s \right| \right] \lesssim \varepsilon \mathbb{E} \left[ \int_0^t
   \| \nabla u^{\varepsilon, \delta} \|_{L^2} (\| m^{\varepsilon, \delta}
   \|_{L^2} + \| \nabla m^{\varepsilon, \delta} \|_{L^2}) d s \right] \]
\[ \lesssim \varepsilon \left( \mathbb{E} \left[ \int_0^t \| \nabla
   u^{\varepsilon, \delta} \|_{L^2}^2 d s \right] \right)^{1 / 2} (\mathbb{E}
   [\| m^{\varepsilon, \delta} \|^2_{L^2} + \| \nabla m^{\varepsilon, \delta}
   \|^2_{L^2}])^{1 / 2} \lesssim \varepsilon \delta^{- 1}, \]
and similarly
\[ \mathbb{E} \left[ \left| \varepsilon \int_0^t \langle \mathbb{P}(m^{\varepsilon,
   \delta} \cdummy \nabla u^{\varepsilon, \delta}), D_u \varphi_2
   (u^{\varepsilon, \delta}, m^{\varepsilon, \delta}) \rangle d s \right|
   \right] \]
\[ =\mathbb{E} \left[ \left| \frac{\varepsilon}{2} \int_0^t \langle
   m^{\varepsilon, \delta} \cdummy \nabla \mathbb{P} (m^{\varepsilon, \delta}
   \cdummy \nabla \mathbb{P} (m^{\varepsilon, \delta} \cdummy \nabla
   u^{\varepsilon, \delta})), h \rangle d s \right| \right] \lesssim
   \varepsilon \delta^{- 1}, \]
\[ \mathbb{E} \left[ \left| \varepsilon^2 \int_0^t \langle \Delta
   u^{\varepsilon, \delta}, D_u \varphi_2 (u^{\varepsilon, \delta},
   m^{\varepsilon, \delta}) \rangle d s \right| \right] =\mathbb{E} \left[
   \left| \frac{\varepsilon^2}{2} \int_0^t \langle m^{\varepsilon, \delta}
   \cdummy \nabla \mathbb{P} (m^{\varepsilon, \delta} \cdummy \nabla (\Delta
   u^{\varepsilon, \delta})), h \rangle d s \right| \right] \lesssim
   \varepsilon^2 \delta^{- 2} . \]
For the convective term of the first corrector, we obtain
\[ \mathbb{E} \left[ \left| \varepsilon \int_0^t \langle \mathbb{P}
   (u^{\varepsilon, \delta} \cdummy \nabla u^{\varepsilon, \delta}), D_u
   \varphi_1 (u^{\varepsilon, \delta}, m^{\varepsilon, \delta}) \rangle d s
   \right| \right] =\mathbb{E} \left[ \left| \varepsilon \int_0^t \langle
   m^{\varepsilon, \delta} \cdummy \nabla \mathbb{P} (u^{\varepsilon, \delta}
   \cdummy \nabla u^{\varepsilon, \delta}), h \rangle d s \right| \right] \]
\[ \lesssim \varepsilon \mathbb{E} \left[ \int_0^t \| m^{\varepsilon, \delta}
   \|_{L^{\infty}} \| \mathbb{P} (u^{\varepsilon, \delta} \cdummy \nabla
   u^{\varepsilon, \delta}) \|_{L^{1+}} d s \right] \lesssim \varepsilon
   \mathbb{E} \left[ \int_0^t \| m^{\varepsilon, \delta} \|_{L^{\infty}} \|
   u^{\varepsilon, \delta} \|_{L^{2+}} \| \nabla u^{\varepsilon, \delta} \|_{L^2}
   d s \right] \lesssim \varepsilon , \]
   with exponents as in Section~\ref{s:notation}.
  In the last step we used interpolation between $L^{\infty}H$ and $L^{2}H^{1}$ to obtain a bound of
$u^{\varepsilon,\delta}$ in $L^{\infty-}L^{2+}$ with $\infty-$ denoting a large but finite exponent.

By a similar argument, the convective term of the second corrector is bounded
as
\[ \mathbb{E} \left[ \left| \varepsilon^2 \int_0^t \langle \mathbb{P}
   (u^{\varepsilon, \delta} \cdummy \nabla u^{\varepsilon, \delta}), D_u
   \varphi_2 (u^{\varepsilon, \delta}, m^{\varepsilon, \delta}) \rangle d s
   \right| \right] \]
\[ =\mathbb{E} \left[ \left| \frac{\varepsilon^2}{2} \int_0^t \langle
   m^{\varepsilon, \delta} \cdummy \nabla \mathbb{P} (m^{\varepsilon, \delta}
   \cdummy \nabla \mathbb{P} (u^{\varepsilon, \delta} \cdummy \nabla
   u^{\varepsilon, \delta})), h \rangle d s \right| \right] \lesssim
   \varepsilon^2 \delta^{- 1} . \]
\subsection{Martingale terms}\label{s:noise1}

The stochastic integral associated to the first corrector $\varphi_1$ is
estimated by Burkholder--Davis--Gundy's inequality, \eqref{eq:eigen}, the fact
that $\mathcal{F}_{\mathbb{R}^d} K$ is a Schwartz function and Parseval's
identity as
\[ \mathbb{E} \left[ \sup_{t \in [0, T]} \left| \int_0^t \langle
   \mathcal{Q}^{1 / 2}_{\delta} d W \cdummy \nabla u^{\varepsilon, \delta}, h
   \rangle \right|^2 \right] \lesssim \mathbb{E} \left[ \int_0^T \sum_{k,
   \alpha} | \langle (\mathcal{Q}^{1 / 2}_{\delta} \sigma_{k, \alpha}) \cdummy
   \nabla u^{\varepsilon, \delta}, h \rangle |^2 d s \right] \]
\[ =\mathbb{E} \left[ \int_0^T \delta^d \sum_{k, \alpha} |
   (\mathcal{F}_{\mathbb{R}^d} K) (\delta k) |^2 | \langle \sigma_{k, \alpha}
   \cdummy \nabla u^{\varepsilon, \delta}, h \rangle |^2 d s \right] \]
\[ \lesssim \delta^d \mathbb{E} \left[ \int_0^T \sum_{k, \alpha} | \langle
   \sigma_{k, \alpha}, (\nabla h)^T u^{\varepsilon, \delta} \rangle |^2 d s
   \right] \lesssim \delta^d \mathbb{E} \left[ \int_0^T \| u^{\varepsilon, \delta}
   \|_{L^2}^2 d s \right] \lesssim \delta^d . \]

The stochastic integrals associated to the second corrector $\varphi_2$ read
as
\[ \frac{\varepsilon}{2} \left( \int_0^t \langle \mathcal{Q}^{1 / 2}_{\delta}
   d W \cdummy \nabla \mathbb{P} (m^{\varepsilon, \delta} \cdummy \nabla
   u^{\varepsilon, \delta}), h \rangle + \int_0^t \langle m^{\varepsilon,
   \delta} \cdummy \nabla \mathbb{P} (\mathcal{Q}^{1 / 2}_{\delta} d W \cdummy
   \nabla u^{\varepsilon, \delta}), h \rangle \right) . \]
We apply It{\^o}'s isometry, \eqref{eq:eigen} and \eqref{eq:supm} to estimate
\[ \varepsilon^2 \mathbb{E} \left[ \left| \int_0^t \langle \mathcal{Q}^{1 /
   2}_{\delta} d W \cdummy \nabla \mathbb{P} (m^{\varepsilon, \delta} \cdummy
   \nabla u^{\varepsilon, \delta}), h \rangle \right|^2 \right] =
   \varepsilon^2 \mathbb{E} \left[ \int_0^t \sum_{k, \alpha} | \langle
   (\mathcal{Q}^{1 / 2}_{\delta} \sigma_{k, \alpha}) \cdummy \nabla \mathbb{P}
   (m^{\varepsilon, \delta} \cdummy \nabla u^{\varepsilon, \delta}), h \rangle
   |^2 d s \right] \]
\[ = \varepsilon^2 \mathbb{E} \left[ \int_0^t \delta^d \sum_{k, \alpha} |
   (\mathcal{F}_{\mathbb{R}^d} K) (\delta k) |^2 | \langle \sigma_{k, \alpha}
   \cdummy \nabla \mathbb{P} (m^{\varepsilon, \delta} \cdummy \nabla
   u^{\varepsilon, \delta}), h \rangle |^2 d s \right] \]
\[ \lesssim \varepsilon^2 \mathbb{E} \left[ \int_0^t \delta^d \sum_{k, \alpha}
   | (\mathcal{F}_{\mathbb{R}^d} K) (\delta k) |^2 \| \sigma_{k, \alpha}
   \|_{L^4}^2 \| m^{\varepsilon, \delta} \|_{L^4}^2 \| \nabla u^{\varepsilon,
   \delta} \|_{L^2}^2 \right] \lesssim \varepsilon^2 \log \varepsilon^{- 1} .
\]
\subsection{Uniform time regularity}\label{s:time}

We will show that the quantity $u^{\varepsilon, \delta} + \varepsilon
\mathbb{P} (m^{\varepsilon, \delta} \cdummy \nabla u^{\varepsilon, \delta})$
is uniformly bounded in $L^2 (\Omega ; W^{\gamma, 2} (0, T ; H^{- \theta}))$
for some $\theta > 0$ and any $\gamma \in (0, 1 / 2)$. While $u^{\varepsilon,
\delta}$ alone does not enjoy uniform time regularity, the first corrector
term $\varepsilon \mathbb{P} (m^{\varepsilon, \delta} \cdot \nabla
u^{\varepsilon, \delta})$ vanishes in the limit, so that the uniform time
regularity of the sum suffices to recover strong convergence of
$u^{\varepsilon, \delta}$ in $L^2 (0, T ; H)$, as needed for the passage to
the limit in the convective term.

It{\^o}'s formula yields
\[ d [u^{\varepsilon, \delta} + \varepsilon \mathbb{P} (m^{\varepsilon,
   \delta} \cdummy \nabla u^{\varepsilon, \delta})] = [\Delta u^{\varepsilon,
   \delta} -\mathbb{P} (u^{\varepsilon, \delta} \cdummy \nabla u^{\varepsilon,
   \delta}) +\mathbb{P} (m^{\varepsilon, \delta} \cdummy \nabla \mathbb{P}
   (m^{\varepsilon, \delta} \cdummy \nabla u^{\varepsilon, \delta}))] d t \]
\begin{equation}
  + [\varepsilon \mathbb{P} (m^{\varepsilon, \delta} \cdummy \nabla (\Delta
  u^{\varepsilon, \delta})) - \varepsilon \mathbb{P} (m^{\varepsilon, \delta}
  \cdummy \nabla \mathbb{P} (u^{\varepsilon, \delta} \cdummy \nabla
  u^{\varepsilon, \delta}))] d t + d\mathbb{P}M^{\varepsilon, \delta},
  \label{eq:118}
\end{equation}
where $M^{\varepsilon, \delta}$ is the martingale term associated to the first
corrector, treated in Section~\ref{s:noise1}. It was shown that
$M^{\varepsilon, \delta} \rightarrow 0$ in $L^2 (\Omega ; C ([0, T] ; H^{-
\theta}))$ for some $\theta > 0$, and Lemma~2.1 in {\cite{MR1339739}} then
yields $M^{\varepsilon, \delta} \rightarrow 0$ in $L^2 (\Omega ; W^{\gamma, 2}
(0, T ; H^{- \theta}))$ for every $\gamma \in (0, 1 / 2)$.

Our goal now is to show a uniform in time regularity of the time integral of
the remaining terms on the right hand side of \eqref{eq:118} which we jointly
denote by $N^{\varepsilon, \delta}$. By \eqref{eq:energy} and \eqref{eq:m13}
and Ladyzhenskaya inequality we have
\[ \mathbb{E} \left[ \left| \int_s^t \langle \Delta u^{\varepsilon, \delta}, h
   \rangle d r \right|^2 \right] \lesssim \mathbb{E} \left[ \left( \int_s^t \|
   u^{\varepsilon, \delta} \|_{L^2} d r \right)^2 \right] \lesssim | t - s
   |^2, \]
\[ \mathbb{E} \left[ \left| \int_s^t \langle u^{\varepsilon, \delta} \cdummy
   \nabla u^{\varepsilon, \delta}, h \rangle d r \right|^2 \right] \lesssim
   \mathbb{E} \left[ \left| \int_s^t \| u^{\varepsilon, \delta} \|_{L^2}^2 d r
   \right|^2 \right] \lesssim | t - s |^2, \]
\[ \mathbb{E} \left[ \left| \int_s^t \langle m^{\varepsilon, \delta} \cdummy
   \nabla \mathbb{P} (m^{\varepsilon, \delta} \cdummy \nabla u^{\varepsilon,
   \delta}), h \rangle d r \right|^2 \right] \lesssim \mathbb{E} \left[ \left(
   \int_s^t \| m^{\varepsilon, \delta} \|_{L^4}^2 \| \nabla u^{\varepsilon,
   \delta} \|_{L^2} d r \right)^2 \right] \lesssim | t - s | . \]
The remaining terms are controlled similarly to Section~\ref{s:err} as follows
\[ \varepsilon^2 \mathbb{E} \left[ \left| \int_s^t \langle m^{\varepsilon,
   \delta} \cdummy \nabla (\Delta u^{\varepsilon, \delta}), h \rangle d r
   \right|^2 \right] \lesssim \varepsilon^2 \mathbb{E} \left[ \left( \int_s^t
   \| m^{\varepsilon, \delta} \|_{H^1} \| \nabla u^{\varepsilon, \delta}
   \|_{L^2} d r \right)^2 \right] \lesssim \varepsilon^2 \delta^{- 2} | t - s
   | \lesssim | t - s |, \]
\[ \varepsilon^2 \mathbb{E} \left[ \left| \int_s^t \langle m^{\varepsilon,
   \delta} \cdummy \nabla \mathbb{P} (u^{\varepsilon, \delta} \cdummy \nabla
   u^{\varepsilon, \delta}), h \rangle d r \right|^2 \right] \lesssim
   \varepsilon^2 \mathbb{E} \left[ \left( \int_s^t \| m^{\varepsilon, \delta}
   \|_{L^{\infty}} \| u^{\varepsilon, \delta} \|_{L^{2+}} \| \nabla
   u^{\varepsilon, \delta} \|_{L^2} d r \right)^2 \right] 
   \]
   \[\lesssim
   \varepsilon^2 | t - s |^{1-\kappa} \]
   for arbitrarily small $\kappa\in (0,1)$.
Altogether, we have shown that for every $\kappa \in (0, 1)$
\[ \mathbb{E} [\| \langle N^{\varepsilon, \delta}, h \rangle \|_{W^{\gamma, 2}
   (0, T)}^2] = \int_0^T \int_0^T \frac{\mathbb{E} [| \langle N^{\varepsilon,
   \delta}_t, h \rangle - \langle N^{\varepsilon, \delta}_s, h \rangle |^2]}{|
   t - s |^{2 \gamma + 1}} d s d t \lesssim \int_0^T \int_0^T \frac{1}{| t - s
   |^{2 \gamma+\kappa}} d s d t \]
which is finite provided $2 \gamma+\kappa < 1$. For any
 $\gamma\in (0,1 / 2)$, it suffices
 to choose $\kappa\in (0,1- 2\gamma)$. Hence $N^{\varepsilon, \delta}$ is
uniformly bounded in $L^2 (\Omega ; W^{\gamma, 2} (0, T ; H^{- \theta}))$ for
some $\theta > 0$ and any $\gamma \in (0, 1 / 2)$. Combining it with the
vanishing of $M^{\varepsilon, \delta}$, this yields the desired uniform bound
of $u^{\varepsilon, \delta} + \varepsilon \mathbb{P}(m^{\varepsilon, \delta} \cdummy
\nabla u^{\varepsilon, \delta})$ in $L^2 (\Omega ; W^{\gamma, 2} (0, T ; H^{-
\theta}))$ for some $\theta > 0$ and any $\gamma \in (0, 1 / 2)$.

\subsection{Compactness and passage to the limit}\label{s:comp}

By the energy estimate \eqref{eq:energy}, the family $u^{\varepsilon,
\delta}$, $\varepsilon, \delta \in (0, 1)$, is tight in
\[ \mathcal{X}_u \assign (L^{\infty} (0, T ; H), w^{\ast}) \cap (L^2 (0, T ;
   H^1), w), \]
where $w^{\ast}$ and $w$ denote the weak-star and weak topology respectively.
Setting $w^{\varepsilon, \delta} \assign u^{\varepsilon, \delta} + \varepsilon
\mathbb{P} (m^{\varepsilon, \delta} \cdummy \nabla u^{\varepsilon, \delta})$,
it follows from Section~\ref{s:time} and the compact embedding $W^{\gamma', 2}
([0, T] ; H^{- \theta'}) \subset W^{\gamma, 2} ([0, T] ; H^{- \theta})$ for
$\gamma' > \gamma$ and $\theta' < \theta$, that the family $w^{\varepsilon,
\delta}, \varepsilon, \delta \in (0, 1)$, is tight in
\[ \mathcal{X}_w \assign W^{\gamma, 2} ([0, T] ; H^{- \theta}), \]
for some $\theta > 0$ and any $\gamma \in (0, 1 / 2)$. Finally, the law of the
cylindrical Wiener process on $H$ is tight in
\[ \mathcal{X}_W \assign C ([0, T] ; H^{- d / 2 - \kappa}) \]
for any $\kappa > 0$.

We apply Skorokhod--Jakubowski representation theorem {\cite{J98}} to the
triple $(u^{\varepsilon, \delta}, w^{\varepsilon, \delta}, W)$ as a random
variable in the sub-Polish path space $\mathcal{X} \assign \mathcal{X}_u
\times \mathcal{X}_w \times \mathcal{X}_W$. Consequently, there exist a
sequence $\varepsilon, \delta \rightarrow 0$, a probability space
$(\tilde{\Omega}, \tilde{\mathcal{F}}, \tilde{\mathbf{P}})$ and random
variables $(\tilde{u}^{\varepsilon, \delta}, \tilde{w}^{\varepsilon, \delta},
\tilde{W}^{\varepsilon, \delta})$ with the same joint law as $(u^{\varepsilon,
\delta}, w^{\varepsilon, \delta}, W)$, such that $(\tilde{u}^{\varepsilon,
\delta}, \tilde{w}^{\varepsilon, \delta}, \tilde{W}^{\varepsilon, \delta})
\rightarrow (\tilde{u}, \tilde{w}, \tilde{W})$ $\tilde{\mathbf{P}}$-a.s. in
$\mathcal{X}$.

Let $\tilde{m}^{\varepsilon, \delta}$ be the stationary Ornstein--Uhlenbeck
process driven by $\tilde{W}^{\varepsilon, \delta}$, i.e. satisfying
\eqref{eq:med} with $W$ replaced by $\tilde{W}^{\varepsilon, \delta}$. Since
$\tilde{m}^{\varepsilon, \delta}$ is a measurable function of
$\tilde{W}^{\varepsilon, \delta}$, equality of joint laws gives
$\tilde{w}^{\varepsilon, \delta} = \tilde{u}^{\varepsilon, \delta} +
\varepsilon \mathbb{P} (\tilde{m}^{\varepsilon, \delta} \cdummy \nabla
\tilde{u}^{\varepsilon, \delta})$ $\tilde{\mathbf{P}}$-a.s., and
$\tilde{u}^{\varepsilon, \delta}$ satisfies \eqref{eq:u} with $m^{\varepsilon,
\delta}$ replaced by $\tilde{m}^{\varepsilon, \delta}$. All uniform estimates
carry over to the new probability space.

It remains to show that $\tilde{u}^{\varepsilon, \delta} \rightarrow
\tilde{u}$ in $L^2 (0, T ; H)$, which is needed to pass to the limit in the
convective term. Since $\varepsilon \mathbb{P} (\tilde{m}^{\varepsilon,
\delta} \cdummy \nabla \tilde{u}^{\varepsilon, \delta}) \rightarrow 0$ in $L^2
(\tilde{\Omega} ; L^2 (0, T ; H^{- \theta}))$ for some $\theta > 0$, passing
to a subsequence we have $\tilde{u}^{\varepsilon, \delta} =
\tilde{w}^{\varepsilon, \delta} - \varepsilon \mathbb{P}
(\tilde{m}^{\varepsilon, \delta} \cdummy \nabla \tilde{u}^{\varepsilon,
\delta}) \rightarrow \tilde{w}$ in $L^2 (0, T ; H^{- \theta})$
$\tilde{\mathbf{P}}$-a.s. By uniqueness of the limit, $\tilde{w} = \tilde{u}$.
Since $\tilde{u}^{\varepsilon, \delta}$ is $\tilde{\mathbf{P}}$-a.s. bounded
in $L^2 (0, T ; H^1)$, interpolation between $L^2 (0, T ; H^{- \theta})$ and
$L^2 (0, T ; H^1)$ yields convergence in $L^2 (0, T ; H)$.

To conclude, we write down the equation satisfied by $\varphi
(\tilde{u}^{\varepsilon, \delta}) + \varepsilon \varphi_1
(\tilde{u}^{\varepsilon, \delta}, \tilde{m}^{\varepsilon, \delta}) +
\varepsilon^2 \varphi_2 (\tilde{u}^{\varepsilon, \delta},
\tilde{m}^{\varepsilon, \delta})$ and pass to the limit following the
arguments of Proposition~\ref{p:11re}, Section~\ref{s:err} and
Section~\ref{s:noise1}. Since $\tilde{u}^{\varepsilon, \delta}$ and
$\tilde{m}^{\varepsilon, \delta}$ have the same laws as $u^{\varepsilon,
\delta}$ and $m^{\varepsilon, \delta}$, the estimates established in
Section~\ref{s:err} and Section~\ref{s:noise1} on the original probability
space apply equally on the new space, and the error terms vanish in the same
sense. Moreover, the term $\langle S_{\delta} u^{\varepsilon, \delta}, h
\rangle$ is handled by writing $\langle S_{\delta} u^{\varepsilon, \delta}, h
\rangle = \langle u^{\varepsilon, \delta}, S_{\delta} h \rangle$ using
self-adjointness of $S_{\delta}$, and then passing to the limit using the
convergence $S_{\delta} h \rightarrow \nu \mathbb{P} \Delta h$ in $L^2$ from
Proposition~\ref{p:11re} (since $h$ is smooth) together with the strong
convergence of $\tilde{u}^{\varepsilon, \delta}$ to $\tilde{u}$ in $L^2 (0, T
; H)$. We conclude that $\tilde{u}$ satisfies the deterministic Navier--Stokes
equations with enhanced diffusion
\[ \partial_t \tilde{u} +\mathbb{P} (\tilde{u} \cdummy \nabla \tilde{u}) = (1
   + \nu) \Delta \tilde{u}, \quad \tmop{div} \tilde{u} = 0, \quad \tilde{u}
   (0) = u_0, \]
with $\nu = \frac{1}{16} \| K \|_{L^2}^2$ in $d = 2$ and $\nu = \frac{1}{5} \|
K \|_{L^2}^2$ in $d=3$, completing the proof of Theorem~\ref{thm:1}.

\begin{remark}
  The estimates established in Section~\ref{s:err} include several error
  terms, most of which vanish even in the critical case $\varepsilon =
  \delta$. The only terms that do not vanish are bounded by $\varepsilon
  \delta^{- 1}$ and $(\varepsilon \delta^{- 1})^2$, which in the critical case
  remain bounded uniformly in $\varepsilon, \delta$. Moreover, the time
  regularity argument of Section~\ref{s:time} likewise goes through in the
  critical case, hence the compactness argument above applies in the critical
  case as well, yielding the existence of a subsequential limit in the strong
  $L^2 (0, T ; H)$ topology. However, since some of the error terms do not
  tend to zero, the present argument does not suffice to identify the limit
  equation, and the critical case remains open.
\end{remark}

\subsection{Strong convergence in two dimensions}

\begin{corollary}
  \label{c:14}Let $d = 2$ and $u_0 \in L^2$. Then the approximate solutions
  $u^{\varepsilon, \delta}$, $\varepsilon, \delta \in (0, 1]$, defined on a
  common probability space converge to the unique Leray solution $u$ of
  \eqref{eq:ulim} strongly in $L^2 (0, T ; H^1)$ in probability as
  $\varepsilon, \delta \rightarrow 0$.
\end{corollary}

\begin{proof}
  In two dimensions, the energy inequality \eqref{eq:energy} becomes an
  equality and the same is valid for the limit solution $u$: for all $t
  \geqslant 0$
  \begin{equation}
    \| u_t^{\varepsilon, \delta} \|_{L^2}^2 + 2 \int_0^t \| \nabla
    u_s^{\varepsilon, \delta} \|_{L^2}^2 d s = \| u_0 \|_{L^2}^2 = \| u_t
    \|_{L^2}^2 + 2 \int_0^t \| \nabla u_s \|_{L^2}^2 d s. \label{eq:energyeq}
  \end{equation}
  By weak lower semicontinuity
  \[ \| u_t \|_{L^2}^2 \leqslant \liminf_{\varepsilon, \delta \rightarrow 0}
     \| u^{\varepsilon, \delta}_t \|_{L^2}^2, \qquad \int_0^t \| \nabla u_s
     \|_{L^2}^2 d s \leqslant \liminf_{\varepsilon, \delta \rightarrow 0}
     \int_0^t \| \nabla u^{\varepsilon, \delta}_s \|_{L^2}^2 d s. \]
  Since the left hand sides sum to $\| u_0 \|_{L^2}^2$ by \eqref{eq:energyeq},
  equality must hold in both, giving
  \[ \| u_t^{\varepsilon, \delta} \|_{L^2}^2 \rightarrow \| u_t \|_{L^2}^2,
     \qquad \int_0^t \| \nabla u_s^{\varepsilon, \delta} \|_{L^2}^2 d s
     \rightarrow \int_0^t \| \nabla u_s \|_{L^2}^2 d s. \]
  Strong convergence in $L^2 (0, T ; H^1)$ then follows by combining norm
  convergence with weak convergence in $L^2 (0, T ; H^1)$.
\end{proof}

\section{Higher order expansion}\label{s:high}

The length $N$ of the corrector expansion is dictated by the size of $\iota >
0$ in the strict subcriticality condition $\varepsilon = o (\delta^{1 +
\iota})$. Specifically, given $\iota > 0$, we determine $N \in \mathbb{N}$ as
the smallest integer satisfying the strict inequality
\begin{equation}
  \iota > \frac{d}{2 N} . \label{eq:N}
\end{equation}
Under this condition, we have $\varepsilon^N \delta^{- N - d / 2} \rightarrow
0$, meaning the corrector terms of order $N$ vanish after the rescaling by
$\delta^{d / 2}$. The corrector terms of lower order must be cancelled by
further correctors.

Moreover, the strict inequality in \eqref{eq:N} permits one to absorb an
additional logarithmic blow up obtained from $L^{\infty}$ in time bounds for
the Ornstein--Uhlenbeck process (cf. Section~\ref{s:OU}). These are required
if we estimate the fluctuation dynamics in $L^2$ with respect to time, since
the velocity gradient belongs only to $L^2$ in time.

In order to make the expansion systematic, recall that for a general
observable $\varphi (u, n)$, the generator acts like
\begin{equation}
  \mathcal{L}^{\varepsilon, \delta} \varphi (u, n) =
  \tmcolor{blue}{\varepsilon^{- 2} \mathcal{M}^{\delta} \varphi (u, n)} +
  \tmcolor{magenta}{\varepsilon^{- 1} \langle \mathbb{P} (n \cdummy \nabla u),
  D_u \varphi \rangle} + \tmcolor{magenta}{\langle \Delta u, D_u \varphi
  \rangle} - \langle \mathbb{P} (u \cdummy \nabla u), D_u \varphi \rangle,
  \label{eq:gen}
\end{equation}
where in the case of the Ornstein--Uhlenbeck process, the generator of
$m^{\varepsilon, \delta} (t, x) = m (\varepsilon^{- 2} t, \delta^{- 1} x)$ is
\[ \varepsilon^{- 2} \mathcal{M}^{\delta} \varphi \assign - \varepsilon^{- 2}
   \langle n, D_n \varphi \rangle + \frac{\varepsilon^{- 2}}{2} \tmop{Tr}
   (\mathcal{Q}_{\delta} D^2_n \varphi) . \]
We always use the blue term associated to each corrector to cancel the mean-free part of one of the magenta terms of another observable. The terms coming
from the convective term are of lower order and do not cause any problems.

To make the expansion systematic, we denote by $\sigma$ a finite binary
sequence, assigning the application of the operator $\varepsilon^{- 1}
\mathbb{P} (n \cdummy \nabla \cdummy)$ (as appearing in the first magenta term
in \eqref{eq:gen}) to the number $0$, whereas the application of $\Delta$ (in
the second magenta term in \eqref{eq:gen}) is assigned to $1$. Denoting by $|
\sigma |$ the length of $\sigma$, we define $L \assign L_{\sigma} \assign
\sum_{i = 1}^{| \sigma |} \sigma_i$ to be the number of times we used
$\Delta$, and $M \assign M_{\sigma} \assign | \sigma | - L$ is the number of
times we used $\varepsilon^{- 1} \mathbb{P} (n \cdummy \nabla \cdummy)$.

The relevant sequences $\sigma$ are taken from the index set $\Sigma$ defined
recursively as follows: $(0), (0, 0) \in \Sigma$ (these are the first and
second corrector which are always included); and if $\sigma \in \Sigma$ and $M
+ 2 L - 1 < N$ then $(\sigma, 0) \in \Sigma$; if $\sigma \in \Sigma$ and $M +
2 L < N$ then $(\sigma, 1) \in \Sigma$. Note that this is consistent with the
situation for $N = 1$. Indeed, in this case  we only have the first and
second corrector, i.e. $\Sigma = \{ (0), (0, 0) \}$. In case $N = 2$ we
additionally include $\sigma = (0, 1)$ and $\sigma = (0, 0, 0)$.

Considering the observable $\varphi (u) = \langle u, h \rangle$ for a
sufficiently smooth test function $h$, the correctors are defined recursively
as
\[ \varphi_{(0)} (u, n) \assign (-\mathcal{M}^{\delta})^{- 1} \langle
   \mathbb{P} (n \cdummy \nabla u), D_u \varphi \rangle, \]
\begin{equation}
  \varphi_{(\sigma, 0)} (u, n) \assign (-\mathcal{M}^{\delta})^{- 1} [\langle
  \mathbb{P} (n \cdummy \nabla u), D_u \varphi_{\sigma} \rangle]^{\circ},
  \label{eq:aa1}
\end{equation}
\begin{equation}
  \varphi_{(\sigma, 1)} (u, n) \assign (-\mathcal{M}^{\delta})^{- 1} [\langle
  \Delta u, D_u \varphi_{\sigma} \rangle]^{\circ}, \label{eq:aa2}
\end{equation}
where $(\sigma, 0)$ and $(\sigma, 1)$ are the multi-indices obtained by
appending 0 or 1 to $\sigma$, respectively.

Observe that each corrector satisfies $\varphi_{\sigma} (\cdummy, n) : L^2
(\mathbb{T}^d) \rightarrow \mathbb{R}$ and is linear and bounded, that is,
$\varphi_{\sigma} (\cdummy, n) \in (L^2 (\mathbb{T}^d))^{\ast} \cong L^2
(\mathbb{T}^d)$ and $\varphi_{\sigma} (u, n) = \langle u, \Phi_{\sigma} (n)
\rangle$ for $\Phi_{\sigma} (n) \in L^2 (\mathbb{T}^d)$ and we identify
$\varphi_{\sigma}$ with $\Phi_{\sigma}$. Then the Fr\'echet derivative satisfies
$D_u \varphi_{\sigma} = \Phi_{\sigma}$ and \eqref{eq:aa1}, \eqref{eq:aa2}
rewrite as
\begin{equation}
  \varphi_{(\sigma, 0)} (u, n) \assign (-\mathcal{M}^{\delta})^{- 1}
  [\varphi_{\sigma} (\mathbb{P} (n \cdummy \nabla u), n)]^{\circ},
  \label{eq:aa11}
\end{equation}
\begin{equation}
  \varphi_{(\sigma, 1)} (u, n) \assign (-\mathcal{M}^{\delta})^{- 1}
  [\varphi_{\sigma} (\Delta u, n)]^{\circ} . \label{eq:aa22}
\end{equation}
Each corrector $\varphi_{\sigma}$ is additionally multiplied by a certain
power of $\varepsilon$, which is fully determined by the parameters $(M, L)$.
Specifically, the first corrector $\varphi_1 = \varphi_{(0)}$ is multiplied by
$\varepsilon$ and has $M = 1$, $L = 0$. The second corrector $\varphi_2 =
\varphi_{(0, 0)}$ is quadratic in $n$, i.e. $M = 2$, and is multiplied by
$\varepsilon^2$. Generally, this leads to the prefactor $\varepsilon^{- M + 2
| \sigma |} = \varepsilon^{M + 2 L}$ since: each application of
$\varepsilon^{- 1} n \cdummy \nabla$ gives $\varepsilon^{- 1}$ hence we obtain
$\varepsilon^{- M}$ and for each new generation, i.e. for each application of
the It{\^o} formula, i.e. application of either $\varepsilon^{- 1} n \cdummy
\nabla$ or $\Delta$ we include $\varepsilon^2$, meant to cancel
$\varepsilon^{- 2}$ in the blue term in \eqref{eq:gen}.

Accordingly, the correctors are of the form $\varepsilon^{M + 2 L}
\varphi_{\sigma}$ and the goal is then to apply the It{\^o} formula to
\[ \varphi (u^{\varepsilon, \delta}_t) + \sum_{\sigma \in \Sigma}
   \varepsilon^{M + 2 L} \varphi_{\sigma} (u^{\varepsilon, \delta}_t,
   m^{\varepsilon, \delta}_t) . \]
In the Ornstein--Uhlenbeck case, there is a simple way to obtain explicit
formulas for the correctors based on the Wiener chaos expansion, which we
exploit in the sequel.

\subsection{Inverse Ornstein--Uhlenbeck generator via Wiener
chaos}\label{s:chaos}

In this section we derive an explicit formula for the inverse
Ornstein--Uhlenbeck generator acting on polynomial functionals of the random
field. The key result is a standard consequence of Wiener chaos theory: the
action of $(-\mathcal{M}^{\delta})^{- 1}$ is expressed as a finite sum of
contraction terms obtained by pairing tensor slots with copies of the
covariance $\frac{1}{2} \mathcal{Q}_{\delta}$.

\begin{proposition}
  \label{p:chaos}Let $k \in \mathbb{N}$, let $G \in H^{\odot k}$ be a
  symmetric kernel, and define the homogeneous polynomial observable $F (n)
  \assign \langle n^{\otimes k}, G \rangle_{H^{\otimes k}}$, $n \in H$. Denote
  by $F^{\circ} \assign F -\mathbb{E}_{\nu^{\delta}} [F]$ the mean-free part.
  For $\ell = 0, 1, \ldots, [(k - 1) / 2]$, define the $\ell$-fold contraction
   $\tilde{\mathcal{C}}^{\delta}_{\ell} G \in H^{\odot (k - 2 \ell)}$ by inserting
  $\ell$ copies of the covariance $\frac{1}{2} \mathcal{Q}_{\delta}$,
  specifically, for $f \in H^{\otimes (k - 2 \ell)}$ we let
  \begin{equation}
    \langle \tilde{\mathcal{C}}_{\ell}^{\delta} G, f \rangle_{H^{\otimes (k -
    2 \ell)}} \assign \sum_{i_1, \ldots, i_{\ell} \in \mathbb{N}} \langle G,
    T_{\delta} e_{i_1} \otimes T_{\delta} e_{i_1} \otimes \cdots \otimes
    T_{\delta} e_{i_{\ell}} \otimes T_{\delta} e_{i_{\ell}} \otimes f
    \rangle_{H^{\otimes k}}, \label{eq:contr12}
  \end{equation}
  where $(e_i)_{i \in \mathbb{N}}$ is any orthonormal basis of $H$ and
  $T_{\delta} : H \rightarrow H$ is defined by $T_{\delta} g =
  \frac{1}{\sqrt{2}} \mathcal{Q}^{1 / 2}_{\delta} g$. Then
  \begin{equation}
    ((-\mathcal{M}^{\delta})^{- 1} F^{\circ}) (m_t^{\delta}) \cong \sum_{\ell
    = 0}^{[(k - 1) / 2]} \langle (m_t^{\delta})^{\otimes (k - 2 \ell)},
    \tilde{\mathcal{C}}_{\ell}^{\delta} G \rangle_{H^{\otimes (k - 2 \ell)}},
    \label{eq:final}
  \end{equation}
  where $\cong$ denotes equality up to combinatorial constants depending only
  on $k$ and $\ell$. In particular, controlling $(-\mathcal{M}^{\delta})^{- 1}
  F^{\circ}$ reduces to estimating $F$ together with finitely many explicit
  contraction terms of strictly lower polynomial degree, each obtained by
  replacing pairs of $n$-factors by independent copies $n_i$ whose
  expectations are taken outside the spatial pairing by Tonelli's theorem
  before estimating.
\end{proposition}

\begin{proof}
  We first develop the Wiener chaos machinery for a canonical
  Ornstein--Uhlenbeck process with identity covariance, then transform to the
  actual process $m^{\delta}$ via a change of measure. This approach keeps the
  chaos formulas standard while allowing estimates to be carried out in the
  natural $L^2$-duality.
  
  Step 1: Wiener chaos analysis for the canonical Ornstein--Uhlenbeck process.
  Consider a probability space $(\Omega, \mathcal{F}, \mathbf{P})$ with an
  auxiliary stationary Ornstein--Uhlenbeck process $\zeta$ on $H$, the
  solenoidal, mean-free subspace of $L^2 (\mathbb{T}^d ; \mathbb{R}^d)$,
  satisfying
  \[ d \zeta = - \zeta d t + \sqrt{2} d W, \]
  with invariant law $\mu \assign \mathcal{N} (0, \tmop{Id})$. For every fixed
  $t \geqslant 0$, the map $\eta (f) \assign \langle \zeta_t, f \rangle$
  defines an isonormal Gaussian process over $H$, that is, a centered Gaussian
  family of random variables such that $\mathbb{E} [\eta (f) \eta (g)] =
  \langle f, g \rangle$ for all $f, g \in H$.
  
  For $k \in \mathbb{N}$ and $f_1, ..., f_k \in H$, we denote by $: \eta (f_1)
  \cdots \eta (f_k) :$ the Wick (normal-ordered) product of the Gaussian
  variables $\eta (f_i)$, defined by subtracting all Gaussian contractions,
  equivalently characterized by the property that it is orthogonal in $L^2
  (\Omega)$ to all polynomials in $\eta$ of degree strictly less than $k$. We
  also include the case $k = 0$ by letting $: : = 1$.
  
  The Gaussian space generated by $\eta$ admits the Wiener chaos
  decomposition, that is, it holds
  \[ L^2 (\Omega, \sigma (\eta), \mathbf{P}) = \bigoplus_{k = 0}^{\infty}
     H^{: k :}, \]
  where $H^{: k :}$ denotes the $k$-th homogeneous Wiener chaos, defined for
  $k = 0$ as $H^{: 0 :} =\mathbb{R}$ and for $k \in \mathbb{N}$ as the closed
  linear span of Wick monomials of degree $k$, i.e.
  \[ H^{: k :} = \overline{\tmop{span}} \{ : \eta (f_1) \cdots \eta (f_k) :,
     f_1, \ldots, f_k \in H \} . \]
  In particular, any polynomial functional of $\eta$ decomposes into a finite
  sum of homogeneous chaoses. Moreover, $H^{: k :}$ is canonically isomorphic
  to the symmetric tensor power $H^{\odot k}$ (equivalently, to the space of
  $k$-fold multiple Wiener integrals), so that elements of $H^{: k :}$ may be
  represented by symmetric kernels. The associated Ornstein--Uhlenbeck
  generator $\mathcal{L}$ coincides with the Kolmogorov generator of the
  Markov process $\zeta$ and acts diagonally on this decomposition: on $H^{: k
  :}$ one has $\mathcal{L}= - k \tmop{Id}$ for every $k \in \mathbb{N}$.
  Consequently, $(-\mathcal{L})^{- 1}$ is well defined on the mean-free
  subspace $\bigoplus_{k \in \mathbb{N}} H^{: k :}$ and acts there by division
  by the chaos order, i.e. $(-\mathcal{L})^{- 1} F = \frac{1}{k} F$, $F \in
  H^{: k :}$, $k \in \mathbb{N}$. All these facts are standard; see, for
  instance {\cite{MR1474726}}, Chapter II or {\cite{MR2200233}}, Chapter 1.
  
  Let $k \in \mathbb{N}$ and let $G \in H^{\odot k}$ be a symmetric kernel. We
  define the corresponding homogeneous polynomial observable as $F (\eta)
  \assign \langle \eta^{\otimes k}, G \rangle_{H^{\otimes k}}$, where the
  pairing is defined by multilinear extension and density from simple tensors.
  By Wick's formula (see {\cite{MR1474726}}, Thm. III.3.15), $F (\eta)$ admits
  the chaos expansion
  \begin{equation}
    \langle \eta^{\otimes k}, G \rangle_{H^{\otimes k}} \cong \sum_{\ell =
    0}^{[k / 2]} \langle : \eta^{\otimes (k - 2 \ell)} :, \mathcal{C}_{\ell} G
    \rangle_{H^{\otimes (k - 2 \ell)}}, \label{eq:wick1}
  \end{equation}
  where $\mathcal{C}_{\ell} G \in H^{\odot (k - 2 \ell)}$ denotes the
  $\ell$-fold contraction of $G$, defined by
  \[ \langle \mathcal{C}_{\ell} G, f \rangle_{H^{\otimes (k - 2 \ell)}}
     \assign \sum_{i_1, \ldots, i_{\ell} \in \mathbb{N}} \langle G, e_{i_1}
     \otimes e_{i_1} \otimes \cdots \otimes e_{i_{\ell}} \otimes e_{i_{\ell}}
     \otimes f \rangle_{H^{\otimes k}}, \]
  for every $f \in H^{\otimes (k - 2 \ell)}$ and any orthonormal basis
  $(e_i)_{i \in \mathbb{N}}$ of $H$. Each contraction corresponds to inserting
  the covariance of the isonormal process, i.e. the identity on $H$, which in
  an orthonormal basis $(e_i)_{i \in \mathbb{N}}$ is represented by $\sum_{i
  \in \mathbb{N}} e_i \otimes e_i$. Note that since $G \in H^{\odot k}$ is
  symmetric, its $\ell$-fold contraction $\mathcal{C}_{\ell} G$ is again a
  symmetric tensor in $H^{\odot (k - 2 \ell)}$.
  
  Denote by $F^{\circ, \mu} \assign F -\mathbb{E}_{\mu} [F]$ the mean-free
  part. The zero-th chaos component (which occurs only when $k$ is even)
  corresponds to $\ell = k / 2$ and subtracting $\mathbb{E}_{\mu} [F]$ removes
  precisely the constant part. Using that $\mathcal{L}$ acts as $- r$ on $H^{:
  r :}$, we obtain
  \[ (-\mathcal{L})^{- 1} F^{\circ, \mu} \cong \sum_{\ell = 0}^{[(k - 1) / 2]}
     \frac{1}{k - 2 \ell} \langle : \eta^{\otimes (k - 2 \ell)} :,
     \mathcal{C}_{\ell} G \rangle_{H^{\otimes (k - 2 \ell)}} . \]
  Finally, to return to expressions suitable for pointwise estimates, we undo
  the Wick ordering. By the inverse Wick formula (see {\cite{MR1474726}}, Thm.
  III.3.4), each Wick monomial can be expanded as an alternating sum of
  ordinary monomials with further contractions
  \[ \langle : \eta^{\otimes r} :, G_r \rangle_{H^{\otimes r}} \cong \sum_{j =
     0}^{[r / 2]} (- 1)^j \langle \eta^{\otimes (r - 2 j)}, \mathcal{C}_j G_r
     \rangle_{H^{\otimes (r - 2 j)}} . \]
  Accordingly, $(-\mathcal{L})^{- 1} F^{\circ, \mu}$ is a finite linear
  combination of ordinary monomials in $\eta$ obtained from $F$ by iterated
  Gaussian contractions. Specifically,
  \begin{equation}
    (-\mathcal{L})^{- 1} F^{\circ, \mu} (\eta) \cong \sum_{\ell = 0}^{[(k - 1)
    / 2]} \langle \eta^{\otimes (k - 2 \ell)}, \mathcal{C}_{\ell} G
    \rangle_{H^{\otimes (k - 2 \ell)}} . \label{eq:formula}
  \end{equation}
  Step 2: Application to the Ornstein--Uhlenbeck process $m^{\delta}$. Our
  original stationary Ornstein--Uhlenbeck process $m^{\delta}$ with invariant
  law $\nu^{\delta} =\mathcal{N}(0, \frac{1}{2} \mathcal{Q}_{\delta})$ is
  obtained from $\zeta$ by a deterministic transformation
  \[ m^{\delta}_t = T_{\delta} \zeta_t \quad \tmop{with} \quad T_{\delta} : H
     \rightarrow H, \quad f \mapsto \frac{1}{\sqrt{2}} \mathcal{Q}^{1 /
     2}_{\delta} f. \]
  In particular, for all $f, g \in H$ and $s, t \geqslant 0$,
  \[ \mathbb{E} [\langle m^{\delta}_t, f \rangle \langle m^{\delta}_s, g
     \rangle] = \frac{1}{2} \mathbb{E} [\langle \zeta_t, \mathcal{Q}^{1 /
     2}_{\delta} f \rangle \langle \zeta_s, \mathcal{Q}^{1 / 2}_{\delta} g
     \rangle] = \frac{1}{2} e^{- | t - s |} \langle \mathcal{Q}_{\delta} f, g
     \rangle, \]
  hence $\nu^{\delta}$ is the pushforward of $\mu$ by $T_{\delta}$.
  
  Denote by $(P^{\zeta}_t)_{t \geqslant 0}$ and $(P^{\delta}_t)_{t \geqslant
  0}$ the Markov semigroups of $\zeta$ and $m^{\delta}$, respectively. Then
  $(P^{\delta}_t)_{t \geqslant 0}$ is the pushforward of $(P^{\zeta}_t)_{t
  \geqslant 0}$ in the sense that
  \[ P^{\delta}_t \varphi (T_{\delta} n) =\mathbb{E} [\varphi (m^{\delta}_t
     (T_{\delta} n))] =\mathbb{E} [(\varphi \circ T_{\delta}) (\zeta_t (n))] =
     P^{\zeta}_t (\varphi \circ T_{\delta}) (n), \quad n \in H, t \geqslant 0,
  \]
  for every bounded measurable $\varphi : H \rightarrow \mathbb{R}$.
  Consequently, their Kolmogorov generators $\mathcal{M}^{\delta}$ and
  $\mathcal{L}$ satisfy the conjugation identity
  \begin{equation}
    (\mathcal{M}^{\delta} \varphi) \circ T_{\delta} =\mathcal{L} (\varphi
    \circ T_{\delta}) . \label{eq:gen33}
  \end{equation}
  Indeed,
  \[ (\mathcal{M}^{\delta} \varphi) (T_{\delta} n) = \lim_{t \rightarrow 0}
     \frac{P^{\delta}_t \varphi (T_{\delta} n) - \varphi (T_{\delta} n)}{t} =
     \lim_{t \rightarrow 0} \frac{P^{\zeta}_t (\varphi \circ T_{\delta}) (n) -
     (\varphi \circ T_{\delta}) (n)}{t} =\mathcal{L} (\varphi \circ
     T_{\delta}) (n), \quad n \in H. \]
  On mean-free functions, the inverses are related by
  \begin{equation}
    ((-\mathcal{M}^{\delta})^{- 1} \varphi) \circ T_{\delta} =
    (-\mathcal{L})^{- 1} (\varphi \circ T_{\delta}), \label{eq:gen33a}
  \end{equation}
  which can be seen as follows: if $\psi = (-\mathcal{M}^{\delta})^{- 1}
  \varphi$ then $-\mathcal{M}^{\delta} \psi = \varphi$ hence $-
  (\mathcal{M}^{\delta} \psi) \circ T_{\delta} = \varphi \circ T_{\delta}$ and
  by \eqref{eq:gen33} $\varphi \circ T_{\delta} = -\mathcal{L} (\psi \circ
  T_{\delta})$ so $\psi \circ T_{\delta} = (-\mathcal{L})^{- 1} (\varphi \circ
  T_{\delta})$ and \eqref{eq:gen33a} follows.
  
  In order to see how $(-\mathcal{M}^{\delta})^{- 1}$ acts on polynomial
  observables, fix $k \in \mathbb{N}$ and $G \in H^{\odot k}$, and set $F (n)
  \assign \langle n^{\otimes k}, G \rangle_{H^{\otimes k}}$, $n \in H$. Fix $t
  \geqslant 0$ and identify $\zeta_t$ with the associated isonormal process
  $\eta (f) = \langle \zeta_t, f \rangle$. Since $T_{\delta}$ is self-adjoint,
  \[ (F \circ T_{\delta}) (\zeta_t) = \langle \eta^{\otimes k},
     T_{\delta}^{\otimes k} G \rangle_{H^{\otimes k}} . \]
  Let $F^{\circ} \assign F -\mathbb{E}_{\nu^{\delta}} [F]$ be the centered
  observable. Since $\nu^{\delta}$ is the pushforward of $\mu$ by
  $T_{\delta}$,
  \[ [F \circ T_{\delta}]^{\circ,\mu} = F \circ T_{\delta} -\mathbb{E}_{\mu} [F
     \circ T_{\delta}] = F \circ T_{\delta} -\mathbb{E}_{\nu^{\delta}} [F] =
     F^{\circ} \circ T_{\delta} . \]
  Hence, combining the conjugation identity \eqref{eq:gen33a} with the formula
  \eqref{eq:formula}, we obtain
  \[ ((-\mathcal{M}^{\delta})^{- 1} F^{\circ}) (m^{\delta}_t) =
     (-\mathcal{L})^{- 1} [\langle \eta^{\otimes k}, T_{\delta}^{\otimes k} G
     \rangle_{H^{\otimes k}}]^{\circ,\mu} (\zeta_t) \]
  \begin{equation}
    \cong \sum_{\ell = 0}^{[(k - 1) / 2]} \langle \eta^{\otimes (k - 2 \ell)},
    \mathcal{C}_{\ell} (T_{\delta}^{\otimes k} G) \rangle_{H^{\otimes (k - 2
    \ell)}} . \label{eq:25}
  \end{equation}
  We now rewrite the right hand side in terms of $m^{\delta}_t$. Writing
  contractions in an orthonormal basis $(e_i)_{i \in \mathbb{N}}$ of $H$, one
  finds that each contraction inserts the tensor $\sum_{i \in \mathbb{N}}
  T_{\delta} e_i \otimes T_{\delta} e_i$, which represents the covariance
  $\mathbb{E}[m_t^{\delta} \otimes m_t^{\delta}] = \tfrac{1}{2}
  \mathcal{Q}_{\delta}$ viewed as an element of $H^{\odot 2}$. Accordingly, we
  define $\tilde{\mathcal{C}}_{\ell}^{\delta} G \in H^{\odot (k - 2 \ell)}$ by
  $\mathcal{C}_{\ell} (T_{\delta}^{\otimes k} G) = T_{\delta}^{\otimes (k - 2
  \ell)} (\tilde{\mathcal{C}}_{\ell}^{\delta} G)$, i.e.
  $\tilde{\mathcal{C}}_{\ell}^{\delta} G$ is obtained from $G$ by inserting
  $\ell$ copies of the covariance $\tfrac{1}{2} \mathcal{Q}_{\delta}$, see
  \eqref{eq:contr12}. The definition does not depend on the choice of the
  orthonormal basis. As a consequence, \eqref{eq:final} follows and the proof
  is complete.
\end{proof}

\subsection{Application to correctors -- motivating example}\label{s:appl}

The first two correctors involve only linear and quadratic observables and can
be treated directly. The first genuinely nontrivial case occurs for the cubic
observable defining the third-order corrector $\varphi_{(0, 0, 0)}$, which we
analyze in detail to illustrate the role of chaos projections and covariance
contractions.

The first corrector $\varphi_{(0)}$ is given by $\varphi_{(0)} (u, n) =
(-\mathcal{M}^{\delta})^{- 1} \langle n \cdummy \nabla u, h \rangle$. Since
$\langle n \cdummy \nabla u, h \rangle$ belongs to the first homogeneous
Wiener chaos, the inverse generator acts trivially, yielding $\varphi_{(0)}
(u, n) = \langle n \cdummy \nabla u, h \rangle$, in agreement with
Section~\ref{s:2}.

The second corrector $\varphi_{(0, 0)}$ involves a quadratic observable and
reads
\[ \varphi_{(0, 0)} (u, n) = (-\mathcal{M}^{\delta})^{- 1} [\langle n \cdummy
   \nabla \mathbb{P} (n \cdummy \nabla u), h \rangle]^{\circ} . \]
Since this observable lies in the second homogeneous chaos, the inverse
generator produces the factor $1 / 2$, yielding $\varphi_{(0, 0)} (u, n) =
\frac{1}{2} [\langle n \cdummy \nabla \mathbb{P} (n \cdummy \nabla u), h
\rangle]^{\circ}$, again consistent with the expressions obtained earlier,
after subtracting the expectation.

The first genuinely new phenomenon appears for the third-order corrector,
where the inverse generator acts on a cubic polynomial and nontrivial chaos
projections and contraction terms arise. The third-order corrector is given by
\[ \varphi_{(0, 0, 0)} (u, n) = \frac{1}{2} (-\mathcal{M}^{\delta})^{- 1}
   (\langle n \cdummy \nabla \mathbb{P} (n \cdummy \nabla \mathbb{P} (n
   \cdummy \nabla u)), h \rangle - \langle S_{\delta} (\mathbb{P} (n \cdummy
   \nabla u)), h \rangle) . \]
Indeed, starting from the definition \eqref{eq:aa1} we compute
\[ \varphi_{(0, 0, 0)} (u, n) = (-\mathcal{M}^{\delta})^{- 1} [\varphi_{(0,
   0)} (\mathbb{P} (n \cdummy \nabla u), n)]^{\circ} = \frac{1}{2}
   (-\mathcal{M}^{\delta})^{- 1} [\langle n \cdummy \nabla \mathbb{P} (n
   \cdummy \nabla v), h \rangle]^{\circ} |_{v =\mathbb{P} (n \cdummy \nabla
   u)} \nobracket \]
\[ = \frac{1}{2} (-\mathcal{M}^{\delta})^{- 1} (\langle n \cdummy \nabla
   \mathbb{P} (n \cdummy \nabla \mathbb{P} (n \cdummy \nabla u)), h \rangle -
   \langle S_{\delta} (\mathbb{P} (n \cdummy \nabla u)), h \rangle) . \]
We now estimate the third-order corrector using the general bound
\eqref{eq:final} derived in the previous subsection. Set $F (n) \assign
\langle n \cdummy \nabla \mathbb{P} (n \cdummy \nabla \mathbb{P} (n \cdummy
\nabla u)), h \rangle$, $n \in H$. This is a homogeneous polynomial observable
of degree $3$ in $n$, hence it can be written as
\[ F (n) = \langle n^{\otimes 3}, G \rangle_{H^{\otimes 3}} \]
for a deterministic symmetric kernel $G \in H^{\odot 3}$ depending linearly on
$u$ and $h$. More explicitly,
\[ F (n) = \int_{(\mathbb{T}^d)^3} n^{k_2} (x) \nabla^{k_2}_x \mathbb{P}^{k_1
   j_1} (x - y) n^{j_2} (y) \nabla^{j_2}_y \mathbb{P}^{j_1 i_1} (y - z)
   n^{i_2} (z) \nabla^{i_2}_z u^{i_1} (z) h^{k_1} (x) d x d y d z, \]
where we sum over repeated indices, $\nabla_x^i$ denotes the $i$-th component
of the gradient in $x$ and $\mathbb{P}^{i j}$ is the $(i, j)$-component of the
Leray projector kernel. From this formula one reads off the a priori
non-symmetric kernel $\tilde{G}$ given componentwise by
\[ \tilde{G}^{k_2 j_2 i_2} (x, y, z) \assign \nabla^{k_2}_x \mathbb{P}^{k_1
   j_1} (x - y) \nabla^{j_2}_y \mathbb{P}^{j_1 i_1} (y - z) \nabla^{i_2}_z
   u^{i_1} (z) h^{k_1} (x). \]
Since $n^{\otimes 3}$ (and likewise its Wick-ordered version) is a symmetric
tensor, only the totally symmetric part of the kernel contributes to the
pairing. We therefore replace $\tilde{G}$ by its symmetrization $G$ without
changing $F (n)$. To this end, let $S_3$ denote the permutation group of $\{
1, 2, 3 \}$ and define
\[ G^{i_1 i_2 i_3} (x_1, x_2, x_3) \assign \frac{1}{3!} \sum_{\sigma \in S_3}
   \tilde{G}^{i_{\sigma (1)} i_{\sigma (2)} i_{\sigma (3)}} (x_{\sigma (1)},
   x_{\sigma (2)}, x_{\sigma (3)}), \]
i.e. we apply the same permutation simultaneously to the spatial variables and
to the tensor indices. Then $\langle n^{\otimes 3}, \tilde{G}
\rangle_{H^{\otimes 3}} = \langle n^{\otimes 3}, G \rangle_{H^{\otimes 3}}$
and applying \eqref{eq:final} with $k = 3$ yields
\begin{equation}
  | (-\mathcal{M}^{\delta})^{- 1} F^{\circ} (m_t^{\delta}) | \lesssim
  | \langle (m_t^{\delta})^{\otimes 3}, G \rangle_{H^{\otimes 3}} | + |
  \langle m_t^{\delta}, \tilde{\mathcal{C}}_1^{\delta} G \rangle |,
  \label{eq:contr45}
\end{equation}
where $\tilde{\mathcal{C}}_1^{\delta}$ denotes the single contraction obtained
by inserting one copy of the covariance $\frac{1}{2} \mathcal{Q}_{\delta}$ in
two tensor slots.

Note that in \eqref{eq:contr12} we effectively defined the contraction
$\tilde{\mathcal{C}}^{\delta}_1$ as contracting the first two slots of a cubic
kernel. The crucial point is that this operator is applied to the symmetric
kernel $G$. As a consequence, $\tilde{\mathcal{C}}_1^{\delta} G$ coincides
with the average of the three possible pairwise contractions of the original
kernel $\tilde{G}$. In particular, up to the natural relabeling of the
remaining variable,
\[ \tilde{\mathcal{C}}_1^{\delta} G = \frac{1}{3} \sum_{\{ i, j \} \subset \{
   1, 2, 3 \}} \tilde{\mathcal{C}}_1^{\delta} (\tilde{G} \text{ with slots }
    (i, j) \text{ contracted}) . \]

The first term in \eqref{eq:contr45} is estimated by H{\"o}lder inequality as
\[ |\langle (m_t^{\delta})^{\otimes 3}, G \rangle_{H^{\otimes 3}} | \lesssim |
   \langle \mathbb{P} (m^{\delta}_t \cdummy \nabla u), m^{\delta}_t \cdummy
   \nabla \mathbb{P} (m^{\delta}_t \cdummy \nabla h) \rangle | \]
\[ \lesssim \| \mathbb{P} (m^{\delta}_t \cdummy \nabla u) \|_{L^2} \|
   m^{\delta}_t \|_{L^{\infty}} \| \nabla \mathbb{P} (m^{\delta}_t \cdummy
   \nabla h) \|_{L^2} \]
\begin{equation}
  \lesssim \| m^{\delta}_t \|^2_{L^{\infty}} \| \nabla u \|_{L^2} (\| \nabla
  m^{\delta}_t \|_{L^{\infty}} \| h \|_{H^1} + \| m^{\delta}_t \|_{L^{\infty}}
  \| h \|_{H^2}) . \label{eq:2323}
\end{equation}
It remains to estimate the contraction term $\langle m_t^{\delta},
\tilde{\mathcal{C}}_1^{\delta} G \rangle$. By definition,
$\tilde{\mathcal{C}}_1^{\delta} G$ is obtained from the kernel $G$ by pairing
two of the three noise slots against the covariance $\frac{1}{2}
\mathcal{Q}_{\delta}$. Writing $\mathcal{Q}_{\delta}$ in physical variables as
convolution with the matrix-valued kernel $Q_{\delta}$, this yields the
explicit representation
\[ (\tilde{\mathcal{C}}^{\delta}_1 G)^i (z) = \frac{1}{2}
   \int_{(\mathbb{T}^d)^2} Q^{j k}_{\delta} (x - y) G^{j k i} (x, y, z) d x d
   y. \]
Here $i, j, k \in \{1, ..., d\}$ are vector-component indices; $Q_{\delta}^{j
k}$ is the $(j, k)$-entry of the matrix $Q_{\delta}$; and we sum over $j, k$.
In other words, it holds
\[ \langle m_t^{\delta}, \tilde{\mathcal{C}}_1^{\delta} G \rangle =
   \frac{1}{3} (\mathbb{E}_{\nu^{\delta}} [\langle \mathbb{P} (n \cdummy
   \nabla u), n \cdummy \nabla \mathbb{P} (m^{\delta}_t \cdummy \nabla h)
   \rangle] \nobracket \]
\begin{equation}
  \nobracket +\mathbb{E}_{\nu^{\delta}} [\langle \mathbb{P} (m^{\delta}_t
  \cdummy \nabla u), n \cdummy \nabla \mathbb{P} (n \cdummy \nabla h) \rangle]
  +\mathbb{E}_{\nu^{\delta}} [\langle \mathbb{P} (n \cdummy \nabla u),
  m_t^{\delta} \cdummy \nabla \mathbb{P} (n \cdummy \nabla h) \rangle]),
  \label{eq:I13}
\end{equation}
where the expectation is with respect to the variable $n$. Here, the key is to
estimate first the inner product inside the expectation using continuity of
the Leray projection on $L^p (\mathbb{T}^d)$, $p \in (1, \infty)$, and only in
the end to estimate the expectation. This way, the bound of the contracted
term is similar to \eqref{eq:2323}, but contains only one copy of
$m^{\delta}_t$. Specifically, for the first term we bound as
\[ | \mathbb{E}_{\nu^{\delta}} [\langle \mathbb{P} (n \cdummy \nabla u), n
   \cdummy \nabla \mathbb{P} (m^{\delta}_t \cdummy \nabla h) \rangle] |
   \lesssim \mathbb{E}_{\nu^{\delta}} [\| \mathbb{P} (n \cdummy \nabla u)
   \|_{L^2} \| n \|_{L^{\infty}} \| \nabla \mathbb{P} (m^{\delta}_t \cdummy
   \nabla h) \|_{L^2}] \]
\[ \lesssim \| \nabla u \|_{L^2} \mathbb{E}_{\nu^{\delta}} [\| n
   \|_{L^{\infty}}^2] (\| \nabla m^{\delta}_t \|_{L^{\infty}} \| h \|_{H^1} +
   \| m^{\delta}_t \|_{L^{\infty}} \| h \|_{H^2}) \]
\[ \lesssim \| \nabla u \|_{L^2} (\| \nabla m^{\delta}_t \|_{L^{\infty}} \| h
   \|_{H^1} + \| m^{\delta}_t \|_{L^{\infty}} \| h \|_{H^2}), \]
and similarly for the other terms
\[ | \mathbb{E}_{\nu^{\delta}} [\langle \mathbb{P} (m^{\delta}_t \cdummy
   \nabla u), n \cdummy \nabla \mathbb{P} (n \cdummy \nabla h) \rangle] | + |
   \mathbb{E}_{\nu^{\delta}} [\langle \mathbb{P} (n \cdummy \nabla u),
   m_t^{\delta} \cdummy \nabla \mathbb{P} (n \cdummy \nabla h) \rangle] | \]
\[ \lesssim \| m^{\delta}_t \|_{L^{\infty}} \| \nabla u \|_{L^2}
   (\mathbb{E}_{\nu^{\delta}} [\| n \|_{L^{\infty}} \| \nabla n
   \|_{L^{\infty}}] \| h \|_{H^1} +\mathbb{E}_{\nu^{\delta}} [\| n
   \|_{L^{\infty}}^2] \| h \|_{H^2}) \]
\[ \lesssim \| m^{\delta}_t \|_{L^{\infty}} \| \nabla u \|_{L^2} (\delta^{- 1}
   \| h \|_{H^1} + \| h \|_{H^2}) . \]
In particular, whenever the derivative hits $m^{\delta}$ or its covariance,
the resulting factor $\delta^{- 1}$ can be traded against reduced regularity
requirement on the test function $h$. This will be exploited systematically
for the higher order correctors.

\begin{remark}
  One might attempt to estimate the contracted kernels
  $\tilde{\mathcal{C}}_1^{\delta} G$ directly in physical variables, leading
  to integral operators involving products of the Leray kernel and
  $Q_{\delta}$ or its derivatives. While such operators can sometimes be
  treated using Calder{\'o}n--Zygmund theory, this approach fails for the last
  term in \eqref{eq:I13}: the uncontracted field $m^{\delta}_t$ appears
  between two singular operators, yielding a genuinely two-variable kernel
  outside the Calder{\'o}n--Zygmund framework. We therefore estimate inner
  products prior to taking expectations, using probabilistic bounds rather
  than kernel-level analysis; this extends uniformly to all higher-order
  correctors, with the relevant estimates carried out as needed below.
\end{remark}

\subsection{Structural properties of the corrector hierarchy}\label{s:gen}

As illustrated in Section~\ref{s:appl}, applying the inverse
Ornstein--Uhlenbeck generator to a polynomial observable yields a finite sum
of contraction terms, each obtained by pairing tensor slots with copies of the
covariance $\frac{1}{2} \mathcal{Q}_{\delta}$. Each contraction reduces the
polynomial degree by two and introduces no additional singularity beyond that
already present in the covariance itself.

By Proposition~\ref{p:chaos}, the explicit Wiener chaos analysis of
Section~\ref{s:appl} need not be repeated for each corrector: every
$\varphi_{\sigma}$ is a polynomial functional of $n$ of degree $M$, and
quantitative bounds follow by distributing derivatives via the Leibniz rule
together with continuity of the Leray projector and moment estimates for
$m^{\delta}$. The resulting estimates are recorded in the following lemma.

\begin{lemma}
  \label{l:38}For every $\sigma \in \Sigma \setminus \{ (0) \}$ it holds
  \[ | \varphi_{\sigma} (u, n) | \lesssim \sum_{j = 0, \ldots, M + 2 L - 2}
     \left( 1 + \sum_{k_1 + \cdots + k_M = j} \prod_{i = 1, \ldots, M} \| n
     \|_{W^{k_i, \infty}} \right) \| \nabla u \|_{L^2} \| \nabla h \|_{H^{M +
     2 L - 2 - j}} . \]
\end{lemma}

\begin{proof}
  We distribute derivatives via the Leibniz rule and apply H{\"o}lder's
  inequality, as in the explicit computation of Section~\ref{s:appl}.
\end{proof}

\subsection{Structure of the expectations}\label{s:exp}

Although the mean-free parts of the magenta terms in \eqref{eq:gen} are
cancelled by construction, certain expectations may persist. These
contributions are responsible for the macroscopic correction introduced later
in the fluctuation analysis. The~following consequence of Lemma~\ref{l:38}
identifies the leading order part of these expectations as a second order
divergence-form operator, which we call the effective operator at order
$\sigma$ and denote by $S_{\sigma}$. Note that $S_{(0)} = S_{\delta}$ is the
enhanced diffusion operator of Section~\ref{s:quali}, so the family
$(S_{\sigma})_{\sigma \in \Sigma}$ generalizes the enhanced diffusion to all
orders of the corrector hierarchy.

\begin{proposition}
  \label{c:c18}Let $\sigma \in \Sigma$. Then $[\varphi_{\sigma} (\Delta u,
  n)]^{\bullet} = 0$. If $M$ is even then $[\varphi_{\sigma} (\mathbb{P} (n
  \cdot \nabla u), n)]^{\bullet} = 0$. If $M$ is odd and $\sigma \neq  (0)
  $ then there exist operators $S_{\sigma} : H^1 \rightarrow H^{- 1}$ and
  $R_{\sigma} : H^1 \rightarrow H^{- M - 2 L}$ such that
  \[ [\varphi_{\sigma} (\mathbb{P} (n \cdot \nabla u), n)]^{\bullet} =
     S_{\sigma} u + R_{\sigma} u, \]
  and
  \[ | \langle S_{\sigma} u, h \rangle | \lesssim \delta^{- (M + 2 L - 1)} \|
     \nabla u \|_{L^2} \| \nabla h \|_{L^2}, \]
  \[ | \langle R_{\sigma} u, h \rangle | \lesssim \| \nabla u \|_{L^2} \sum_{j
     = 1, \ldots, M + 2 L - 1} \delta^{- (M + 2 L - 1 - j)} \| \nabla h
     \|_{H^j} . \]
\end{proposition}

\begin{proof}
  By definition of the correctors \eqref{eq:aa1}, \eqref{eq:aa2} and since the
  inverse Ornstein--Uhlenbeck generator acts diagonally on the Wiener chaoses,
  particularly by \eqref{eq:final}, we directly obtain $[\varphi_{\sigma}
  (\Delta u, n)]^{\bullet} = 0$ and $[\varphi_{\sigma} (\mathbb{P} (n \cdot
  \nabla u), n)]^{\bullet} = 0$ for $M$ even.
  
  The decomposition of $[\varphi_{\sigma} (\mathbb{P} (n \cdot \nabla u),
  n)]^{\bullet}$ for $M$ odd and the estimates follow by the arguments from
  Lemma~\ref{l:38}. Specifically, $S_{\sigma}$ collects all the terms with the
  highest number of derivatives split among the copies of $n$, costing us
  negative powers of $\delta$, $R_{\sigma}$ collects the terms where at least
  one of those derivatives is applied to the test function $h$, which lowers
  the power of ${\delta^{- 1}} $ but requires more regularity from $h$.
\end{proof}

It will be seen in Section~\ref{s:setup} that the expectations treated in
Proposition~\ref{c:c18} appear in our expansion as
\[ \varepsilon^{M + 2 L - 1} [\varphi_{\sigma} (\mathbb{P} (n \cdot \nabla u),
   n)]^{\bullet}, \qquad \varepsilon^{M + 2 L} [\varphi_{\sigma} (\Delta u,
   n)]^{\bullet} . \]
Hence, Proposition~\ref{c:c18} identifies the leading order part $S_{\sigma}$
which does not vanish after the division by $\delta$, whereas the part
$R_{\sigma}$ can be divided by $\delta$ provided we estimate in a function
space of sufficiently low regularity. Our goal is to include the operators
$\varepsilon^{M + 2 L - 1} S_{\sigma}$ as perturbations of the generator of
the semigroup.

\section{Decomposition of the fluctuation dynamics}\label{s:setup}

Since the limit is unique and deterministic in $d = 2$, convergence in law
from Theorem~\ref{thm:1} implies that the entire sequence $u^{\varepsilon,
\delta}$ converges in probability on the original probability space to the
deterministic limit $u$. The fluctuation analysis may therefore be carried out
on the original probability space throughout. While a number of arguments in
the following sections hold for general $d$, the conclusion of
Theorem~\ref{thm:2} is restricted to $d = 2$, and we indicate explicitly where
this restriction is used.

\tmtextbf{Fluctuation variable. }We introduce the rescaled fluctuation
variable
\[ z^{\varepsilon, \delta} \assign \frac{u^{\varepsilon, \delta} -
   v^{\varepsilon, \delta} - u}{\delta^{d / 2}}, \]
where $u^{\varepsilon, \delta}$ solves \eqref{eq:u}, $u$ is the deterministic
limit from Theorem~\ref{thm:1} and $v^{\varepsilon, \delta}$ is a macroscopic
correction, specified below, accounting for terms that vanish more slowly than
$\delta^{d / 2}$ and therefore persist after the rescaling but do not
contribute to the limiting fluctuation. The limit fluctuation variable is
denoted by $z$.

\tmtextbf{Application of It{\^o}'s formula and decomposition. }We apply
It{\^o}'s formula to
\begin{equation}
  \varphi (u^{\varepsilon, \delta}_t) + \sum_{\sigma \in \Sigma}
  \varepsilon^{M + 2 L} \varphi_{\sigma} (u^{\varepsilon, \delta}_t,
  m^{\varepsilon, \delta}_t), \label{eq:sum}
\end{equation}
where the correctors $\varphi_{\sigma}$ are defined recursively in
Section~\ref{s:high}. Each corrector $\varphi_{\sigma}$ is linear in the test
function $h$ and may be identified with a distribution-valued functional of
$(u, n)$; in the sequel we systematically work with this identification,
writing $\varphi_{\sigma} (u, n)$ for the corresponding distribution rather
than the scalar $\varphi_{\sigma} (u, n ; h)$. In particular, $\varphi
(u^{\varepsilon, \delta})$ is identified with $u^{\varepsilon, \delta}$
itself.

Recall that in the color coding of the generator from \eqref{eq:gen}, the
correctors were designed in a way that the blue term associated to each
corrector $\varphi_{\sigma}$ cancels the mean-free part of one of the magenta
terms associated to its predecessor. To treat the persistent expectations, we
apply Proposition~\ref{c:c18} to write $[\varphi_{\sigma} (\mathbb{P} (n
\cdummy \nabla u^{\varepsilon, \delta}), n)]^{\bullet} = S_{\sigma}
u^{\varepsilon, \delta} + R_{\sigma} u^{\varepsilon, \delta}$. The operator
$S_{\sigma}$ becomes part of the generator of the semigroup whereas
$R_{\sigma}$ remains as error. Therefore, the generator
$\mathcal{L}^{\varepsilon, \delta}$ from \eqref{eq:gen} acting on the sum
\eqref{eq:sum} produces
\[ d \left[ u^{\varepsilon, \delta}_t + \sum_{\sigma \in \Sigma}
   \varepsilon^{M + 2 L} \varphi_{\sigma} (u^{\varepsilon, \delta},
   m^{\varepsilon, \delta}) \right] \]
\[ = \left( (1 + \nu) \mathbb{P} \Delta + \sum_{\{ \bar{\sigma} \in \Sigma ; M
   \tmop{odd}, (\bar{\sigma}, 0) \in \Sigma \} \setminus \{ (0) \}}
   \varepsilon^{M_{\bar{\sigma}} + 2 L_{\bar{\sigma}} - 1} S_{\bar{\sigma}}
   \right) \left[ u^{\varepsilon, \delta} + \sum_{\sigma \in \Sigma}
   \varepsilon^{M + 2 L} \varphi_{\sigma} (u^{\varepsilon, \delta},
   m^{\varepsilon, \delta}) \right] d t \]
\[ - \left( (1 + \nu) \mathbb{P} \Delta + \sum_{\{ \bar{\sigma} \in \Sigma ; M
   \tmop{odd}, (\bar{\sigma}, 0) \in \Sigma \} \setminus \{ (0) \}}
   \varepsilon^{M_{\bar{\sigma}} + 2 L_{\bar{\sigma}} - 1} S_{\bar{\sigma}}
   \right) \left[ \sum_{\sigma \in \Sigma} \varepsilon^{M + 2 L}
   \varphi_{\sigma} (u^{\varepsilon, \delta}, m^{\varepsilon, \delta}) \right]
   d t \]
\[ -\mathbb{P} (u^{\varepsilon, \delta} \cdummy \nabla u^{\varepsilon,
   \delta}) d t \]
\[ - \sum_{\sigma \in \Sigma} \varepsilon^{M + 2 L} \varphi_{\sigma}
   (\mathbb{P} (u^{\varepsilon, \delta} \cdummy \nabla u^{\varepsilon,
   \delta}), m^{\varepsilon, \delta}) d t \]
\[ +\mathbb{P} [n \cdummy \nabla \mathbb{P} (n \cdummy \nabla u^{\varepsilon,
   \delta})]^{\bullet} d t - \nu \mathbb{P} \Delta u^{\varepsilon, \delta} d t
\]
\[ + \sum_{\{ \sigma \in \Sigma ; M \tmop{odd}, (\sigma, 0) \in \Sigma \}
   \setminus \{ (0) \}} \varepsilon^{M + 2 L - 1} R_{\sigma} u^{\varepsilon,
   \delta} d t \]
\[ +\mathbb{P} [\mathcal{Q}^{1 / 2}_{\delta} d W \cdummy \nabla
   u^{\varepsilon, \delta}] \]
\[ + \sum_{\sigma \in \Sigma \setminus \{ (0) \}} \varepsilon^{M + 2 L - 1}
   \langle \mathcal{Q}^{1 / 2}_{\delta} d W, D_n \varphi_{\sigma}
   (u^{\varepsilon, \delta}, m^{\varepsilon, \delta}) \rangle \]
\[ + \sum_{\{ \sigma \in \Sigma ; (\sigma, 0) \notin \Sigma \}} \varepsilon^{M
   + 2 L - 1} \varphi_{\sigma} (\mathbb{P} (m^{\varepsilon, \delta} \cdummy
   \nabla u^{\varepsilon, \delta}), m^{\varepsilon, \delta}) d t + \sum_{\{
   \sigma \in \Sigma ; (\sigma, 1) \notin \Sigma \}} \varepsilon^{M + 2 L}
   \varphi_{\sigma} (\Delta u^{\varepsilon, \delta}, m^{\varepsilon, \delta})
   d t \]
\[ \backassign J_1 + \cdots + J_9 \]
where we recognize, in the order of appearance, the linear diffusion $J_1$,
the (perturbed) Laplacian remainder of the correctors $J_2$, the approximate
convective term $J_3$, the convective term associated to correctors $J_4$, the
enhanced diffusion correction $J_5$, the remaining expectations of the
cancelled magenta terms $J_6$, the stochastic integral associated to the first
corrector $J_7$, the stochastic integrals associated to higher correctors
$J_8$ and the magenta terms of the last generation $J_9$. We also wrote $M$
and $L$ for $M_{\sigma}$ and $L_{\sigma}$ respectively, whereas for the
parameters associated to $\bar{\sigma}$ we used $M_{\bar{\sigma}}$ and
$L_{\bar{\sigma}}$.

Note that $J_2$ arises from adding and subtracting the perturbed Laplacian of
the correctors, incorporating the correctors into the linear diffusion
operator $J_1$ in preparation for the mild formulation described below.
Similarly, the second term in $J_5$ arises from adding and subtracting the
enhanced viscosity part $\nu \mathbb{P} \Delta u^{\varepsilon, \delta}$.

\tmtextbf{Splitting into low and high Fourier modes. }We introduce a spatial
frequency cut-off parameter $R > 1$ and decompose $\tmop{Id} = P_{\leqslant R}
+ P_{> R}$ into low and high Fourier modes. This localization provides a
quantitative trade-off: derivatives acting on $m^{\delta}$ or
$\mathcal{Q}_{\delta}$ produce factors of $\delta^{- 1}$, whereas derivatives
acting on the test function generate powers of $R$. Balancing these
contributions, together with the $L^{\infty}$-integrability of the noise
established in Section~\ref{s:OU}, allows the truncated corrector hierarchy to
be controlled uniformly across all orders up to $N$. As we verify in
Section~\ref{s:removal}, the choice $R = \delta^{- 1}$ is sufficient
throughout.

The high-frequency component $P_{> R} (z^{\varepsilon, \delta} - z)$ is
controlled directly by the uniform a priori bounds available for
$u^{\varepsilon, \delta}$, $v^{\varepsilon, \delta}$, $u$ and $z$. This yields
a factor $\delta^{- d / 2}$, which must be compensated by a suitable negative
power of $R$ obtained from the gain of regularity in the high-frequency
regime. The low-frequency component $P_{\leqslant R} (z^{\varepsilon, \delta}
- z)$ is analyzed using It{\^o}'s formula applied to the truncated corrector
hierarchy described above.

\tmtextbf{Remaining expectations.} It was discussed in Section~\ref{s:exp}
that the expectations $\varepsilon^{M + 2 L - 1} [\varphi_{\sigma} (\mathbb{P}
(n \cdummy \nabla u^{\varepsilon, \delta}), n)]^{\bullet}$ can be rewritten as
the sum of $\varepsilon^{M + 2 L - 1} S_{\sigma} u^{\varepsilon, \delta}$ and
$\varepsilon^{M + 2 L - 1} R_{\sigma} u^{\varepsilon, \delta}$. Here the
second term belongs to $J_6$ and is treated as a lower order remainder in
Section~\ref{s:exp2}. This approach is not possible for the leading order part
$\varepsilon^{M + 2 L - 1} S_{\sigma} u^{\varepsilon, \delta}$ which does not
vanish after the division by $\delta$. Therefore, the idea is to introduce
counterterms into the definition of the macroscopic correction
$v^{\varepsilon, \delta}$. However, we do not want to put directly
$\varepsilon^{M + 2 L - 1} S_{\sigma} u^{\varepsilon, \delta}$ into
$v^{\varepsilon, \delta}$ as this would make $v^{\varepsilon, \delta}$ random.
Our goal is to include $\varepsilon^{M + 2 L - 1} S_{\sigma} u$ for the limit
solution $u$ into the definition of $v^{\varepsilon, \delta}$. In order to
achieve this, we include the operator $\varepsilon^{M + 2 L - 1} S_{\sigma}$
as perturbation of the Stokes operator $(1 + \nu) \mathbb{P} \Delta$ into the
semigroup.

\tmtextbf{Mild formulation. }In order to absorb the convective term, we
include an additional exponential factor $e^{- \lambda t}$ into the semigroup.
In particular, the low-frequency component $P_{\leqslant R} e^{- \lambda
\cdummy} (z^{\varepsilon, \delta} - z)$ is estimated using the mild
formulation via the semigroup $(S_t)_{t \geqslant0}$ generated by
\begin{equation}
  - \lambda \tmop{Id} + (1 + \nu) \mathbb{P} \Delta + \sum_{\{ \sigma \in
  \Sigma ; M \tmop{odd}, (\sigma, 0) \in \Sigma \} \setminus \{ (0) \}}
  \varepsilon^{M + 2 L - 1} S_{\sigma} . \label{eq:A}
\end{equation}
The construction of the semigroup together with the necessary smoothing
estimates is provided in Section~\ref{s:semigroup}.

\tmtextbf{Criticality of the convective term. }The approximate convective term
$J_3$ shall be coupled with the convective term of the limit equation for $u$.
Additionally, we require counterterms in the equation for $v^{\varepsilon,
\delta}$ as well as in the equation for $z$. Specifically, writing
\[ \frac{u^{\varepsilon, \delta} \cdummy \nabla u^{\varepsilon, \delta} - u
   \cdummy \nabla u}{\delta^{d / 2}} = \frac{u^{\varepsilon, \delta} -
   u}{\delta^{d / 2}} \cdummy \nabla u + u^{\varepsilon, \delta} \cdummy
   \nabla \frac{u^{\varepsilon, \delta} - u}{\delta^{d / 2}} \]
and including the equations for $\delta^{- d / 2} v^{\varepsilon, \delta}$ and
$z$, namely, subtracting the terms
\[ \frac{v^{\varepsilon, \delta} \cdummy \nabla u + u \cdummy \nabla
   v^{\varepsilon, \delta} + v^{\varepsilon, \delta} \cdummy \nabla
   v^{\varepsilon, \delta}}{\delta^{d / 2}} \quad \infixand \quad z \cdummy
   \nabla u + u \cdummy \nabla z, \]
this expression becomes
\begin{equation}
  \left( \frac{u^{\varepsilon, \delta} - v^{\varepsilon, \delta} -
  u}{\delta^{d / 2}} - z \right) \cdummy \nabla u + u^{\varepsilon, \delta}
  \cdummy \nabla \left( \frac{u^{\varepsilon, \delta} - v^{\varepsilon,
  \delta} - u}{\delta^{d / 2}} - z \right) + \left( \frac{u^{\varepsilon,
  \delta} - v^{\varepsilon, \delta} - u}{\delta^{d / 2}} - z \right) \cdummy
  \nabla v^{\varepsilon, \delta} \label{eq:plo1}
\end{equation}
\begin{equation}
  + (u^{\varepsilon, \delta} - u) \cdummy \nabla z + z \cdummy \nabla
  v^{\varepsilon, \delta} . \label{eq:42}
\end{equation}
The terms in \eqref{eq:plo1} are the most delicate ones, as they contain the
quantity $z^{\varepsilon, \delta} - z$ we aim to control. These contributions
are critical for two independent -- and structurally unavoidable -- reasons.
First, $z^{\varepsilon, \delta} - z$ can only be estimated in a negative Besov
space, reflecting the negative regularity of the fluctuation limit $z$, which
is intrinsic and dictated by the stochastic integral in the limiting equation.
Second, the velocity fields $u$, $u^{\varepsilon, \delta}$ and
$v^{\varepsilon, \delta}$ appearing in these terms are controlled solely by
the energy estimate, and no additional spatial regularity or time
integrability is available. As a consequence, the associated paraproduct
interactions occur at the borderline level dictated by these bounds, and the
resulting time convolution kernels are only barely integrable. The smoothing
provided by the semigroup is therefore exactly critical: it suffices to close
the estimate, but it yields no additional gain in time. Handling this
borderline structure requires a dedicated endpoint convective estimate, which
we establish in Section~\ref{s:convective} as a self-contained result.

As already noted in the mild formulation above, the exponential factor $e^{-
\lambda t}$ is incorporated into the semigroup generator, so the endpoint
estimate of Section~\ref{s:convective} is applied to $e^{- \lambda \cdot}
(z^{\varepsilon, \delta} - z)$. The resulting contribution of the first three
terms is then bounded by $\| e^{- \lambda \cdot} (z^{\varepsilon,
\delta} - z) \|_{L^2 H^{- \beta}}$ times a factor that vanishes: by the
convergence $u^{\varepsilon, \delta} \rightarrow u$ in probability,
approximation of $u$ by a smooth Galerkin truncation $\bar{u}$, and the
exponential decay of the semigroup applied to $\bar{u}$ via
Proposition~\ref{p:22} with $\lambda$ chosen sufficiently large. This allows
all three terms to be absorbed into the left-hand side; see
Section~\ref{s:concl}. The last two terms are lower order and vanish as
$\varepsilon, \delta \rightarrow 0$.

\tmtextbf{Limit stochastic integral. }The only stochastic integral which
survives the limit after division by $\delta^{d / 2}$ is the one associated to
the first corrector, i.e. $J_7$. In Section~\ref{s:stoch} we identify its
limit as $\chi d W \cdummy \nabla u$, which is the stochastic integral
appearing in the equation of $z$. Here $\chi = (\mathcal{F}_{\mathbb{R}^2} K)
(0)$ denotes the noise intensity. Since the corresponding stochastic
convolution in the mild formulation is well defined only in negative Sobolev
spaces, this term is precisely what restricts the conclusion of
Theorem~\ref{thm:2} to $L^2 (0, T ; H^{- \beta})$ for any $\beta > 0$.

\tmtextbf{Limit fluctuation equation.} The fluctuation limit $z$ solves the
stochastic linearized Navier--Stokes equations
\[ d z + [z \cdummy \nabla u + u \cdummy \nabla z+ \nabla p_z ] d t = (1 +
   \nu) \Delta z d t + \chi d W \cdummy \nabla u, \quad \tmop{div} z = 0,
   \quad z (0) = 0, \]
with a cylindrical Wiener process $W$ on $H$, enhanced viscosity $\nu$ and
$\chi = (\mathcal{F}_{\mathbb{R}^2} K) (0)$, $u$ being the limit solution to
\eqref{eq:ulim} and with the associated pressure $p_z$. Well-posedness for
this equation is presented in Section~\ref{s:z}.

\tmtextbf{Additional errors.} Including $\sum_{\{ \sigma \in \Sigma ; M
\tmop{odd}, (\sigma, 0) \in \Sigma \} \setminus \{ (0) \}} \varepsilon^{M + 2
L - 1} S_{\sigma}$ into the generator of the semigroup requires the additional
terms
\[ \sum_{\{ \sigma \in \Sigma ; M \tmop{odd}, (\sigma, 0) \in \Sigma \}
   \setminus \{ (0) \}} \varepsilon^{M + 2 L - 1} S_{\sigma} u, \qquad
   \sum_{\{ \sigma \in \Sigma ; M \tmop{odd}, (\sigma, 0) \in \Sigma \}
   \setminus \{ (0) \}} \varepsilon^{M + 2 L - 1} S_{\sigma} z. \]
The first one does not vanish after the division by $\delta^{d / 2}$ and must
therefore be included in $v^{\varepsilon, \delta}$. The second does not have
to be divided by $\delta^{d / 2}$ and therefore vanishes as we show in
Section~\ref{s:addz}.

\tmtextbf{Macroscopic corrections $v^{\varepsilon, \delta}$.} Motivated by the
preceding decomposition, we define $v^{\varepsilon, \delta}$ as the solution
to
\[ \partial_t v^{\varepsilon, \delta} + v^{\varepsilon, \delta} \cdummy \nabla
   u + u \cdummy \nabla v^{\varepsilon, \delta} + v^{\varepsilon, \delta}
   \cdummy \nabla v^{\varepsilon, \delta} + \nabla p_{v^{\varepsilon, \delta}}
\]
\[ = \left( (1 + \nu) \Delta + \sum_{\{ \sigma \in \Sigma ; M \tmop{odd},
   (\sigma, 0) \in \Sigma \} \setminus \{ (0) \}} \varepsilon^{M + 2 L - 1}
   S_{\sigma} \right) v^{\varepsilon, \delta} + \sum_{\{ \sigma \in \Sigma ; M
   \tmop{odd}, (\sigma, 0) \in \Sigma \} \setminus \{ (0) \}} \varepsilon^{M +
   2 L - 1} S_{\sigma} u, \]
\[ \tmop{div} v^{\varepsilon, \delta} = 0, \quad v^{\varepsilon, \delta} (0) =
   0, \]
where $p_{v^{\varepsilon, \delta}}$ denotes the associated pressure. We
provide a well-posedness result for this equation in Section~\ref{s:v}.

\tmtextbf{Enhanced diffusion error.} The term $J_5$ accounts for the
discrepancy between the plain viscosity in the equation for $u^{\varepsilon,
\delta}$ and the enhanced viscosity $1 + \nu$ in the semigroup, and is treated
quantitatively in Section~\ref{s:quanti}.

\tmtextbf{Lower order remainder terms. }The terms $J_2$, $J_4$, $J_6$, $J_8$,
and $J_9$, together with the contributions of the corrector sum $\sum_{\sigma
\in \Sigma} \varepsilon^{M + 2 L} \varphi_{\sigma} (u^{\varepsilon, \delta},
m^{\varepsilon, \delta})$ evaluated at times $t$ and $0$ arising from the mild
formulation, vanish as $\varepsilon, \delta \rightarrow 0$ after division by
$\delta^{d / 2}$. Their treatment is carried out in Section~\ref{s:error},
relying on the polynomial structure of the corrector hierarchy, derivative
counting, frequency localization and maximal regularity.

\tmcolor{black}{\tmtextbf{Summary.} We observed that the dynamics governing the
difference $z^{\varepsilon, \delta} - z$ reads
\[ d \left[ z^{\varepsilon, \delta} - z + \delta^{- d / 2} \sum_{\sigma \in
   \Sigma} \varepsilon^{M + 2 L} \varphi_{\sigma} (u^{\varepsilon, \delta},
   m^{\varepsilon, \delta}) \right] \]
\[ = \left( (1 + \nu) \mathbb{P} \Delta + \sum_{\{ \bar{\sigma} \in \Sigma ; M
   \tmop{odd}, (\bar{\sigma}, 0) \in \Sigma \} \setminus \{ (0) \}}
   \varepsilon^{M_{\bar{\sigma}} + 2 L_{\bar{\sigma}} - 1} S_{\bar{\sigma}}
   \right) \left[ z^{\varepsilon, \delta} - z + \delta^{- d / 2} \sum_{\sigma
   \in \Sigma} \varepsilon^{M + 2 L} \varphi_{\sigma} (u^{\varepsilon,
   \delta}, m^{\varepsilon, \delta}) \right] \]
\[ +\mathbb{P} [(z^{\varepsilon, \delta} - z) \cdummy \nabla u +
   u^{\varepsilon, \delta} \cdummy \nabla (z^{\varepsilon, \delta} - z) +
   (z^{\varepsilon, \delta} - z) \cdummy \nabla v^{\varepsilon, \delta}] \]
\[ +\mathbb{P} [(u^{\varepsilon, \delta} - u) \cdummy \nabla z + z \cdummy
   \nabla v^{\varepsilon, \delta}] \]
\[ +\mathbb{P} [\delta^{- d / 2} \mathcal{Q}^{1 / 2}_{\delta} d W \cdummy
   \nabla u^{\varepsilon, \delta} - \chi d W \cdummy \nabla u] \]
\[ + \delta^{- d / 2} (J_2 + J_4 + J_5 + J_6 + J_8 + J_9) \]
\[ + \sum_{\{ \bar{\sigma} \in \Sigma ; M \tmop{odd}, (\bar{\sigma}, 0) \in
   \Sigma \} \setminus \{ (0) \}} \varepsilon^{M_{\bar{\sigma}} + 2
   L_{\bar{\sigma}} - 1} S_{\bar{\sigma}} z. \]
Hence, in the mild formulation using the semigroup generated by \eqref{eq:A},
since $z^{\varepsilon, \delta}_0 - z_0 = 0$,
\[ e^{- \lambda t} (z^{\varepsilon, \delta}_t - z_t) = e^{- \lambda t} \left[
   z_t^{\varepsilon, \delta} - z_t + \delta^{- d / 2} \sum_{\sigma \in \Sigma}
   \varepsilon^{M + 2 L} \varphi_{\sigma} (u_t^{\varepsilon, \delta},
   m_t^{\varepsilon, \delta}) \right] \]
\[ - e^{- \lambda t} \delta^{- d / 2} \sum_{\sigma \in \Sigma} \varepsilon^{M
   + 2 L} \varphi_{\sigma} (u_t^{\varepsilon, \delta}, m_t^{\varepsilon,
   \delta}) \]
\[ = \delta^{- d / 2} \sum_{\sigma \in \Sigma} \varepsilon^{M + 2 L} S_t
   \varphi_{\sigma} (u_0^{\varepsilon, \delta}, m_0^{\varepsilon, \delta}) \]
\[ + \int_0^t S_{t - s} \mathbb{P} [e^{- \lambda s} (z_s^{\varepsilon, \delta}
   - z_s) \cdummy \nabla u_s + u_s^{\varepsilon, \delta} \cdummy \nabla e^{-
   \lambda s} (z_s^{\varepsilon, \delta} - z_s) + e^{- \lambda s}
   (z_s^{\varepsilon, \delta} - z_s) \cdummy \nabla v_s^{\varepsilon, \delta}]
   d s \]
\[ + \int_0^t S_{t - s} \mathbb{P} [e^{- \lambda s} (u_s^{\varepsilon, \delta}
   - u_s) \cdummy \nabla z_s + e^{- \lambda s} z_s \cdummy \nabla
   v_s^{\varepsilon, \delta}] d s \]
\[ + \int_0^t S_{t - s} \mathbb{P} [\delta^{- d / 2} \mathcal{Q}^{1 /
   2}_{\delta} d W_s \cdummy \nabla e^{- \lambda s} u_s^{\varepsilon, \delta}
   - \chi d W_s \cdummy \nabla e^{- \lambda s} u_s] \]
\[ + \delta^{- d / 2} \int_0^t S_{t - s} e^{- \lambda s} (J_2 + J_4 + J_5 +
   J_6 + J_8 + J_9) \]
\[ + \sum_{\{ \bar{\sigma} \in \Sigma ; M \tmop{odd}, (\bar{\sigma}, 0) \in
   \Sigma \} \setminus \{ (0) \}} \varepsilon^{M_{\bar{\sigma}} + 2
   L_{\bar{\sigma}} - 1} \int_0^t S_{t - s} e^{- \lambda s} S_{\bar{\sigma}}
   z_s d s \]
\[ - e^{- \lambda t} \delta^{- d / 2} \sum_{\sigma \in \Sigma} \varepsilon^{M
   + 2 L} \varphi_{\sigma} (u_t^{\varepsilon, \delta}, m_t^{\varepsilon,
   \delta}), \]
where the integrals involving the $J_i$ are understood against the
corresponding measures $d s$ or $d W_s$.

After applying the low modes projection $P_{\leqslant R}$, which commutes with
the semigroup (see Section~\ref{s:perturb}), the right hand side is estimated
as follows. Section~\ref{s:stoch} treats the stochastic integrals and
Section~\ref{s:quanti} the enhanced diffusion correction $J_5$.
Section~\ref{s:error} handles the other lower order terms: the remaining
$J_i$, the boundary terms of the corrector sums, and the error term involving
$S_{\bar{\sigma}} z$. Section~\ref{s:concl} then absorbs the convective
integrals into the left hand side by means of the endpoint estimate of
Section~\ref{s:convective}, and concludes with the final choice $R = \delta^{-
1}$.}

\section{Construction of the semigroup}\label{s:semigroup}

\subsection{Fourier multiplier structure}\label{s:transl}

We show that the effective operators $S_{\sigma}$ are Fourier multipliers on
$H$, which in particular implies that they commute with the spectral
projections $P_{\leqslant R}$, $P_{> R}$, and the Stokes operator $\Delta
\mathbb{P}$. The key step is establishing translation invariance, where we
denote $(\tau_y f) (x) = f (x - y)$, $y \in \mathbb{T}^d$.

\begin{proposition}
  \label{p:37}For every $\sigma \in \Sigma$ with $M$ odd and $(\sigma, 0) \in
  \Sigma$, the operator $S_{\sigma}$ is translation invariant, i.e.
  $S_{\sigma} \tau_y = \tau_y S_{\sigma}$ for every $y \in \mathbb{T}^d$.
  Consequently, $S_{\sigma}$ is a Fourier multiplier on $H$ with symbol
  $s_{\sigma} : \mathbb{Z}_0^d \rightarrow \mathbb{R}^{(d - 1) \times (d -
  1)}$, where $s_{\sigma} (k)$ represents the action of $S_{\sigma}$ on
  $k^{\perp}$ in the basis $(\sigma_{k, \alpha})_{\alpha = 1, \ldots, d - 1}$,
  and commutes with the spectral projections $P_{\leqslant R}$ and $P_{> R}$,
  and the Stokes operator $\Delta \mathbb{P}$.
\end{proposition}

\begin{proof}
  Step 1: Fourier multiplier structure. Assuming translation invariance, we
  show that $S_{\sigma}$ is a Fourier multiplier on $H$. Since $\tau_y
  \sigma_{k, \alpha} = e^{- i k \cdot y} \sigma_{k, \alpha}$, translation
  invariance implies that $S_{\sigma} \sigma_{k, \alpha}$ is an eigenfunction
  of every $\tau_y$ with eigenvalue $e^{- i k \cdot y}$. Comparing Fourier
  coefficients forces $S_{\sigma} \sigma_{k, \alpha}$ to lie in $\tmop{span}
  \{ \sigma_{k, \beta}, \beta = 1, \ldots, d - 1 \}$, so
  
  \[ S_{\sigma} \sigma_{k, \alpha} = \sum_{\beta = 1}^{d - 1} s_{\sigma}
     (k)^{\alpha \beta} \sigma_{k, \beta} \]
  for some matrix $s_{\sigma} (k) \in \mathbb{R}^{(d - 1) \times (d - 1)}$. By
  linearity and completeness this defines $S_{\sigma}$ as a Fourier multiplier
  with symbol $s_{\sigma} (k)$. It then commutes with every Fourier multiplier
  on $H$ whose symbol is a scalar multiple of the identity at each frequency,
  in particular $P_{\leqslant R}$, $P_{> R}$ and $\Delta \mathbb{P}$.
  
  Step 2: Translation invariance. Translation invariance follows from three
  ingredients: stationarity of $n$, which allows replacing $n$ by $\tau_y n$
  under the expectation; translation invariance of $\nabla$ and $\mathbb{P}$,
  which commute with $\tau_y$ as Fourier multipliers; and the fact that
  $\tau_y$ commutes with expectation due to linearity. Since each term in
  $S_{\sigma}$ is a finite composition of these operations, the argument is
  the same for all terms. We illustrate it on one representative term from
  $S_{(0, 0, 0)}$ below.
  
  Step 3: Case study. Recall that $S_{(0, 0, 0)}$ contains the leading order
  terms in $[\varphi_{(0, 0, 0)} (\mathbb{P} (n \cdummy \nabla u),
  n)]^{\bullet}$, that is, the terms where after the application of Leibniz
  rule, one derivative hits $u$, one derivative hits the test function and the
  remaining $M + 2 L - 1 = 2$ derivatives hit one of the copies of the
  Ornstein--Uhlenbeck process. Recall from Section~\ref{s:appl} that, up to
  some constants,
  \[ \varphi_{(0, 0, 0)} (u, n) \cong \langle n^{\otimes 3}, G
     \rangle_{H^{\otimes 3}} + \langle n, \tilde{\mathcal{C}}_1^{\delta} G
     \rangle, \]
  where $G$ is defined through
  \[ F (n) = \langle n \cdummy \nabla \mathbb{P} (n \cdummy \nabla \mathbb{P}
     (n \cdummy \nabla u)), h \rangle = \langle n^{\otimes 3}, G
     \rangle_{H^{\otimes 3}} \]
  and the contraction terms are defined as
  \[ \langle n, \tilde{\mathcal{C}}_1^{\delta} G \rangle \cong
     \mathbb{E}_{n_1} [\langle n_1 \cdummy \nabla \mathbb{P} (n_1 \cdummy
     \nabla \mathbb{P} (n \cdummy \nabla u)), h \rangle] +\mathbb{E}_{n_1}
     [\langle n_1 \cdummy \nabla \mathbb{P} (n \cdummy \nabla \mathbb{P} (n_1
     \cdummy \nabla u)), h \rangle] \]
  \[ +\mathbb{E}_{n_1} [\langle n \cdummy \nabla \mathbb{P} (n_1 \cdummy
     \nabla \mathbb{P} (n_1 \cdummy \nabla u)), h \rangle] . \]
  Accordingly,
  \[ [\varphi_{(0, 0, 0)} (\mathbb{P} (n \cdummy \nabla u), n)]^{\bullet}
     \cong [\langle n \cdummy \nabla \mathbb{P} (n \cdummy \nabla \mathbb{P}
     (n \cdummy \nabla \mathbb{P} (n \cdummy \nabla u))), h \rangle]^{\bullet}
  \]
  \[ + [\mathbb{E}_{n_1} [\langle n_1 \cdummy \nabla \mathbb{P} (n_1 \cdummy
     \nabla \mathbb{P} (n \cdummy \nabla \mathbb{P} (n \cdummy \nabla u))), h
     \rangle]]^{\bullet} + [\mathbb{E}_{n_1} [\langle n_1 \cdummy \nabla
     \mathbb{P} (n \cdummy \nabla \mathbb{P} (n_1 \cdummy \nabla \mathbb{P} (n
     \cdummy \nabla u))), h \rangle]]^{\bullet} \]
  \[ + [\mathbb{E}_{n_1} [\langle n \cdummy \nabla \mathbb{P} (n_1 \cdummy
     \nabla \mathbb{P} (n_1 \cdummy \nabla \mathbb{P} (n \cdummy \nabla u))),
     h \rangle]]^{\bullet}, \]
  where the outside expectation $[\cdummy]^{\bullet}$ is with respect to $n$.
  Integrating by parts and applying the Leibniz rule, we collect the terms
  where only one derivative is applied to $h$ and one to $u$. All four terms
  have the same structure -- a finite composition of $\nabla$, $\mathbb{P}$
  and expectations over stationary fields -- and differ only in which copies
  of $n$ are contracted against $n_1$. The translation invariance argument is
  therefore identical for each.
  
  Let us illustrate the calculation on one representative term, for instance,
  the third term on the right hand side rewrites as
  \[ [\mathbb{E}_{n_1} [\langle n_1 \cdummy \nabla \mathbb{P} (n \cdummy
     \nabla \mathbb{P} (n_1 \cdummy \nabla \mathbb{P} (n \cdummy \nabla u))),
     h \rangle]]^{\bullet} = - [\mathbb{E}_{n_1} [\langle \mathbb{P} (n
     \cdummy \nabla u), n_1 \cdummy \nabla \mathbb{P} (n \cdummy \nabla
     \mathbb{P} (n_1 \cdummy \nabla h)) \rangle]]^{\bullet} \]
  and one of the leading order terms is
  \[ - [\mathbb{E}_{n_1} [\langle \mathbb{P} (n \cdummy \nabla u), n^{i_1}_1
     \mathbb{P} ((\partial^{i_1} n^{i_2}) \mathbb{P} ((\partial^{i_2} n_1)
     \cdummy \nabla h)) \rangle]]^{\bullet} \]
  \[ = [\mathbb{E}_{n_1} [\langle \tmop{div} ((\partial^{i_2} n_1) \mathbb{P}
     ((\partial^{i_1} n^{i_2}) \mathbb{P} (n_1^{i_1} \mathbb{P} (n \cdummy
     \nabla u)))), h \rangle]]^{\bullet} . \]
  In order to show its translation invariance, we include $\tau_y$, use the
  commutativity of $\nabla$ and $\tau_y$, translation invariance of the law of
  $n$, commutativity of $\mathbb{P}$ and $\tau_y$, translation invariance of
  the law of $n_1$, and the commutation of $\tau_y$ with expectations as
  follows
  \[ [\mathbb{E}_{n_1} [\tmop{div} ((\partial^{i_2} n_1) \mathbb{P}
     ((\partial^{i_1} n^{i_2}) \mathbb{P} (n_1^{i_1} \mathbb{P} (n \cdummy
     \nabla \tau_y u))))]]^{\bullet} \]
  \[ = [\mathbb{E}_{n_1} [\tmop{div} ((\partial^{i_2} n_1) \mathbb{P}
     ((\partial^{i_1} n^{i_2}) \mathbb{P} (n_1^{i_1} \mathbb{P} (n \cdummy
     \tau_y \nabla u))))]]^{\bullet} \]
  \[ = [\mathbb{E}_{n_1} [\tmop{div} ((\partial^{i_2} n_1) \mathbb{P}
     ((\partial^{i_1} \tau_y n^{i_2}) \mathbb{P} (n_1^{i_1} \mathbb{P} (\tau_y
     n \cdummy \tau_y \nabla u))))]]^{\bullet} \]
  \[ = [\mathbb{E}_{n_1} [\tmop{div} ((\partial^{i_2} n_1) \mathbb{P}
     ((\partial^{i_1} \tau_y n^{i_2}) \mathbb{P} (n_1^{i_1} \tau_y \mathbb{P}
     (n \cdummy \nabla u))))]]^{\bullet} \]
  \[ = [\mathbb{E}_{n_1} [\tmop{div} ((\tau_y \partial^{i_2} n_1) \mathbb{P}
     (\tau_y (\partial^{i_1} n^{i_2}) \tau_y \mathbb{P} (n_1^{i_1} \mathbb{P}
     (n \cdummy \nabla u))))]]^{\bullet} \]
  \[ = [\mathbb{E}_{n_1} [\tmop{div} ((\tau_y \partial^{i_2} n_1) \tau_y
     \mathbb{P} ((\partial^{i_1} n^{i_2}) \mathbb{P} (n_1^{i_1} \mathbb{P} (n
     \cdummy \nabla u))))]]^{\bullet} \]
  \[ = \tau_y [\mathbb{E}_{n_1} [\tmop{div} ((\partial^{i_2} n_1) \mathbb{P}
     ((\partial^{i_1} n^{i_2}) \mathbb{P} (n_1^{i_1} \mathbb{P} (n \cdummy
     \nabla u))))]]^{\bullet} . \]
\end{proof}

\subsection{Semigroup construction and smoothing estimates}\label{s:perturb}

\begin{proposition}
  \label{p:p20}\tmcolor{black}{If $\varepsilon \delta^{- 1}$ is sufficiently
  small, }the operator $A$ defined in \eqref{eq:A} is a generator of a
  $C_0$-semigroup $S_t$, $t \geqslant 0$, on $H^{\alpha}$ for every $\alpha
  \in \mathbb{R}$. For all $\alpha \in \mathbb{R}$ and $\theta > 0$,
  \[ \| S_t u \|_{H^{\alpha + \theta}} \lesssim e^{- \lambda t} t^{- \theta /
     2} \| u \|_{H^{\alpha}} . \]
  Moreover, $S_t$ commutes with the spectral projections $P_{\leqslant R}$ and
  $P_{> R}$ for all $t \geqslant 0$, and for all $\alpha \in \mathbb{R}$,
  \[ \left\| {\int_0} ^{\cdummy} S_{\cdummy - s} f_s d s \right\|_{L^2
     H^{\alpha + 2}} \lesssim \| f \|_{L^2 {H^{\alpha}}  }. \]
\end{proposition}

\begin{proof}
  Step 1: Symbol bound. By Proposition~\ref{p:37}, each $S_{\sigma}$ is a
  Fourier multiplier on $H$ with matrix-valued symbol $s_{\sigma} (k) \in
  \mathbb{R}^{(d - 1) \times (d - 1)}$. Consequently, $A$ is a Fourier
  multiplier on $H$ with symbol
  \[ a (k) \assign - \lambda \tmop{Id} - (1 + \nu) | k |^2 \mathbb{P} (k) +
     \sum_{\{ \sigma \in \Sigma ; M \tmop{odd}, (\sigma, 0) \in \Sigma \}
     \setminus \{ (0) \}} \varepsilon^{M + 2 L - 1} s_{\sigma} (k), \]
  where $\mathbb{P} (k) = \tmop{Id}_{\mathbb{R}^d} - \frac{k \otimes k}{| k
  |^2}$ is the symbol of Leray projection. By Proposition~\ref{c:c18} applied
  to $u = \sigma_{k, \alpha}$ and $h = \sigma_{k, \beta}$ for which $\| \nabla
  u \|_{L^2} = \| \nabla h \|_{L^2} = | k |$, it holds for any $k \in
  \mathbb{Z}^d_0$
  \begin{equation}
    | s_{\sigma} (k)^{\alpha \beta} | = | \langle S_{\sigma} u, h \rangle |
    \lesssim \delta^{- (M + 2 L - 1)} | k |^2, \label{eq:sk}
  \end{equation}
  for all $\alpha, \beta \in \{ 1, \ldots, d - 1 \}$.
  
  Step 2: Exponential bound. For any $v \in k^{\perp}$, set $w (t) \assign
  e^{t a (k)} v$. Testing the ODE $\dot{w} = a (k) w$ against $w$,
  \[ \frac{d}{d t} | w |^2 = 2 \langle a (k) w, w \rangle \leqslant 2 (-
     \lambda - c (1 + \nu) | k |^2) | w |^2, \]
  where the bound on $\langle a (k) w, w \rangle$ follows from Step 1 and
  subcriticality $\varepsilon = o (\delta^{1 + \iota})$, which ensures the
  perturbation terms are absorbed \tmcolor{black}{provided $\varepsilon
  \delta^{- 1}$ is sufficiently small}. Gr{\"o}nwall's inequality and taking
  the supremum over all $v \in k^{\perp}$ with $| v | = 1$ yields
  \begin{equation}
    | e^{t a (k)} | \leqslant e^{- \lambda t} e^{- c (1 + \nu) | k |^2 t} .
    \label{eq:matrix}
  \end{equation}
  Step 3: Semigroup. Since $A$ is a Fourier multiplier, we define $S_t$
  frequency by frequency via the matrix exponential,
  \[ \mathcal{F} [S_t u] (k) \assign e^{t a (k)} \hat{u} (k) . \]
  That this defines a $C_0$-semigroup with generator $A$ on every $H^{\alpha}$
  follows from the fact that $t \mapsto e^{t a (k)}$ is the unique solution to
  the matrix ODE $\dot{X} (t) = a (k) X (t)$, $X (0) = \tmop{Id}$: the
  semigroup law $S_t S_s = S_{t + s}$ holds frequency by frequency from the
  matrix identity $e^{t a (k)} e^{s a (k)} = e^{(t + s) a (k)}$, and strong
  continuity $S_t u \rightarrow u$ in $H^{\alpha}$ as $t \rightarrow 0$
  holds by dominated convergence, since \eqref{eq:matrix} implies $| e^{t a
  (k)} | \leqslant 1$ for $t \geqslant 0$ uniformly in $k$. Commutativity with
  $P_{\leqslant R}$ and $P_{> R}$ is immediate since these are scalar Fourier
  multipliers.
  
  Step 4: Smoothing. By \eqref{eq:matrix} and the elementary bound $e^{- c (1
  + \nu) | k |^2 t} \lesssim_{\theta} (c (1 + \nu) | k |^2 t)^{- \theta / 2}$,
  which follows from $s^{\theta / 2} e^{- s} \lesssim_{\theta} 1$ for all $s >
  0$ and $\theta > 0$,
  \[ \| S_t u \|^2_{H^{\alpha + \theta}} = \sum_{k \in \mathbb{Z}_0^d} | k
     |^{2 \alpha + 2 \theta} | e^{t a (k)} \hat{u} (k) |^2 \lesssim e^{- 2
     \lambda t} t^{- \theta} \sum_{k \in \mathbb{Z}^d_0} | k |^{2 \alpha} |
     \hat{u} (k) |^2 \lesssim e^{- 2 \lambda t} t^{- \theta} \| u
     \|^2_{H^{\alpha}} . \]
  Step 5: Maximal regularity. From the matrix exponential bound
  \eqref{eq:matrix} we obtain
  \[ \int_0^{\infty} | e^{t a (k)} | d t \lesssim \int_0^{\infty} e^{- \lambda
     t} e^{- c (1 + \nu) | k |^2 t} d t = \frac{1}{\lambda + c (1 + \nu) | k
     |^2} \lesssim | k |^{- 2} \]
  uniformly in $k$. By Young's convolution inequality applied at each
  frequency $k$,
  \[ \int_0^T \left| \int_0^t e^{(t - s) a (k)} \hat{f}_s (k) d s \right|^2 d
     t \leqslant \left( \int_0^{\infty} | e^{t a (k)} | d t \right)^2 \int_0^T
     | \hat{f}_s (k) |^2 d s \lesssim | k |^{- 4} \int_0^T | \hat{f}_s (k) |^2
     d s. \]
  Multiplying by $| k |^{2 \alpha + 4}$ and summing over $k \in
  \mathbb{Z}^d_0$,
  \[ \left\| {\int_0} ^{\cdummy} S_{\cdummy - s} f_s d s \right\|_{L^2
     H^{\alpha + 2}}^2 = \int_0^T \sum_{k \in \mathbb{Z}^d_0} | k |^{2 \alpha
     + 4} \left| \int_0^t e^{(t - s) a (k)} \hat{f}_s (k) d s \right|^2 d t \]
  \[ \lesssim \int_0^T \sum_{k \in \mathbb{Z}^d_0} | k |^{2 \alpha} |
     \hat{f}_s (k) |^2 d s = \| f \|_{L^2 H^{\alpha}}^2 . \]
\end{proof}

\begin{lemma}
  \label{l:l21}For all $\alpha \in \mathbb{R}$ and any $L_2 (H,
  H^{\alpha})$-valued progressively measurable process $F$,
  \[ \mathbb{E} \left\| \int_0^{\cdummy} S_{\cdummy - s} F_s d W_s
     \right\|_{L^2 H^{\alpha + 1}}^2 \lesssim \mathbb{E} \| F \|^2_{L^2 L_2
     (H, H^{\alpha})} . \]
\end{lemma}

\begin{proof}
  By stochastic Fubini's theorem
  \[ \mathcal{F} \left[ \int_0^{\cdummy} S_{\cdummy - s} F_s d W_s \right] (k)
     = \sum_{\ell, \beta} \int_0^{\cdummy} \mathcal{F} [S_{\cdummy - s} F_s
     \sigma_{\ell, \beta}] (k) d W^{\ell, \beta}_{s} = \int_0^{\cdummy}
     e^{(\cdummy - s) a (k)} \hat{F}_s (k) d W_s, \]
  where $\hat{F}_s (k)$ denotes the operator $H \rightarrow \mathbb{R}^{d -
  1}$ defined by $\hat{F}_s (k) \sigma_{\ell, \beta} =\mathcal{F} [F_s
  \sigma_{\ell, \beta}] (k)$. Hence by It{\^o}'s isometry and the inequality
  $\| B C \|_{L_2 (H, \mathbb{R}^{d - 1})} \leqslant \| B \|_{\mathcal{L}
  (\mathbb{R}^{d - 1}, \mathbb{R}^{d - 1})} \| C \|_{L_2 (H, \mathbb{R}^{d -
  1})}$,
  \[ \mathbb{E} \left[ \int_0^T \left| \int_0^t e^{(t - s) a (k)} \hat{F}_s
     (k) d W_s \right|^2 d t \right] =\mathbb{E} \left[ \int_0^T \int_0^t \|
     e^{(t - s) a (k)} \hat{F}_s (k) \|_{L_2 (H, \mathbb{R}^{d - 1})}^2 d s d
     t \right] \]
  \[ \leqslant \mathbb{E} \left[ \int_0^T \int_0^t | e^{(t - s) a (k)} |^2 \|
     \hat{F}_s (k) \|_{L_2 (H, \mathbb{R}^{d - 1})}^2 d s d t \right] . \]
  By Young's inequality for convolutions and the matrix exponential bound
  \eqref{eq:matrix},
  \[ \lesssim \left( \int_0^{\infty} | e^{t a (k)} |^2 d t \right) \mathbb{E}
     \left[ \int_0^T \| \hat{F}_s (k) \|_{L_2 (H, \mathbb{R}^{d - 1})}^2 d s
     \right] \lesssim | k |^{- 2} \mathbb{E} \left[ \int_0^T \| \hat{F}_s (k)
     \|_{L_2 (H, \mathbb{R}^{d - 1})}^2 d s \right], \]
  where we used
  \[ \int_0^{\infty} | e^{t a (k)} |^2 d t \lesssim \int_0^{\infty} e^{- 2
     \lambda  t} e^{- 2 c (1 + \nu) | k |^2 t} d t \lesssim | k |^{-
     2} . \]
  Multiplying by $| k |^{2 \alpha + 2}$ and summing over $k \in
  \mathbb{Z}^d_0$ gives the result by the Parseval identity for
  Hilbert--Schmidt operators.
\end{proof}

\section{Endpoint convective estimate}\label{s:convective}

The first result is designed to treat the terms in \eqref{eq:plo1},
\eqref{eq:42}, where the available bounds for $u$, $u^{\varepsilon, \delta}$
and $v^{\varepsilon, \delta}$ are limited by the energy space $L^{\infty} (0,
T ; L^2) \cap L^2 (0, T ; H^1)$.

\begin{proposition}
  \label{p:2}Let $\beta \in (0, 1]$, $p \in (1, \infty)$. Assume that $f \in
  L^p (0, T ; H^{- \beta})$ and $g \in L^{\infty} (0, T ; L^2) \cap L^2 (0, T
  ; H^1)$ are both divergence free. Then
  \begin{equation}
    \left\| \int_0^{\cdummy} S_{\cdummy - s} (f_s \cdummy \nabla g_s) d s
    \right\|_{L^p H^{- \beta}} + \left\| \int_0^{\cdummy} S_{\cdummy - s} (g_s
    \cdummy \nabla f_s) d s \right\|_{L^p H^{- \beta}} \lesssim \| f \|_{L^p
    H^{- \beta}} \| g \|_{L^{2 / \beta} H^{\beta} \cap L^4 L^4} .
    \label{eq:mmm}
  \end{equation}
\end{proposition}

\begin{proof}
  First, we observe that since $g \in L^{\infty} (0, T ; L^2) \cap L^2 (0, T ;
  H^1)$, by interpolation $g \in L^{2 / \beta} (0, T ; H^{\beta})$ and by
  Ladyzhenskaya's inequality in $d = 2$ $g \in L^4 (0, T ; L^4)$, so the right
  hand side of the estimate is finite.
  
We estimate
  \[ \| S_{t - s} (f_s \cdummy \nabla g_s) \|_{H^{- \beta}} + \| S_{t - s}
     (g_s \cdummy \nabla f_s) \|_{H^{- \beta}} \]
  \[ \lesssim \sup_{h \in H^{\beta}, \| h \|_{H^{\beta}} \leqslant 1} \langle
     S_{t - s} (f_s \cdummy \nabla g_s), h \rangle + \sup_{h \in H^{\beta}, \|
     h \|_{H^{\beta}} \leqslant 1} \langle S_{t - s} (g_s \cdummy \nabla f_s),
     h \rangle \]
  \[ \leqslant \sup_{h \in H^{\beta}, \| h \|_{H^{\beta}} \leqslant 1} \langle
     g_s, f_s \cdummy \nabla S^{\ast}_{t - s} h \rangle + \sup_{h \in
     H^{\beta}, \| h \|_{H^{\beta}} \leqslant 1} \langle f_s, g_s \cdummy
     \nabla S^{\ast}_{t - s} h \rangle \]
  \begin{equation}
    \lesssim \| f \|_{H^{- \beta}} \sup_{h \in H^{\beta}, \| h \|_{H^{\beta}}
    \leqslant 1} \| (\nabla S^{\ast}_{t - s} h) g_s \|_{H^{\beta}} .
    \label{eq:3211}
  \end{equation}
  Here $S_{t - s}^{\ast}$ denotes the adjoint semigroup generated by the
  adjoint symbol $a (k)^{\ast}$. Since  the Gr{\"o}nwall argument of Step 2
  in Proposition~\ref{p:p20} applies verbatim to $a (k)^{\ast}$, it satisfies
  the same smoothing estimates.
  
  The objective is to estimate the right-hand side of \eqref{eq:3211} in such
  a way that the resulting time convolution still belongs to $L^p (0, T)$.
  This is a genuinely critical regime: any additional gain in regularity from
  the semigroup leads to a time kernel with a non-integrable singularity, so
  the smoothing must be exploited in a scale-invariant fashion via Lorentz
  spaces and O'Neil's convolution inequality, Lemma~\ref{l:lor}.
  
  To estimate the norm in the supremum in \eqref{eq:3211}, we use the
  paraproduct decomposition (see, e.g., {\cite[Chapter 2]{MR2768550}})
  \[ (\nabla S^{\ast}_{t - s} h) g_s = (\nabla S^{\ast}_{t - s} h) \prec g_s +
     (\nabla S^{\ast}_{t - s} h) \succ g_s + (\nabla S^{\ast}_{t - s} h) \circ
     g_s . \]
  Among the various ways of estimating the paraproduct terms and the resonant
  product, it is crucial to control the resulting singular behavior in $(t -
  s)^{- 1}$. In particular, we use the following paraproduct estimates (which
  can be found in Lemma 2.2 in {\cite{MR4581461}} and in Theorem 3.17 in
  {\cite{MR3693966}})
  \begin{equation}
    \| (\nabla S^{\ast}_{t - s} h) \prec g_s \|_{H^{\beta}} \lesssim \| g_s
    \|_{H^{\beta}} \| \nabla S^{\ast}_{t - s} h \|_{L^{\infty}}, \label{eq:p1}
  \end{equation}
  \begin{equation}
    \| (\nabla S^{\ast}_{t - s} h) \succ g_s \|_{H^{\beta}} \lesssim \| g_s
    \|_{L^4} \| \nabla S^{\ast}_{t - s} h \|_{B^{\beta}_{4, 2}}, \label{eq:p2}
  \end{equation}
  \begin{equation}
    \| (\nabla S^{\ast}_{t - s} h) \circ g_s \|_{H^{\beta}} \lesssim \| g_s
    \|_{B^0_{4, \infty}} \| \nabla S^{\ast}_{t - s} h \|_{B^{\beta}_{4, 2}} .
    \label{eq:p3}
  \end{equation}
  We make use of the smoothing of the semigroup from Proposition~\ref{p:p20}.
  Specifically, combined with Agmon's inequality in 2d and the Besov
  embedding, we obtain
  
  $\| \nabla S^{\ast}_{t - s} h \|_{L^{\infty}} \lesssim \| \nabla S^{\ast}_{t
  - s} h \|_{L^2}^{1 / 2} \| \nabla S^{\ast}_{t - s} h \|_{H^2}^{1 / 2}
  \lesssim (t - s)^{- (1 - \beta) / 4 - (3 - \beta) / 4} \| h \|_{H^{\beta}}
  \lesssim (t - s)^{- (2 - \beta) / 2} \| h \|_{H^{\beta}},$
  
  and
  \[ \| \nabla S^{\ast}_{t - s} h \|_{B^{\beta}_{4, 2}} \lesssim \| \nabla
     S^{\ast}_{t - s} h \|_{H^{\beta + 1 / 2}} \lesssim (t - s)^{- (1 / 2 + 1)
     / 2} \| h \|_{H^{\beta}} \lesssim (t - s)^{- 3 / 4} \| h \|_{H^{\beta}} .
  \]
  Next, we observe that $L^4 \subset B^0_{4, \infty}$, since each
  Littlewood--Paley block $\Delta_j$ is a Fourier multiplier with a smooth
  compactly supported symbol, hence bounded on $L^4$ with an operator norm
  independent of $j$. Specifically, by Young inequality for convolutions
  \[ \| f \|_{B^0_{4, \infty}} = \sup_{j \geqslant - 1} \| \Delta_j f \|_{L^4}
     \leqslant \sup_{j \geqslant - 1} \| (\mathcal{F}^{- 1} \Delta_j) \ast f
     \|_{L^4} \leqslant \| f \|_{L^4} . \]
  As a consequence, the right hand side of \eqref{eq:p3} is controlled by the
  right hand side of \eqref{eq:p2}.
  
  Putting these bounds back in \eqref{eq:3211}, we obtain
  \[ \left\| \int_0^{\cdummy} S_{\cdummy - s} (f_s \cdummy \nabla g_s + g_s
     \cdummy \nabla f_s) d s \right\|_{L^p H^{- \beta}} \lesssim \left\|
     \int_0^{\cdummy} (\cdummy - s)^{- (2 - \beta) / 2} \| f_s \|_{H^{-
     \beta}} \| g_s \|_{H^{\beta}} d s \right\|_{L^p} \]
  \[ + \left\| \int_0^{\cdummy} (\cdummy - s)^{- 3 / 4} \| f_s \|_{H^{-
     \beta}} \| g_s \|_{L^4} d s \right\|_{L^p} . \]
  The key observation is that $g \in L^{2 / \beta} H^{\beta}$, and that the
  time singularity in the first term is of order $(t - s)^{- 1 / a}$, where $a
  = 2 / (2 - \beta)$ is the conjugate exponent to $2 / \beta$. Similarly, for
  the second term we have $g \in L^4 (0, T ; L^4)$ and the corresponding time
  kernel behaves like $(t - s)^{- 1 / b}$, with $b = 4 / 3$ being the
  conjugate exponent to $4$.
  
  In both cases, the kernels fail to belong to $L^a (0, T)$ and $L^b (0, T)$,
  respectively, but they do lie in the weak Lebesgue spaces $L^{a, \infty} (0,
  T)$ and $L^{b, \infty} (0, T)$. We are therefore precisely in the endpoint
  regime where Lemma~\ref{l:lor} applies, which allows us to conclude the
  estimate \eqref{eq:mmm}.
\end{proof}

The following result, sufficient for our purposes, illustrates that stronger
time integrability of $g$ yields a gain of a negative power of $\lambda$,
giving smallness for large $\lambda$.

\begin{proposition}
  \label{p:22}Let $\beta \in (0, 1]$, $p \in (1, \infty)$ and $\lambda > 0$.
  Assume that $f \in L^p (0, T ; H^{- \beta})$ and $g \in L^{\infty} (0, T ;
  H^{\beta}) \cap L^{\infty} (0, T ; L^4)$ are both divergence free. Then
  there exists $\vartheta > 0$ so that
  \[ \left\| \int_0^{\cdummy} S_{\cdummy - s} (f_s \cdummy \nabla g_s) d s
     \right\|_{L^p H^{- \beta}} + \left\| \int_0^{\cdummy} S_{\cdummy - s}
     (g_s \cdummy \nabla f_s) d s \right\|_{L^p H^{- \beta}} \lesssim_T
     \lambda^{- \vartheta} \| f \|_{L^p H^{- \beta}} \| g \|_{L^{\infty}
     H^{\beta} \cap L^{\infty} L^4} . \]
\end{proposition}

\begin{proof}
  We proceed exactly as in the proof of Proposition~\ref{p:2} making use of
  the exponential decay of the semigroup from Proposition~\ref{p:p20} to
  obtain
  \[ \left\| \int_0^{\cdummy} S_{\cdummy - s} (f_s \cdummy \nabla g_s) d s
     \right\|_{L^p H^{- \beta}} + \left\| \int_0^{\cdummy} S_{\cdummy - s}
     (g_s \cdummy \nabla f_s) d s \right\|_{L^p H^{- \beta}} \]
  \[ \lesssim \left\| \int_0^{\cdummy} e^{- \lambda (\cdummy - s)} (\cdummy -
     s)^{- (2 - \beta) / 2} \| f_s \|_{H^{- \beta}} \| g_s \|_{H^{\beta}} d s
     \right\|_{L_t^p} \]
  \[ + \left\| \int_0^{\cdummy} e^{- \lambda (\cdummy - s)} (\cdummy - s)^{- 3
     / 4} \| f_s \|_{H^{- \beta}} \| g_s \|_{L^4} d s \right\|_{L_t^p} . \]
  Different from Proposition~\ref{p:2}, $g$ now belongs to $L^{\infty} (0, T ;
  H^{\beta}) \cap L^{\infty} (0, T ; L^4)$ so this is further bounded by
  \[ \lesssim \| g \|_{L^{\infty} H^{\beta}} \left\| \int_0^{\cdummy} e^{-
     \lambda (\cdummy - s)} (\cdummy - s)^{- (2 - \beta) / 2} \| f_s \|_{H^{-
     \beta}} d s \right\|_{L^p} \]
  \[ + \| g \|_{L^{\infty} L^4} \left\| \int_0^{\cdummy} e^{- \lambda
     (\cdummy - s)} (\cdummy - s)^{- 3 / 4} \| f_s \|_{H^{- \beta}} d s
     \right\|_{L^p} \]
  and classical Young's convolution inequality yields the claim since $e^{-
  \lambda r} r^{- (2 - \beta) / 2}$ and $e^{- \lambda r} r^{- 3 / 4}$ belong
  to $L^1 (0, \infty)$ with $L^1$-norms bounded by $\lambda^{- \vartheta}$ for
  some $\vartheta > 0$ depending only on $\beta$.
\end{proof}

\section{Limiting and auxiliary equations}\label{s:10}

\subsection{Limit stochastic integral}\label{s:stoch}

The stochastic integral associated to the first corrector gives rise, after
rescaling by $\delta^{- d / 2}$, to the only stochastic term surviving in the
limit equation for the fluctuation $z$.

\begin{proposition}
  \label{p:stoch}For every $\kappa > d / 2 - 1$, the following convergence
  holds in $L^2 (\Omega ; L^2 (0, T ; H^{- \kappa}))$:
  \[ \delta^{- d / 2} \int_0^t S_{t - s} (\mathcal{Q}^{1 / 2}_{\delta} d W_s
     \cdummy \nabla u_s^{\varepsilon, \delta}) \rightarrow \chi \int_0^t S_{t
     - s} (d W_s \cdummy \nabla u_s), \]
  where $\chi = (\mathcal{F}_{\mathbb{R}^d} K) (0)$.
\end{proposition}

\begin{proof}
  By maximal regularity for stochastic convolutions Lemma~\ref{l:l21}
  \[ \mathbb{E} \left[ \left\| \delta^{- d / 2} \int_0^{\cdummy} S_{\cdummy -
     s} (\mathcal{Q}^{1 / 2}_{\delta} d W_s \cdummy \nabla u_s^{\varepsilon,
     \delta}) - \chi \int_0^{\cdummy} S_{\cdummy - s} (d W_s \cdummy \nabla
     u_s) \right\|_{L^2 H^{- \kappa}}^2 \right] \]
  \[ \lesssim \mathbb{E} [\| \delta^{- d / 2} \mathcal{Q}^{1 / 2}_{\delta}
     (\cdummy) \cdummy \nabla u^{\varepsilon, \delta} - \chi (\cdummy) \cdummy
     \nabla u \|_{L^2 L_2 (H ; H^{- 1 - \kappa})}^2] . \]
  We split the right hand side as
  \begin{equation}
    \lesssim \mathbb{E} [\| \delta^{- d / 2} \mathcal{Q}^{1 / 2}_{\delta}
    (\cdummy) \cdummy \nabla (u^{\varepsilon, \delta} - u) \|_{L^2 L_2 (H ;
    H^{- 1 - \kappa})}^2] \label{eq:500}
  \end{equation}
  \begin{equation}
    +\mathbb{E} [\| (\delta^{- d / 2} \mathcal{Q}^{1 / 2}_{\delta} (\cdummy) -
    \chi (\cdummy)) \cdummy \nabla u \|_{L^2 L_2 (H ; H^{- 1 - \kappa})}^2] .
    \label{eq:511}
  \end{equation}
  For \eqref{eq:500}, we apply \eqref{eq:eigen} and boundedness of
  $\mathcal{F}_{\mathbb{R}^d} K$ to obtain
  \[ \| \delta^{- d / 2} \mathcal{Q}^{1 / 2}_{\delta} (\cdummy) \cdummy \nabla
     (u^{\varepsilon, \delta}_{s} - u_{s}) \|_{L_2 (H ; H^{- 1 - \kappa})}^2 =
     \sum_{k, \alpha} | (\mathcal{F}_{\mathbb{R}^d} K) (\delta k) |^2 \|
     \sigma_{k, \alpha} \cdummy \nabla (u_s^{\varepsilon, \delta} - u_{s})
     \|_{H^{- 1 - \kappa}}^2 \]
  \[ \lesssim \sum_{k, \alpha} \| \sigma_{k, \alpha} \cdummy \nabla
     (u_s^{\varepsilon, \delta} - u_{s}) \|_{H^{- 1 - \kappa}}^2 . \]
  Using that $(\sigma_{\ell, \beta})_{\ell \in \mathbb{Z}^d_0, \beta \in \{ 1,
  \ldots, d - 1 \}}$ is an orthonormal basis of $H$ and summing in $k, \alpha$
  first via Parseval's identity and using $1 + \kappa > d / 2$, this is
  further estimated as
  \[ \sum_{k, \alpha} \| \sigma_{k, \alpha} \cdummy \nabla (u_s^{\varepsilon,
     \delta} - u_s) \|_{H^{- 1 - \kappa}}^2 = \sum_{k, \alpha, \ell, \beta} (1
     + | \ell |)^{- 2 (1 + \kappa)} | \langle \sigma_{k, \alpha} \cdummy
     \nabla (u_s^{\varepsilon, \delta} - u_s), \sigma_{\ell, \beta} \rangle
     |^2 \]
  \[ = \sum_{\ell, \beta} (1 + | \ell |)^{- 2 (1 + \kappa)} \| (\nabla
     (u_s^{\varepsilon, \delta} - u_s))^T \sigma_{\ell, \beta} \|_{L^2}^2
     \lesssim \| \nabla (u_s^{\varepsilon, \delta} - u_s) \|_{L^2}^2
     \sum_{\ell} (1 + | \ell |)^{- 2 (1 + \kappa)} \]
  \[ \lesssim \| \nabla (u_s^{\varepsilon, \delta} - u_s) \|_{L^2}^2 . \]
  Integrating in time and taking expectation, this upper bound vanishes as
  $\varepsilon, \delta \rightarrow 0$ by Corollary~\ref{c:14}, the energy
  inequality \eqref{eq:energy} and Vitali's convergence theorem.
  
  An analogous argument gives
  \[ \| (\delta^{- d / 2} \mathcal{Q}^{1 / 2}_{\delta} (\cdummy) - \chi
     (\cdummy)) \cdummy \nabla u_s \|_{L_2 (H ; H^{- 1 - \kappa})}^2 =
     \sum_{k, \alpha} | (\mathcal{F}_{\mathbb{R}^d} K) (\delta k) -
     (\mathcal{F}_{\mathbb{R}^d} K) (0) |^2 \| \sigma_{k, \alpha} \cdummy
     \nabla u_s \|_{H^{- 1 - \kappa}}^2 \]
  \[ = \sum_{k, \alpha, \ell, \beta} | (\mathcal{F}_{\mathbb{R}^d} K) (\delta
     k) - (\mathcal{F}_{\mathbb{R}^d} K) (0) |^2 (1 + | \ell |)^{- 2 (1 +
     \kappa)} | \langle \sigma_{k, \alpha} \cdummy \nabla u_s, \sigma_{\ell,
     \beta} \rangle |^2 . \]
  Since $\mathcal{F}_{\mathbb{R}^d} K$ is a Schwartz function, it is bounded
  and continuous at $k = 0$, so $(\mathcal{F}_{\mathbb{R}^d} K) (\delta k)
  \rightarrow (\mathcal{F}_{\mathbb{R}^d} K) (0)$ for every $k$. The
  dominating function is summable since
  \[ \mathbb{E} \left[ \int_0^T \sum_{k, \alpha, \ell, \beta} |
     (\mathcal{F}_{\mathbb{R}^d} K) (\delta k) - (\mathcal{F}_{\mathbb{R}^d}
     K) (0) |^2 (1 + | \ell |)^{- 2 (1 + \kappa)} | \langle \sigma_{k, \alpha}
     \cdummy \nabla u_s, \sigma_{\ell, \beta} \rangle |^2 d s \right] \]
  \[ \lesssim \mathbb{E} \left[ \int_0^T \sum_{k, \alpha, \ell, \beta} (1 + |
     \ell |)^{- 2 (1 + \kappa)} | \langle \sigma_{k, \alpha} \cdummy \nabla
     u_s, \sigma_{\ell, \beta} \rangle |^2 d s \right] \]
  \[ =\mathbb{E} \left[ \int_0^T \sum_{\ell, \beta} (1 + | \ell |)^{- 2 (1 +
     \kappa)} \| (\nabla u_s)^T \sigma_{\ell, \beta} \|_{L^2}^2 d s \right]
     \lesssim \mathbb{E} [\| \nabla u \|_{L^2 L^2}^2] \]
  by Parseval's identity and $1 + \kappa > d / 2$. Dominated convergence
  therefore yields convergence of \eqref{eq:511} to zero as $\delta
  \rightarrow 0$.
\end{proof}

\subsection{Fluctuation limit $z$}\label{s:z}

Recall that the fluctuation limit shall solve the stochastic linearized
Navier--Stokes equations
\begin{equation}
  d z + [z \cdummy \nabla u + u \cdummy \nabla z+\nabla p_{z}] d t = (1 + \nu) \Delta z d t
  + \chi d W \cdummy \nabla u, \quad \tmop{div} z = 0, \quad z (0) = 0,
  \label{eq:z}
\end{equation}
with a space-time white noise $W$ on $H$, enhanced viscosity $\nu =
\frac{1}{16} \| K \|_{L^2}^2$ and $\chi = (\mathcal{F}_{\mathbb{R}^2} K) (0)$
and with $u$ being the limit solution to \eqref{eq:ulim}.

\begin{lemma}
  \label{l:z}For every $\kappa > 0$, \eqref{eq:z} admits a unique solution $z
  \in L^2 (\Omega ; L^2 (0, T ; H^{- \kappa}))$.
\end{lemma}

\begin{proof}
  We aim to prove the existence of a unique solution $z$ by mild formulation.
  Due to the criticality of the convective estimate in Proposition~\ref{p:2},
  this would be directly possible under an additional smallness assumption on
  $u_0 \in L^2$. To overcome this issue, we approximate $u$ by its suitable
  Galerkin approximant $\bar{u}$ and we include an additional exponential
  factor $e^{- \lambda t}$ for $\lambda > 0$ chosen sufficiently large. The
  equation for $\tilde{z}_t \assign e^{- \lambda t} z_t$ reads as
  \[ d \tilde{z} + [\tilde{z} \cdummy \nabla u + u \cdummy \nabla \tilde{z}+ \nabla p_{\tilde{z}} ] d
     t = (1 + \nu) \Delta \tilde{z} d t - \lambda \tilde{z} d t + e^{- \lambda
     t} \chi d W \cdummy \nabla u, \quad \tmop{div} \tilde{z} = 0, \quad
     \tilde{z} (0) = 0. \]
  Similarly to Section~\ref{s:stoch} the stochastic convolution
  \[ y_t \assign \chi \int_0^t e^{- \lambda (t - s)} e^{(1 + \nu) \Delta
     \mathbb{P} (t - s)} e^{- \lambda s} d W_s \cdummy \nabla u_s \]
  belongs to $L^2 L^2 H^{- \kappa}$ for $\kappa > 0$ in $d = 2$ uniformly in
  $\lambda > 0$. Thus, we seek a fixed point of the map
  \[ (\Gamma \tilde{z})_t \assign - \int_0^t e^{- \lambda (t - s)} e^{(1 + \nu) \Delta
     \mathbb{P} (t - s)}
     [\tilde{z}_s \cdummy \nabla u_s + u_s \cdummy \nabla \tilde{z}_s] d s +
     y_t \]
  in $L^2 L^2 H^{- \kappa}$ for $\kappa > 0$. For the deterministic part, we
  add and subtract $\bar{u}$ and apply Proposition~\ref{p:2} and
  Proposition~\ref{p:22} which continue to hold exactly the same way for the
  semigroup $e^{- \lambda t} e^{(1 + \nu) \Delta \mathbb{P}t}$, $t \geqslant
  0$, to obtain
  \[ \left\| \int_0^t e^{- \lambda (t - s)} e^{(1 + \nu) \Delta \mathbb{P} (t
     - s)} [\tilde{z}_s \cdummy \nabla u_s + u_s \cdummy \nabla \tilde{z}_s] d
     s \right\|_{L^2 L^2 H^{- \kappa}} \]
  \[ \lesssim \| \tilde z \|_{L^2 L^2 H^{- \kappa}} (\| u - \bar{u} \|_{L^{2 /
     \kappa} H^{\kappa} \cap L^4 L^4} + \lambda^{- \vartheta} \| \bar{u}
     \|_{L^{\infty} H^{\kappa} \cap L^{\infty} L^4}) . \]
  Therefore, choosing $\bar{u}$ and then $\lambda > 0$ to satisfy
  \[ \| u - \bar{u} \|_{L^{2 / \kappa} H^{\kappa} \cap L^4 L^4} + \lambda^{-
     \vartheta} \| \bar{u} \|_{L^{\infty} H^{\kappa} \cap L^{\infty} L^4} \ll
     1, \]
  we deduce that the map $\Gamma$ maps bounded sets of $L^2 L^2 H^{- \kappa}$
  to itself and is contractive. Hence there exists a unique fixed point.
\end{proof}

\subsection{Macroscopic correction $v^{\varepsilon, \delta}$}\label{s:v}

Recall that $v^{\varepsilon, \delta}$ is defined as a solution to
\[ \partial_t v^{\varepsilon, \delta} + v^{\varepsilon, \delta} \cdummy \nabla
   u + u \cdummy \nabla v^{\varepsilon, \delta} + v^{\varepsilon, \delta}
   \cdummy \nabla v^{\varepsilon, \delta} + \nabla p_{v^{\varepsilon, \delta}}
\]
\[ = (1 + \nu) \Delta v^{\varepsilon, \delta} + \sum_{\{ \sigma \in \Sigma ; M
   \tmop{odd}, (\sigma, 0) \in \Sigma \} \setminus \{ (0) \}} \varepsilon^{M +
   2 L - 1} S_{\sigma} v^{\varepsilon, \delta} \]
\begin{equation}
  + \sum_{\{ \sigma \in \Sigma ; M \tmop{odd}, (\sigma, 0) \in \Sigma \}
  \setminus \{ (0) \}} \varepsilon^{M + 2 L - 1} S_{\sigma} u, \label{eq:Sv}
\end{equation}
\[ \tmop{div} v^{\varepsilon, \delta} = 0, \quad v^{\varepsilon, \delta} (0) =
   0, \]
where $p_{v^{\varepsilon, \delta}}$ denotes the associated pressure.

\begin{lemma}
  \label{l:v}\tmcolor{black}{If $\varepsilon \delta^{- 1}$ is sufficiently
  small}, there exists a unique solution $v^{\varepsilon, \delta} \in
  L^{\infty} (0, T ; H) \cap L^2 (0, T ; H^1)$ satisfying
  \[ \frac{1}{2} \| v^{\varepsilon, \delta}_t \|_{L^2}^2 + \left( \frac{1}{2}
     + \nu \right) \int_0^t \| \nabla v^{\varepsilon, \delta}_s \|_{L^2}^2 d s
     \lesssim_{T, u_0} (\varepsilon \delta^{- 1})^4 . \]
\end{lemma}

\begin{proof}
  We proceed via a Galerkin approximation and derive uniform energy estimates;
  the passage to the limit is standard and omitted.
  
  Testing the equation for $v^{\varepsilon, \delta}$ against $v^{\varepsilon,
  \delta}$ and using $\tmop{div} v^{\varepsilon, \delta} = 0$ and
  Proposition~\ref{c:c18}, we obtain
  \[ \frac{1}{2} \frac{d}{d t} \| v^{\varepsilon, \delta} \|_{L^2}^2 + (1 +
     \nu) \| \nabla v^{\varepsilon, \delta} \|_{L^2}^2 = - \langle
     v^{\varepsilon, \delta} \cdummy \nabla u, v^{\varepsilon, \delta} \rangle
     + \sum_{\{ \sigma \in \Sigma ; M \tmop{odd}, (\sigma, 0) \in \Sigma \}
     \setminus \{ (0) \}} \varepsilon^{M + 2 L - 1} \langle S_{\sigma}
     v^{\varepsilon, \delta}, v^{\varepsilon, \delta} \rangle \]
  \[ + \sum_{\{ \sigma \in \Sigma ; M \tmop{odd}, (\sigma, 0) \in \Sigma \}
     \setminus \{ (0) \}} \varepsilon^{M + 2 L - 1} \langle S_{\sigma} u,
     v^{\varepsilon, \delta} \rangle \]
  \[ \lesssim \| \nabla u \|_{L^2} \| v^{\varepsilon, \delta} \|_{L^4}^2 +
     \sum_{\{ \sigma \in \Sigma ; M \tmop{odd}, (\sigma, 0) \in \Sigma \}
     \setminus \{ (0) \}} \varepsilon^{M + 2 L - 1} \delta^{- (M + 2 L - 1)}
     \| \nabla v^{\varepsilon, \delta} \|^2_{L^2} \]
  \[ + \sum_{\{ \sigma \in \Sigma ; M \tmop{odd}, (\sigma, 0) \in \Sigma \}
     \setminus \{ (0) \}} \varepsilon^{M + 2 L - 1} \delta^{- (M + 2 L - 1)}
     \| \nabla u \|_{L^2} \| \nabla v^{\varepsilon, \delta} \|_{L^2} . \]
  The first term on the right hand side is bounded by Ladyzhenskaya inequality
  in $d = 2$
  \[ \| \nabla u \|_{L^2} \| v^{\varepsilon, \delta} \|_{L^4}^2 \lesssim \|
     \nabla u \|_{L^2} \| v^{\varepsilon, \delta} \|_{L^2} \| \nabla
     v^{\varepsilon, \delta} \|_{L^2} \leqslant \frac{1}{4} \| \nabla
     v^{\varepsilon, \delta} \|_{L^2}^2 + C \| \nabla u \|^2_{L^2} \|
     v^{\varepsilon, \delta} \|_{L^2}^2 . \]
  The first part is absorbed into the left hand side and the second part
  enters Gr{\"o}nwall's inequality. The second term is absorbed into the left
  hand side \tmcolor{black}{for $\varepsilon \delta^{- 1}$ sufficiently small}
  by subcriticality. The third term is bounded by Young's inequality
  \[ \leqslant \frac{1}{4} \| \nabla v^{\varepsilon, \delta} \|_{L^2}^2 + C
     \left( \sum_{\{ \sigma \in \Sigma ; M \tmop{odd}, (\sigma, 0) \in \Sigma
     \} \setminus \{ (0) \}} \varepsilon^{M + 2 L - 1} \delta^{- (M + 2 L -
     1)} \right)^2 \| \nabla u \|^2_{L^2} . \]
  The first part is absorbed to the left hand side. For the second part, the
  leading order terms in the sum are $\varepsilon^2 S_{(0, 1)} u$ and
  $\varepsilon^2 S_{(0, 0, 0)} u$ both controlled by $(\varepsilon \delta^{-
  1})^2$, which gives the bound in the statement after integrating in time and
  applying Gr{\"o}nwall's inequality with the energy estimate for $u$.
  
  For uniqueness, let $v_1$, $v_2$ be two solutions. After a preliminary
  mollification, we test the equation for the difference $v_1 - v_2$ by
  itself. The only term requiring attention is the difference of the
  nonlinearities:
  \[ \langle v_1 \cdummy \nabla v_1 - v_2 \cdummy \nabla v_2, v_1 - v_2
     \rangle = \langle (v_1 - v_2) \cdummy \nabla v_1, v_1 - v_2 \rangle +
     \langle v_2 \cdummy \nabla (v_1 - v_2), v_1 - v_2 \rangle \]
  \[ = \langle (v_1 - v_2) \cdummy \nabla v_1, v_1 - v_2 \rangle \leqslant \|
     \nabla v_1 \|_{L^2} \| v_1 - v_2 \|_{L^4}^2, \]
  where we used $\langle v_2 \cdummy \nabla (v_1 - v_2), v_1 - v_2 \rangle =
  0$ by the divergence-free constraint. The same argument as for the first
  term above then closes the estimate, yielding $v_1 = v_2$.
\end{proof}

\section{Quantitative enhanced diffusion}\label{s:quanti}

The operator $S_{\delta}$ is given explicitly by
\[ S_{\delta} v =\mathbb{P} [n \cdummy \nabla \mathbb{P} (n \cdummy \nabla
   v)]^{\bullet} = \frac{\delta^d}{2} \sum_{k \in \mathbb{Z}_0^d, \alpha = 1,
   \dots,d-1} | (\mathcal{F}_{\mathbb{R}^d} K) (\delta k) |^2 \mathbb{P} (\sigma_{k,
   \alpha} \cdummy \nabla \mathbb{P} (\sigma_{- k, \alpha} \cdummy \nabla v))
   . \]
For the fluctuation analysis, the remainder $J_5$ in the mild formulation
involves $S_{\delta} - \nu \mathbb{P} \Delta$ after an additional division by
$\delta^{d / 2}$. A qualitative convergence statement is therefore
insufficient: one needs a quantitative rate that is fast enough to compensate
this additional factor.

We establish quantitative improvements of the two approximation steps in
Section~\ref{s:quali}, Lemma~\ref{lem:3} and Lemma~\ref{lem:5}, giving
explicit rates. Throughout, we work in dimension $d = 2$, where the
two-dimensional lattice summability plays a crucial role, see
Remark~\ref{r:21}.

\subsection{Improved Lemma~\ref{lem:3}}\label{s:12.1}

The key object in the analysis of $s_{\delta} (\ell) - \tilde{s}_{\delta}
(\ell)$ is the rank-one projection
\[ \Pi (p) \assign \frac{p \otimes p}{| p |^2} \qquad p \in \mathbb{R}^2
   \setminus \{ 0 \}, \]
in terms of which the Leray projection is given by $\tmop{Id}_{\mathbb{R}^2} -
\Pi (p)$. The difference $s_{\delta} (\ell) - \tilde{s}_{\delta} (\ell)$
involves sums of $\Pi (k - \ell) - \Pi (k)$ over the lattice, weighted by $|
(\mathcal{F}_{\mathbb{R}^2} K) (\delta k) |^2$. The main improvement over the
qualitative argument in the proof of Lemma~\ref{lem:3} is a second-order
Taylor expansion of $\Pi (k - \ell)$ around $k$: since $\Pi$ is even, its
derivative $D \Pi$ is odd, and the first-order term cancels in the symmetric
lattice sum. The effective remainder is therefore of order $| \ell |^2 / | k
|^2$, which is summable in two dimensions up to a logarithmic correction. The
precise statement is given in Lemma~\ref{lem:8}, applied in
Corollary~\ref{cor:9}.

\begin{remark}
  \label{r:21}The restriction to $d = 2$ is not merely technical. In dimension
  $d = 3$, the analogous lattice sum $\sum_{1 \leqslant | k | \leqslant
  \delta^{- 1}} | k |^{- 2}$ diverges like $\delta^{- 1}$ rather than
  logarithmically, leading to a remainder that does not vanish after the
  fluctuation rescaling by $\delta^{d / 2}$. This reflects a genuine
  obstruction already at the level of the linear Stokes equation,
  independently of the nonlinearity of the Navier--Stokes system.
\end{remark}

\begin{corollary}
  \label{cor:9}For every $\ell \in \mathbb{Z}^2$ it holds
  \[ | s_{\delta} (\ell) - \tilde{s}_{\delta} (\ell) | \lesssim | \ell |^4
     \delta^2 (1 + \log \delta^{- 1}) . \]
\end{corollary}

\begin{proof}
  We apply Lemma~\ref{lem:8} to $w_{\delta} (k, \ell) = |
  (\mathcal{F}_{\mathbb{R}^2} K) (\delta k) |^2 \hspace{0.27em} \sin^2
  (\angle_{k, \ell})$, which is even in $k$ and dominated by a Schwartz
  function because $\mathcal{F}_{\mathbb{R}^2} K$ is Schwartz and $\sin^2
  \leqslant 1$. Since the symbols $s_{\delta} (\ell)$ and $\tilde{s}_{\delta}
  (\ell)$ both vanish at $\ell = 0$ and carry an explicit factor of $| \ell
  |^2$, the difference $s_{\delta} (\ell) - \tilde{s}_{\delta} (\ell)$ is
  expressed in terms of $\Sigma_{\delta} (\ell)$ with an additional factor of
  $| \ell |^2$, yielding the result.
\end{proof}

\subsection{Improved Lemma~\ref{lem:5}}

Recall that $\tilde{s}_{\delta} (\ell)$ and its limit $\tilde{s} (\ell)$ in $d
= 2$ are given by
\[ \tilde{s}_{\delta} (\ell) = - \frac{(2 \pi)^{- 2}}{2} \left( \tmop{Id} -
   \frac{\ell \otimes \ell}{| \ell |^2} \right) | \ell |^2 J_{\delta} (\ell),
\]
with
\[ J_{\delta} (\ell) \assign \delta^2 \sum_{k \in \mathbb{Z}^2_0} w (\delta k)
   \sin^2 (\angle_{k, \ell}) \left( \tmop{Id} - \frac{k \otimes k}{| k |^2}
   \right), \qquad w \assign | \mathcal{F}_{\mathbb{R}^2} K |^2, \]
and
\[ \tilde{s} (\ell) = - \frac{(2 \pi)^{- 2}}{2} \left( \tmop{Id} - \frac{\ell
   \otimes \ell}{| \ell |^2} \right) | \ell |^2 J (\ell), \]
with
\[ J (\ell) \assign \int_{\mathbb{R}^2} w (k) \sin^2 (\angle_{k, \ell}) \left(
   \tmop{Id} - \frac{k \otimes k}{| k |^2} \right) d k = 2 (2
   \pi)^2 \nu, \]
so that $\tilde{s}$ is the symbol of $\nu \mathbb{P} \Delta$.
The difference $\tilde{s}_{\delta} (\ell) - \tilde{s} (\ell)$ reduces to
estimating the Riemann sum approximation error $J_{\delta} (\ell) - J (\ell)$.
The key technical ingredient is a quadrature estimate for even functions
Lemma~\ref{l:13}, which exploits the evenness of the integrand to cancel the
first-order error term.

\begin{lemma}
  \label{lem:24}For every $\ell \in \mathbb{Z}_0^2$ it holds
  \[ | J_{\delta} (\ell) - J (\ell) | \lesssim \delta^2 (1 + \log \delta^{-
     1}) . \]
\end{lemma}

\begin{proof}
  First, we observe that
  \[ J_{\delta} (\ell) = \delta^2 \sum_{k \in \mathbb{Z}^2_0} w (\delta k)
     \sin^2 (\angle_{\delta k, \ell}) \left( \tmop{Id} - \frac{\delta k
     \otimes \delta k}{| \delta k |^2} \right) \backassign \delta^2 \sum_{k
     \in \mathbb{Z}^2_0} F_{\ell} (\delta k), \]
  and
  \[ J (\ell) = \int_{\mathbb{R}^2} F_{\ell} (k) d k. \]
  Thus, $J_{\delta} (\ell) - J (\ell)$ is a Riemann sum approximation error.
  The main issue is that $\left( \tmop{Id} - \frac{\xi \otimes \xi}{| \xi |^2}
  \right)$ is not continuous at $\xi = 0$ so the integrand fails to be
  continuous at $\xi = 0$ unless $(\mathcal{F}_{\mathbb{R}^2} K) (0) = 0$,
  which we do not want to assume. To overcome this, we truncate as follows.
  
  Step 1: Smooth cut-off at scale $\delta$. Let $\chi \in C^{\infty}
  (\mathbb{R}^2)$ be radial, even, with $\chi (k) = 0$ for $| k | \leqslant
  1$, $\chi (k) = 1$ for $| k | \geqslant 2$. Define $\chi_{\delta} (k)
  \assign \chi (\delta^{- 1} k)$, $f_{\delta} (k) \assign \chi_{\delta} (k)
  F_{\ell} (k)$. Then $f_{\delta} \in C^{\infty} (\mathbb{R}^2)$, $f_{\delta}$
  is even, $D^2 f_{\delta} \in L^1 (\mathbb{R}^2)$.
  
  Step 2: Decompose the error. Write
  \[ J_{\delta} (\ell) - J (\ell) = \left( \delta^2 \sum_{k \in
     \mathbb{Z}^2_0} f_{\delta} (\delta k) - \int_{\mathbb{R}^2} f_{\delta}
     (k) d k \right) \]
  \[ + \delta^2 \sum_{k \in \mathbb{Z}^2_0} (1 - \chi_{\delta} (\delta k))
     F_{\ell} (\delta k) - \int_{\mathbb{R}^2} (1 - \chi_{\delta} (k))
     F_{\ell} (k) d k = : A_{\delta} (\ell) + B_{\delta} (\ell) + C_{\delta}
     (\ell), \]
  where $A_{\delta}$ is the Riemann sum error away from $0$, $B_{\delta}
  (\ell)$ is the small-$k$ sum and $C_{\delta} (\ell)$ is the small-$k$
  integral.
  
  Step 3: Small-$k$ expressions are $O (\delta^2)$. It follows immediately
  that $| F_{\ell} (k) | \leqslant \| w \|_{L^{\infty}} \lesssim 1$ hence
  \[ | C_{\delta} (\ell) | \lesssim \int_{| k | \leqslant 2 \delta} d k
     \lesssim \delta^2, \]
  and similarly
  \[ | B_{\delta} (\ell) | \lesssim \delta^2 \sum_{k \in \mathbb{Z}^2_0, 0 < |
     k | \leqslant 2} 1 \lesssim \delta^2 . \]
  Step 4: Riemann sum away from $0$ with a rate using the corner-evenness
  quadrature Lemma~\ref{l:13}. Since $f_{\delta}$ is smooth, even and $D^2
  f_{\delta} \in L^1 (\mathbb{R}^2)$, Lemma~\ref{l:13} implies
  \[ \left| \delta^2 \sum_{k \in \mathbb{Z}^2} f_{\delta} (\delta k) -
     \int_{\mathbb{R}^2} f_{\delta} (x) d x \right| \lesssim \delta^2 \| D^2
     f_{\delta} \|_{L^1 (\mathbb{R}^2)} . \]
  Thus, it remains to estimate $\| D^2 f_{\delta} \|_{L^1 (\mathbb{R}^2)}$.
  
  Step 5: Estimate of $\| D^2 f_{\delta} \|_{L^1 (\mathbb{R}^2)}$. By Leibniz'
  rule
  \[ D^2 f_{\delta} = \chi_{\delta} D^2 F_{\ell} + 2 (D \chi_{\delta}) \otimes
     (D F_{\ell}) + (D^2 \chi_{\delta}) F_{\ell} . \]
  We estimate each term. First, we recall that
  \[ \sin^2 (\angle_{k, \ell}) = 1 - \frac{(k \cdummy \ell)^2}{| k |^2 | \ell
     |^2} \backassign 1 - A_{\ell} (k) \]
  hence
  \[ F_{\ell} (k) = w (k) \left( 1 - \frac{(k \cdummy \ell)^2}{| k |^2 | \ell
     |^2} \right) \left( \tmop{Id} - \frac{k \otimes k}{| k |^2} \right), \]
  and $A_{\ell} (k)$ is homogeneous of degree $0$ in $k$, smooth away from $k
  = 0$. Hence by the same homogeneity argument as in Step 1 of
  Lemma~\ref{lem:8}, $| D A_{\ell} (k) | \lesssim | k |^{- 1}$, $| D^2
  A_{\ell} (k) | \lesssim | k |^{- 2}$, with constants independent of $\ell$,
  since $\ell$ is fixed and only appears via the unit vector $\ell / | \ell
  |$. Similarly, the Leray projector symbol $\tmop{Id} - \frac{k \otimes k}{|
  k |^2}$ satisfies the same scaling $| D\mathbb{P} (k) | \lesssim | k |^{-
  1}$, $| D^2 \mathbb{P} (k) | \lesssim | k |^{- 2}$. This leads to
  \[ | D^2 F_{\ell} (k) | \lesssim | D^2 w (k) | + | D w
     (k) | | k |^{- 1} + | w (k) | | k |^{- 2} \]
  valid for $| k | > \delta$. Moreover, $D \chi_{\delta}$ and $D^2
  \chi_{\delta}$ are supported in the annulus $\{ \delta \leqslant | k |
  \leqslant 2 \delta \}$ and $| D \chi_{\delta} | \lesssim \delta^{- 1}$, $|
  D^2 \chi_{\delta} | \lesssim \delta^{- 2}$.
  
  Putting everything together, we end up with
  \[ | D^2 f_{\delta} | \lesssim (| D^2 w (k) | + | D w (k)
     | | k |^{- 1} + | w (k) | | k |^{- 2}) 1_{| k | \geqslant \delta} \]
  \[ + \delta^{- 1} (| D w (k) | + | w (k) | | k |^{- 1}) 1_{\delta \leqslant
     | k | \leqslant 2 \delta} + \delta^{- 2} | w (k) | 1_{\delta \leqslant |
     k | \leqslant 2 \delta} . \]
  Here, the terms on the second line are bounded by $\lesssim \delta^{- 2}
  1_{\delta \leqslant |k| \leqslant 2 \delta}$ hence their integral over
  $\mathbb{R}^2$ is bounded by 1 uniformly in $\delta$. For the very first
  term, we have
  \[ \int_{| k | \geqslant \delta} | D^2 w (k) | d k
     \lesssim 1. \]
  For the remaining terms, we decompose
  \[ \int_{| k | \geqslant \delta} (| D w (k) | | k |^{- 1} + | w (k) | | k
     |^{- 2}) d k = \left( \int_{1 \geqslant | k | \geqslant \delta} + \int_{|
     k | \geqslant 1} \right) (| D w (k) | | k |^{- 1} + | w (k) | | k |^{-
     2}) d k. \]
  At infinity, i.e., for $| k | \geqslant 1$, we use that $w = |
  \mathcal{F}_{\mathbb{R}^2} K |^2$ is Schwartz, hence
  \[ \int_{| k | \geqslant 1} (| D w (k) | | k |^{- 1} + | w (k) | | k |^{-
     2}) d k \lesssim 1 \]
  uniformly in $\delta$. Near zero, i.e. $\delta \leqslant | k | \leqslant 1$,
  we obtain the logarithmic blow up:
  \[ \int_{1 \geqslant | k | \geqslant \delta} (| D w (k) | | k |^{- 1} + | w
     (k) | | k |^{- 2}) d k \lesssim \int_{1 \geqslant | k | \geqslant \delta}
     (| k |^{- 1} + | k |^{- 2}) d k \]
  and changing to polar coordinates gives
  \[ \lesssim \int_{\delta}^1 (r^{- 1} + r^{- 2}) r d r \lesssim 1 + \log
     \delta^{- 1} . \]
  Accordingly,
  \[ \| D^2 f_{\delta} \|_{L^1 (\mathbb{R}^2)} \lesssim 1 + \log \delta^{- 1}
  \]
  and the proof is complete.
\end{proof}

As an immediate consequence, we obtain the following quantitative improvement
of Lemma~\ref{lem:5} in $d = 2$.

\begin{corollary}
  \label{cor:12}For every $\ell \in \mathbb{Z}^2$ it holds
  \[ | \tilde{s}_{\delta} (\ell) - \tilde{s} (\ell) | \lesssim | \ell |^2
     \delta^2 (1 + \log \delta^{- 1}) . \]
\end{corollary}

\begin{proof}
  Since $\tilde{s}_{\delta} (\ell) - \tilde{s} (\ell) = - \frac{(2 \pi)^{-
  2}}{2} \left( \tmop{Id} - \frac{\ell \otimes \ell}{| \ell |^2} \right) |
  \ell |^2 (J_{\delta} (\ell) - J (\ell))$, the result follows directly from
  Lemma~\ref{lem:24}.
\end{proof}

\subsection{Enhanced diffusion correction}\label{s:enh}

We now combine the quantitative rates on both $s_{\delta} (\ell) -
\tilde{s}_{\delta} (\ell)$ from Corollary~\ref{cor:9} and $\tilde{s}_{\delta}
(\ell) - \tilde{s} (\ell)$ from Corollary~\ref{cor:12} to bound the enhanced
diffusion remainder $J_5$ in the mild formulation.

\begin{lemma}
  \label{l:11.3}Let $\beta > 0$. Then
  \begin{equation}
    \delta^{- 1} \left\| \int_0^{\cdummy} S_{\cdummy - s} P_{\leqslant R}
    [S_{\delta} u_s^{\varepsilon, \delta} - \nu \Delta {u}^{\varepsilon,
    \delta}_s] d s \right\|_{L^2 L^2 H^{- \beta}} \lesssim \delta (1 + \log
    \delta^{- 1}) (R^{1 - \beta} \vee 1) . \label{eq:newcond}
  \end{equation}
\end{lemma}

\begin{proof}
  We decompose $S_{\delta} - \nu \mathbb{P} \Delta = (S_{\delta} -
  \tilde{S}_{\delta}) + (\tilde{S}_{\delta} - \nu \mathbb{P} \Delta)$ and
  estimate each term separately. Specifically, by maximal regularity
  \[ \left\| \int_0^{\cdummy} S_{\cdummy - s} P_{\leqslant R} [S_{\delta}
     u_s^{\varepsilon, \delta} - \tilde{S}_{\delta} u^{\varepsilon, \delta}_s]
     d s \right\|_{L^2 L^2 H^{- \beta}} \lesssim \| P_{\leqslant R}
     [S_{\delta} u^{\varepsilon, \delta} - \tilde{S}_{\delta} u^{\varepsilon,
     \delta}] \|_{L^2 L^2 H^{- 2 - \beta}}, \]
  where by Corollary~\ref{cor:9}
  \[ \| P_{\leqslant R} [S_{\delta} u^{\varepsilon, \delta} -
     \tilde{S}_{\delta} u^{\varepsilon, \delta}] \|^2_{H^{- 2 - \beta}} =
     \sum_{1 \leqslant | \ell | \leqslant R} | \ell |^{- 4 - 2 \beta} |
     (s_{\delta} (\ell) - \tilde{s}_{\delta} (\ell))
     \mathcal{F}u^{\varepsilon, \delta} (\ell) |^2 \]
  \[ \lesssim \delta^4 (1 + \log \delta^{- 1})^2 \sum_{1 \leqslant | \ell |
     \leqslant R} | \ell |^{4 - 2 \beta} | \mathcal{F}u^{\varepsilon, \delta}
     (\ell) |^2 \lesssim \delta^4 (1 + \log \delta^{- 1})^2 (R^{2 - 2 \beta}
     \vee 1) \| u^{\varepsilon, \delta} \|_{H^1}^2 . \]
  Dividing by $\delta$ we obtain by energy inequality \eqref{eq:energy}
  \[ \delta^{- 1} \left\| \int_0^{\cdummy} S_{\cdummy - s} P_{\leqslant R}
     [S_{\delta} u_s^{\varepsilon, \delta} - \tilde{S}_{\delta}
     u^{\varepsilon, \delta}_s] d s \right\|_{L^2 L^2 H^{- \beta}} \lesssim
     \delta (1 + \log \delta^{- 1}) R^{1 - \beta} \| u_0 \|_{L^2} \]
  and \eqref{eq:newcond} is satisfied.
  
  The same reasoning based on Corollary~\ref{cor:12} yields
  \[ \left\| \int_0^{\cdummy} S_{\cdummy - s} P_{\leqslant R}
     [\tilde{S}_{\delta} u_s^{\varepsilon, \delta} - \nu \Delta
     u^{\varepsilon, \delta}_s] d s \right\|_{L^2 L^2 H^{- \beta}} \lesssim \|
     P_{\leqslant R} [\tilde{S}_{\delta} u_s^{\varepsilon, \delta} - \nu
     \Delta u^{\varepsilon, \delta}_s] \|_{L^2 L^2 H^{- 2 - \beta}}, \]
  where
  \[ \| P_{\leqslant R} [\tilde{S}_{\delta} u_s^{\varepsilon, \delta} - \nu
     \Delta u^{\varepsilon, \delta}_s] \|^2_{H^{- 2 - \beta}} \lesssim
     \delta^4 (1 + \log \delta^{- 1})^2 \sum_{1 \leqslant | \ell | \leqslant
     R} | \ell |^{- 2 \beta} | \mathcal{F}u^{\varepsilon, \delta} (\ell) |^2
  \]
  \[ \lesssim \delta^4 (1 + \log \delta^{- 1})^2 \| u^{\varepsilon, \delta}
     \|^2_{H^1}, \]
  hence
  \[ \delta^{- 1} \left\| \int_0^{\cdummy} S_{\cdummy - s} P_{\leqslant R}
     [\tilde{S}_{\delta} u_s^{\varepsilon, \delta} - \nu \Delta
     u^{\varepsilon, \delta}_s] d s \right\|_{L^2 L^2 H^{- \beta}} \lesssim
     \delta (1 + \log \delta^{- 1}) \]
  which completes the proof.
\end{proof}

\section{Lower order error terms}\label{s:error}

\subsection{Laplacian remainders}\label{s:laplace}

\begin{lemma}
  \label{l:26}Let $\beta > 0$. For every $\sigma \in \Sigma$, the Laplacian
  remainder term satisfies
  \[ \delta^{- 1} \varepsilon^{M + 2 L} \left\| \int_0^{\cdummy} S_{\cdummy -
     s} \left( \mathbb{P} \Delta + \sum_{\{ \bar{\sigma} \in \Sigma ; M
     \tmop{odd}, (\bar{\sigma}, 0) \in \Sigma \} \setminus \{ (0) \}}
     \varepsilon^{M_{\bar{\sigma}} + 2 L_{\bar{\sigma}} - 1} S_{\bar{\sigma}}
     \right) P_{\leqslant R} \varphi_{\sigma} (u^{\varepsilon, \delta}_s,
     m^{\varepsilon, \delta}_s) d s \right\|_{L^2 L^2 H^{- \beta}} \]
  \begin{equation}
    \lesssim \varepsilon^{M + 2 L} (\log \varepsilon^{- 1})^{M / 2} (\delta^{-
    M - 2 L + 1} + \delta^{- 1} (R^{M + 2 L - 1 - \beta} \vee 1)) .
    \label{cond:11.1}
  \end{equation}
\end{lemma}

\begin{proof}
  For every $\sigma \in \Sigma$, we first apply maximal regularity to bound
  \begin{equation}
    \varepsilon^{M + 2 L} \left\| \int_0^{\cdummy} S_{\cdummy - s} \mathbb{P}
    \Delta P_{\leqslant R} \varphi_{\sigma} (u^{\varepsilon, \delta}_s,
    m^{\varepsilon, \delta}_s) d s \right\|_{L^2 L^2 H^{- \beta}} \lesssim
    \varepsilon^{M + 2 L} \| P_{\leqslant R} \varphi_{\sigma} (u^{\varepsilon,
    \delta}, m^{\varepsilon, \delta}) \|_{L^2 L^2 H^{- \beta}} .
    \label{eq:51aa}
  \end{equation}
  Next, we write the $H^{- \beta}$-norm in duality form against test functions
  in $H^{\beta}$ and observe that $\varphi_{\sigma}$ contains $M + 2 L$
  derivatives and at most $M + 2 L - 1$ go to the test function and at most $M
  + 2 L - 2$ hit $m^{\varepsilon, \delta}$ or the covariance
  $\mathcal{Q}_{\delta}$. Since
  \[ \| P_{\leqslant R} h \|_{H^{M + 2 L - 1}} \lesssim R^{M + 2 L - 1 -
     \beta} \| h \|_{H^{\beta}}, \]
  we obtain by Proposition~\ref{c:sup}
  \[ \varepsilon^{M + 2 L} \left\| \int_0^{\cdummy} S_{\cdummy - s} \mathbb{P}
     \Delta P_{\leqslant R} \varphi_{\sigma} (u^{\varepsilon, \delta}_s,
     m^{\varepsilon, \delta}_s) d s \right\|_{L^2 L^2 H^{- \beta}} \]
  \[ \lesssim \| u_0 \|_{L^2} \varepsilon^{M + 2 L} (\log \varepsilon^{-
     1})^{M / 2} (\delta^{- M - 2 L + 2} + R^{M + 2 L - 1 - \beta} \vee 1), \]
  which after the division by $\delta$ leads to \eqref{cond:11.1}.
  
  The part with $\varepsilon^{M_{\bar{\sigma}} + 2 L_{\bar{\sigma}} - 1}
  S_{\bar{\sigma}}$ is estimated similarly since due to \eqref{eq:sk}
  \[ \varepsilon^{M_{\bar{\sigma}} + 2 L_{\bar{\sigma}} - 1} \|
     S_{\bar{\sigma}} v \|_{{H^{- 2 - \beta}} } \lesssim
     \varepsilon^{M_{\bar{\sigma}} + 2 L_{\bar{\sigma}} - 1} \delta^{-
     M_{\bar{\sigma}} - 2 L_{\bar{\sigma}} + 1} \| v \|_{H^{- \beta}} \lesssim
     \| v \|_{H^{- \beta}} \]
  and consequently by maximal regularity
  \[ \varepsilon^{M + 2 L} \left\| \int_0^{\cdummy} S_{\cdummy - s}
     (\varepsilon^{M_{\bar{\sigma}} + 2 L_{\bar{\sigma}} - 1}
     S_{\bar{\sigma}}) P_{\leqslant R} \varphi_{\sigma} (u^{\varepsilon,
     \delta}_s, m^{\varepsilon, \delta}_s) d s \right\|_{L^2 L^2 H^{- \beta}}
  \]
  \[ \lesssim \varepsilon^{M + 2 L} \| (\varepsilon^{M_{\bar{\sigma}} + 2
     L_{\bar{\sigma}} - 1} S_{\bar{\sigma}}) P_{\leqslant R} \varphi_{\sigma}
     (u^{\varepsilon, \delta}, m^{\varepsilon, \delta}) \|_{L^2 L^2 H^{- 2 -
     \beta}} \lesssim \varepsilon^{M + 2 L} \| P_{\leqslant R}
     \varphi_{\sigma} (u^{\varepsilon, \delta}, m^{\varepsilon, \delta})
     \|_{L^2 L^2 H^{- \beta}} . \]
\end{proof}

\subsection{Boundary terms}\label{s:LHS}

\begin{lemma}
  \label{l:27}Let $\beta > 0$. For every $\sigma \in \Sigma$, the boundary
  terms satisfy
  \[ \delta^{- 1} \varepsilon^{M + 2 L} \| P_{\leqslant R} \varphi_{\sigma}
     (u^{\varepsilon, \delta}, m^{\varepsilon, \delta}) \|_{L^2 L^2 H^{-
     \beta}} \]
  \[ \lesssim \varepsilon^{M + 2 L} (\log \varepsilon^{- 1})^{M / 2}
     (\delta^{- M - 2 L + 1} + \delta^{- 1} (R^{M + 2 L - 1 - \beta} \vee 1)),
  \]
  and
  \begin{equation}
    \delta^{- 1} \varepsilon^{M + 2 L} \| S_t P_{\leqslant R} \varphi_{\sigma}
    (u_0, m_0^{\varepsilon, \delta}) \|_{L^2 L^2 H^{- \beta}} \lesssim
    \varepsilon^{M + 2 L} (\delta^{- M - 2 L} + \delta^{- 1} (R^{M + 2 L - 1 -
    \beta} \vee 1)) . \label{cond:11.2}
  \end{equation}
\end{lemma}

\begin{proof}
  The first estimate follows directly from the bound of \eqref{eq:51aa}
  established in the proof of Lemma~\ref{l:26}, after division by $\delta$.
  
  The second boundary term in the mild formulation is
  \[ \varepsilon^{M + 2 L} \| S_t P_{\leqslant R} \varphi_{\sigma} (u_0,
     m_0^{\varepsilon, \delta}) \|_{L^2 L^2 H^{- \beta}} . \]
  Since we assume only $u_0 \in L^2$, all $M + 2 L$ derivatives hit the test function or
  $M + 2 L - 1$ derivatives hit $m^{\varepsilon, \delta}$ or its covariance.
  Since $S_t v$ solves $\dot{w} = A w$, $w (0) = v$, testing against $w$ and
  integrating yields $\| S_{ \cdummy} v \|_{L^2 H^{- \beta}} \lesssim
  \| v \|_{H^{- 1 - \beta}}$. Applying this with $v = P_{\leqslant R}
  \varphi_{\sigma} (u_0, m_0^{\varepsilon, \delta})$ gives
  \[ \varepsilon^{M + 2 L} \| S_t P_{\leqslant R} \varphi_{\sigma} (u_0,
     m_0^{\varepsilon, \delta}) \|_{L^2 L^2 H^{- \beta}} \lesssim
     \varepsilon^{M + 2 L} \| P_{\leqslant R} \varphi_{\sigma} (u_0,
     m_0^{\varepsilon, \delta}) \|_{H^{- 1 - \beta}} \]
  \[ \lesssim \varepsilon^{M + 2 L} (\delta^{- M - 2 L + 1} + R^{M + 2 L - 1 -
     \beta} \vee 1) \]
  and dividing by $\delta$ yields the stated estimate.
\end{proof}

\subsection{Convective remainders}\label{s:11.3}

\begin{lemma}
  \label{l:28}Let $\beta > 0$. For every $\sigma \in \Sigma$ and $\kappa \in
  (0, \beta/2)$, the convective remainder satisfies
  \[ \delta^{- 1} \varepsilon^{M + 2 L} \left\| \int_0^{\cdummy} S_{\cdummy -
     s} P_{\leqslant R} \varphi_{\sigma} (\mathbb{P} (u_s^{\varepsilon,
     \delta} \cdummy \nabla u_s^{\varepsilon, \delta}), m_s^{\varepsilon,
     \delta}) d s \right\|_{L^2 L^2 H^{- \beta}} \]
  \begin{equation}
    \lesssim \varepsilon^{M + 2 L} (\log \varepsilon^{- 1})^{M / 2} (\delta^{-
    M - 2 L} + \delta^{- 1} (R^{M + 2 L - 1 - \beta + 2\kappa} \vee 1)) .
    \label{cond:11.3}
  \end{equation}
\end{lemma}

\begin{proof}
  We first consider the case $\sigma = (0)$, which illustrates the nonlinear
  structure of the argument, and subsequently estimate the general case
  $\sigma \in \Sigma$.
  
  The convective term associated to the first corrector $\varphi_{(0)}$ reads
  as
  \[ - \varepsilon \varphi_{(0)} (\mathbb{P} (u^{\varepsilon, \delta} \cdummy
     \nabla u^{\varepsilon, \delta}), m^{\varepsilon, \delta}) = - \varepsilon
     \langle m^{\varepsilon, \delta} \cdummy \nabla \mathbb{P}
     (u^{\varepsilon, \delta} \cdummy \nabla u^{\varepsilon, \delta}),
     \mathbb{P}h \rangle . \]
After the cut-off we estimate using the smoothing of
the semigroup with $\kappa \in (0, \beta)$ as
\[ \varepsilon \left\| \int_0^{\cdummy} S_{\cdummy - s} P_{\leqslant R}
   \varphi_{(0)} (\mathbb{P} (u_s^{\varepsilon, \delta} \cdummy \nabla
   u_s^{\varepsilon, \delta}), m_s^{\varepsilon, \delta}) d s \right\|_{L^2
   H^{- \beta}} \]
\[ \leqslant \varepsilon \left\| \int_0^{\cdummy} \| S_{\cdummy - s}
   P_{\leqslant R} \varphi_{(0)} (\mathbb{P} (u_s^{\varepsilon, \delta}
   \cdummy \nabla u_s^{\varepsilon, \delta}), m_s^{\varepsilon, \delta})
   \|_{H^{- \beta}} d s \right\|_{L^2_t} \]
\[ \lesssim \varepsilon \left\| \int_0^{\cdummy} (\cdummy - s)^{- 1 + \kappa /
   2} \| P_{\leqslant R} \varphi_{(0)} (\mathbb{P} (u_s^{\varepsilon, \delta}
   \cdummy \nabla u_s^{\varepsilon, \delta}), m_s^{\varepsilon, \delta})
   \|_{H^{- 2 - \beta + \kappa}} d s \right\|_{L^2_t} . \]
Next, we write
\[ \| P_{\leqslant R} \varphi_{(0)} (\mathbb{P} (u_s^{\varepsilon, \delta}
   \cdummy \nabla u_s^{\varepsilon, \delta}), m_s^{\varepsilon, \delta})
   \|_{H^{- 2 - \beta + \kappa}} \]
\[ = \sup_{h \in H^{2 + \beta - \kappa}, \| h \|_{H^{2 + \beta - \kappa}}
   \leqslant 1} | \langle m_s^{\varepsilon, \delta} \cdummy \nabla (\mathbb{P}
   (u_s^{\varepsilon, \delta} \cdummy \nabla u_s^{\varepsilon, \delta})),
   P_{\leqslant R} \mathbb{P}h \rangle | \]
and since in two dimensions $H^{2 + \beta - \kappa} \subset W^{1, \infty}$ for
$\kappa < \beta$, this is further bounded by
\[ \lesssim \| m^{\varepsilon, \delta}_s \|_{L^{\infty}} \| \mathbb{P}
   (u_s^{\varepsilon, \delta} \cdummy \nabla u_s^{\varepsilon, \delta})
   \|_{L^{1 +}} \lesssim \| m^{\varepsilon, \delta}_s \|_{L^{\infty}} \|
   u_s^{\varepsilon, \delta} \|_{L^{2 +}} \| \nabla u_s^{\varepsilon, \delta}
   \|_{L^2} . \]
Consequently,
\[ \varepsilon \left\| \int_0^{\cdummy} S_{\cdummy - s} P_{\leqslant R}
   \varphi_{(0)} (\mathbb{P} (u_s^{\varepsilon, \delta} \cdummy \nabla
   u_s^{\varepsilon, \delta}), m_s^{\varepsilon, \delta}) d s \right\|_{L^2
   H^{- \beta}} \]
\[ \lesssim \varepsilon \left\| \int_0^{\cdummy} (\cdummy - s)^{- 1 + \kappa /
   2} \| m^{\varepsilon, \delta}_s \|_{L^{\infty}} \| u_s^{\varepsilon,
   \delta} \|_{L^{2 +}} \| \nabla u_s^{\varepsilon, \delta} \|_{L^2} d s
   \right\|_{L^2_t} \]
\[ \lesssim \varepsilon \| m^{\varepsilon, \delta} \|_{L^{\infty} L^{\infty
   }} \left\| \int_0^{\cdummy} (\cdummy -s)^{- 1 + \kappa / 2} \|
   u_s^{\varepsilon, \delta} \|_{L^{2 +}} \| \nabla u_s^{\varepsilon, \delta}
   \|_{L^2} d s \right\|_{L^2_t}, \]
and Young's convolution inequality yields
\[ \lesssim \varepsilon \| m^{\varepsilon, \delta} \|_{L^{\infty} L^{\infty
  }} \| t \mapsto t ^{- 1 + \kappa / 2} \|_{L_t^{1 +}} \| \|
   u_t^{\varepsilon, \delta} \|_{L^{2 +}} \| \nabla u_t^{\varepsilon, \delta}
   \|_{L^2} \|_{L^{2 -}_t}, \]
where given $\kappa < \beta$, we first choose the exponent $1 +$, which
determines the exponent $2 -$ and that further forces the time integrability
required from $\| u_s^{\varepsilon, \delta} \|_{L^{2 +}}$. By interpolation
from the energy bound of $u^{\varepsilon, \delta}$ in $L^{\infty} L^2 \cap L^2
H^1$, we choose the space integrability $2 +$  so that the above is
bounded by
\[ \lesssim \varepsilon \| m^{\varepsilon, \delta} \|_{L^{\infty} L^{\infty}} (\| u^{\varepsilon, \delta} \|_{L^{\infty} L^2} + \| u^{\varepsilon,
   \delta} \|_{L^2 H^1}) \| \nabla u^{\varepsilon, \delta} \|_{L^2 L^2} . \]
This finally leads to
\[ \varepsilon  \left\| \int_0^{\cdummy} S_{\cdummy - s} P_{\leqslant R}
   \varphi_{(0)} (\mathbb{P} (u_s^{\varepsilon, \delta} \cdummy \nabla
   u_s^{\varepsilon, \delta}), m_s^{\varepsilon, \delta}) d s \right\|_{L^2
   L^2 H^{- \beta}} \lesssim \varepsilon \sqrt{\log \varepsilon^{- 1}} . \]
Dividing by $\delta$, we are led to the bound
  $\varepsilon \sqrt{\log \varepsilon^{- 1}} \delta^{- 1}$, which is vanishing
  under our assumption $\varepsilon = o (\delta^{1 + \iota})$.
  
  For a general $\sigma \in \Sigma$, we estimate with $\kappa \in (0, \beta /
2)$ as
\[ \varepsilon^{M + 2 L} \left\| \int_0^{\cdummy} S_{\cdummy - s} P_{\leqslant
   R} \varphi_{\sigma} (\mathbb{P} (u_s^{\varepsilon, \delta} \cdummy \nabla
   u_s^{\varepsilon, \delta}), m_s^{\varepsilon, \delta}) d s \right\|_{L^2
   H^{- \beta}} \]
\begin{equation}
  \lesssim \varepsilon^{M + 2 L} \left\| \int_0^{\cdummy} (\cdummy - s)^{- 1 +
  \kappa / 2} \| P_{\leqslant R} \varphi_{\sigma} (\mathbb{P}
  (u_s^{\varepsilon, \delta} \cdummy \nabla u_s^{\varepsilon, \delta}),
  m_s^{\varepsilon, \delta}) \|_{H^{- 2 - \beta + \kappa}} d s
  \right\|_{L^2_t} . \label{eq:42dd}
\end{equation}
This expression contains $M + 2 L + 1$ derivatives, out of which at most $M +
2 L$ go to the test function and at most $M + 2 L - 1$ derivatives go to
$m^{\varepsilon, \delta}$ . If $M + 2 L + 1 + \kappa > 2 + \beta - \kappa$, we
apply the Sobolev embedding and the smoothing of $h$ as
\[ \| P_{\leqslant R} \mathbb{P}h \|_{W^{M + 2 L, \infty}} \lesssim \|
   P_{\leqslant R} \mathbb{P}h \|_{H^{M + 2 L + 1 + \kappa}} \lesssim R^{M + 2
   L + 1 + \kappa - 2 - \beta + \kappa} \| h \|_{H^{2 + \beta - \kappa}} \]
\[ = R^{M + 2 L - 1 - \beta + 2 \kappa} \| h \|_{H^{2 + \beta - \kappa}} \]
and we estimate \eqref{eq:42dd} as
\[ \lesssim \varepsilon^{M + 2 L} (\log \varepsilon^{- 1})^{M / 2} (\delta^{-
   M - 2 L + 1} + R^{M + 2 L - 1 - \beta + 2 \kappa} \vee 1) \left\|
   \int_0^{\cdummy} (\cdummy - s)^{- 1 + \kappa / 2} \| u_s^{\varepsilon,
   \delta} \|_{L^{2 +}} \| \nabla u_s^{\varepsilon, \delta} \|_{L^2} d s
   \right\|_{L^2_t} . \]
With the same reasoning as in the case $\sigma = (0)$ above, we finally bound
this by
\[ \lesssim \varepsilon^{M + 2 L} (\log \varepsilon^{- 1})^{M / 2} (\delta^{-
   M - 2 L + 1} + R^{M + 2 L - 1 - \beta + 2 \kappa} \vee 1) . \]
Thus, after division by $\delta$ the claim follows.
\end{proof}

\subsection{Remaining expectations}\label{s:exp2}

We are concerned with the terms
\[ \sum_{\{ \sigma \in \Sigma, M \tmop{odd}, (\sigma, 0) \in \Sigma \}
   \setminus \{ (0) \}} \varepsilon^{M + 2 L - 1} R_{\sigma} u^{\varepsilon,
   \delta} \]
where $R_{\sigma}$ was defined in Proposition~\ref{c:c18}.

\begin{lemma}
  \label{l:31R}Let $\beta > 0$. For every $\sigma \in \Sigma \setminus \{ (0)
  \}$ with $M$ odd and $(\sigma, 0) \in \Sigma$, the remaining expectation
  satisfies
  \[ \delta^{- 1} \varepsilon^{M + 2 L - 1} \left\| \int_0^{\cdummy}
     S_{\cdummy - s} P_{\leqslant R} R_{\sigma} u_s^{\varepsilon, \delta} d s
     \right\|_{L^2 L^2 H^{- \beta}} \]
  \begin{equation}
    \lesssim \varepsilon^{M + 2 L - 1} (\delta^{- M - 2 L + 1} + \delta^{- 1}
    (R^{M + 2 L - 2 - \beta} \vee 1)) . \label{cond:12.4R}
  \end{equation}
\end{lemma}

\begin{proof}
  Applying $R_{\sigma} u$ to $P_{\leqslant R} h$ we obtain by maximal
  regularity and Proposition~\ref{c:c18}
  \[ \varepsilon^{M + 2 L - 1} \left\| \int_0^{\cdummy} S_{\cdummy - s}
     P_{\leqslant R} R_{\sigma} u_s^{\varepsilon, \delta} d s \right\|_{L^2
     L^2 H^{- \beta}} \lesssim \varepsilon^{M + 2 L - 1} \| P_{\leqslant R}
     R_{\sigma} u_s^{\varepsilon, \delta} \|_{L^2 L^2 H^{- 2 - \beta}} \]
  \[ \lesssim \varepsilon^{M + 2 L - 1} \| \nabla u^{\varepsilon, \delta}
     \|_{L^2 L^2} \sum_{j = 1, \ldots, M + 2 L - 1} \delta^{- (M + 2 L - 1 -
     j)} (R^{j + 1 - 2 - \beta} \vee 1) \]
  \[ \lesssim \varepsilon^{M + 2 L - 1} (\delta^{- M - 2 L + 2} + R^{M + 2 L -
     2 - \beta} \vee 1) \]
  and \eqref{cond:12.4R} follows after the division by $\delta$.
\end{proof}

\subsection{Stochastic integrals}\label{s:11.4}

\begin{lemma}
  \label{l:29}Let $\beta > 0$. For every $\sigma \in \Sigma \setminus \{ (0)
  \}$ the stochastic integral satisfies
  \[ \delta^{- 1} \varepsilon^{M + 2 L - 1} \left\| \int_0^{\cdummy}
     S_{\cdummy - s} P_{\leqslant R} \langle \mathcal{Q}^{1 / 2}_{\delta} d
     W_s, D_n \varphi_{\sigma} (u^{\varepsilon, \delta}_s, m^{\varepsilon,
     \delta}_s) \rangle \right\|_{L^2 L^2 H^{- \beta}} \]
  \begin{equation}
    \lesssim \varepsilon^{M + 2 L - 1} (\log \varepsilon^{- 1})^{(M - 1) / 2}
    (\delta^{- M - 2 L + 1} + \delta^{- 1} (R^{M + 2 L - 2 - \beta} \vee 1)) .
    \label{cond:11.4}
  \end{equation}
\end{lemma}

\begin{proof}
  We begin with the second corrector, which already illustrates the structure
  of the stochastic estimate. Since $\varepsilon^2 \varphi_{(0, 0)} (u, n) =
  \frac{\varepsilon^2}{2} [n \cdummy \nabla \mathbb{P} (n \cdummy \nabla
  u)]^{\circ}$, the relevant part is only the leading order term
  $\frac{\varepsilon^2}{2} n \cdummy \nabla \mathbb{P} (n \cdummy \nabla u)$
  and the associated stochastic integrals read as
  \[ \frac{\varepsilon}{2} \mathbb{P} (\mathcal{Q}^{1 / 2}_{\delta} d W
     \cdummy \nabla \mathbb{P} (m^{\varepsilon, \delta} \cdummy \nabla
     u^{\varepsilon, \delta}) + m^{\varepsilon, \delta} \cdummy \nabla
     \mathbb{P} (\mathcal{Q}^{1 / 2}_{\delta} d W \cdummy \nabla
     u^{\varepsilon, \delta})) . \]
  After the cut-off we estimate by stochastic maximal regularity
  Lemma~\ref{l:l21}
  \[ \mathbb{E} \left[ \int_0^T \left\| \frac{\varepsilon}{2} \int_0^t S_{t -
     s} P_{\leqslant R} \mathbb{P} (\mathcal{Q}^{1 / 2}_{\delta} d W_s \cdummy
     \nabla \mathbb{P} (m^{\varepsilon, \delta}_s \cdummy \nabla
     u^{\varepsilon, \delta}_s)) \right\|_{H^{- \beta}}^2 d t \right] \]
  \[ \lesssim \varepsilon^2 \mathbb{E} \left[ \int_0^T \| P_{\leqslant R}
     (\mathcal{Q}^{1 / 2}_{\delta} (\cdummy) \cdummy \nabla \mathbb{P}
     (m^{\varepsilon, \delta}_s \cdummy \nabla u^{\varepsilon, \delta}_s))
     \|_{L_2 (H, H^{- 1 - \beta})}^2 d s \right], \]
  where
  \[ \| P_{\leqslant R} (\mathcal{Q}^{1 / 2}_{\delta} (\cdummy) \cdummy \nabla
     \mathbb{P} (m^{\varepsilon, \delta}_s \cdummy \nabla u^{\varepsilon,
     \delta}_s)) \|_{L_2 (H, H^{- 1 - \beta})}^2 = \sum_{k, \alpha} \|
     P_{\leqslant R} ((\mathcal{Q}^{1 / 2}_{\delta} \sigma_{k, \alpha})
     \cdummy \nabla \mathbb{P} (m^{\varepsilon, \delta}_s \cdummy \nabla
     u^{\varepsilon, \delta}_s)) \|_{H^{- 1 - \beta}}^2 \]
  \[ \lesssim \delta^d \sum_{k, \alpha} | (\mathcal{F}_{\mathbb{R}^d} K)
     (\delta k) |^2 \| P_{\leqslant R} (\sigma_{k, \alpha} \cdummy \nabla
     \mathbb{P} (m^{\varepsilon, \delta}_s \cdummy \nabla u^{\varepsilon,
     \delta}_s)) \|_{H^{- 1 - \beta}}^2 \]
  \[ = \delta^d \sum_{k, \alpha} | (\mathcal{F}_{\mathbb{R}^d} K) (\delta k)
     |^2 \sup_{h \in H^{1 + \beta}, \| h \|_{H^{1 + \beta}} \leqslant 1} |
     \langle \sigma_{k, \alpha} \cdummy \nabla \mathbb{P} (m^{\varepsilon,
     \delta}_s \cdummy \nabla u^{\varepsilon, \delta}_s), P_{\leqslant R} h
     \rangle |^2 . \]
  Here the projector $P_{\leqslant R}$ is not needed and we estimate directly
  \[ \lesssim \delta^d \sum_{k, \alpha} | (\mathcal{F}_{\mathbb{R}^d} K)
     (\delta k) |^2 \| \mathbb{P} (m^{\varepsilon, \delta}_s \cdummy \nabla
     u^{\varepsilon, \delta}_s) \|_{L^2}^2 \]
  \[ \lesssim \| m^{\varepsilon, \delta}_s \|_{L^{\infty}}^2 \| \nabla
     u^{\varepsilon, \delta}_s \|^2_{L^2} \delta^d \sum_{k, \alpha} |
     (\mathcal{F}_{\mathbb{R}^d} K) (\delta k) |^2 \lesssim \| m^{\varepsilon,
     \delta}_s \|_{L^{\infty}}^2 \| \nabla u^{\varepsilon, \delta}_s
     \|^2_{L^2}, \]
  because $\delta^d \sum_{k, \alpha} | (\mathcal{F}_{\mathbb{R}^d} K) (\delta
  k) |^2$ is a Riemann sum approximation of the integral $\int_{\mathbb{R}^d}
  | (\mathcal{F}_{\mathbb{R}^d} K) (x) |^2 d x$.
  
  Accordingly, by Proposition~\ref{c:sup}
  \[ \mathbb{E} \left[ \int_0^T \left\| \frac{\varepsilon}{2} \int_0^t S_{t -
     s} P_{\leqslant R} (\mathcal{Q}^{1 / 2}_{\delta} d W_s \cdummy \nabla
     (m^{\varepsilon, \delta}_s \cdummy \nabla u^{\varepsilon, \delta}_s))
     \right\|_{H^{- \beta}}^2 d t \right] \lesssim \varepsilon^2 \log
     \varepsilon^{- 1} . \]
  Taking square root and dividing by $\delta$, yields the bound $\varepsilon
  \sqrt{\log \varepsilon^{- 1}} \delta^{- 1}$.
  
  Consider a general $\sigma \in \Sigma$ and the associated corrector term
  $\varepsilon^{M + 2 L} \varphi_{\sigma}$. The stochastic integral lowers the
  power of $\varepsilon$ by one and replaces one of $n$ by $\mathcal{Q}^{1 /
  2}_{\delta} d W$. It contains $M + 2 L$ derivatives, out of which at most $M
  + 2 L - 1$ hit the test function, or $M + 2 L - 2$ derivatives go to
  $m^{\varepsilon, \delta}$ or $\mathcal{Q}^{1 / 2}_{\delta}$. If $M + 2 L - 1
  > 1 + \beta$, we apply the projector as
  \[ \| P_{\leqslant R} h \|_{H^{M + 2 L - 1}} \lesssim R^{M + 2 L - 1 - 1 -
     \beta} \| h \|_{H^{1 + \beta}} = R^{M + 2 L - 2 - \beta} \| h \|_{H^{1 +
     \beta}} . \]
  Therefore, we obtain by maximal regularity
  \[ (\varepsilon^{M + 2 L - 1})^2 \mathbb{E} \left[ \int_0^T \left\| \int_0^t
     S_{t - s} P_{\leqslant R} \langle \mathcal{Q}^{1 / 2}_{\delta} d W_s, D_n
     \varphi_{\sigma} (u^{\varepsilon, \delta}_s, m^{\varepsilon, \delta}_s)
     \rangle \right\|_{H^{- \beta}}^2 d t \right] \]
  \[ \lesssim (\varepsilon^{M + 2 L - 1})^2 \mathbb{E} \left[ \int_0^T \|
     P_{\leqslant R} \langle \mathcal{Q}^{1 / 2}_{\delta} (\cdummy), D_n
     \varphi_{\sigma} (u^{\varepsilon, \delta}_s, m^{\varepsilon, \delta}_s)
     \rangle \|^2_{L_2 (H, H^{- 1 - \beta})} d s \right] \]
  \[ \lesssim (\varepsilon^{M + 2 L - 1})^2 (\log \varepsilon^{- 1})^{M - 1}
     (\delta^{- M - 2 L + 2} + R^{M + 2 L - 2 - \beta} \vee 1)^2 . \]
  Taking square root and dividing by $\delta$ yields \eqref{cond:11.4}.
\end{proof}

\subsection{Last magenta terms}\label{s:magenta}

\begin{lemma}
  \label{l:30}Let $\beta > 0$. If $\sigma \in \Sigma$ so that $(\sigma, 0)
  \notin \Sigma$ then
  \[ \delta^{- 1} \varepsilon^{M + 2 L - 1} \left\| \int_0^{\cdummy}
     S_{\cdummy - s} P_{\leqslant R} \varphi_{\sigma} (\mathbb{P}
     (m_s^{\varepsilon, \delta} \cdummy \nabla u_s^{\varepsilon, \delta}),
     m_s^{\varepsilon, \delta}) d s \right\|_{L^2 L^2 H^{- \beta}} \]
  \begin{equation}
    \lesssim \varepsilon^{M + 2 L - 1} (\log \varepsilon^{- 1})^{(M + 1) / 2}
    (\delta^{- M - 2 L} + \delta^{- 1} (R^{M + 2 L - 2 - \beta} \vee 1)) .
    \label{cond:11.5}
  \end{equation}
  If $\sigma \in \Sigma$ so that $(\sigma, 1) \notin \Sigma$ then
  \[ \delta^{- 1} \varepsilon^{M + 2 L} \left\| \int_0^{\cdummy} S_{\cdummy -
     s} P_{\leqslant R} \varphi_{\sigma} (\Delta u_s^{\varepsilon, \delta},
     m^{\varepsilon, \delta}_s) d s \right\|_{L^2 L^2 H^{- \beta}} \]
  \begin{equation}
    \lesssim \varepsilon^{M + 2 L} (\log \varepsilon^{- 1})^{M / 2} (\delta^{-
    M - 2 L - 1} + \delta^{- 1} (R^{M + 2 L - 1 - \beta} \vee 1)) .
    \label{cond:11.5b}
  \end{equation}
\end{lemma}

\begin{proof}
  If $\sigma \in \Sigma$ so that $(\sigma, 0) \notin \Sigma$ then the term
  \[ \varepsilon^{M + 2 L - 1} P_{\leqslant R} \varphi_{\sigma} (\mathbb{P}
     (m^{\varepsilon, \delta} \cdummy \nabla u^{\varepsilon, \delta}),
     m^{\varepsilon, \delta}) \]
  contains $M + 2 L + 1$ derivatives, $M + 2 L$ hitting the test function and
  $M + 2 L - 1$ hitting $m^{\varepsilon, \delta}$ or its covariance. Thus, by
  maximal regularity
  \[ \varepsilon^{M + 2 L - 1} \left\| \int_0^{\cdummy} S_{\cdummy - s}
     P_{\leqslant R} \varphi_{\sigma} (\mathbb{P} (m_s^{\varepsilon, \delta}
     \cdummy \nabla u_s^{\varepsilon, \delta}), m_s^{\varepsilon, \delta}) d s
     \right\|_{L^2 L^2 H^{- \beta}} \]
  \[ \lesssim \varepsilon^{M + 2 L - 1} \| P_{\leqslant R} \varphi_{\sigma}
     (\mathbb{P} (m^{\varepsilon, \delta} \cdummy \nabla u^{\varepsilon,
     \delta}), m^{\varepsilon, \delta}) \|_{L^2 L^2 H^{- 2 - \beta}} \]
  \[ \lesssim \varepsilon^{M + 2 L - 1} (\log \varepsilon^{- 1})^{(M + 1) / 2}
     (\delta^{- M - 2 L + 1} + R^{M + 2 L - 2 - \beta} \vee 1) \]
  and by division by $\delta$, \eqref{cond:11.5} follows.
  
  If $\sigma \in \Sigma$ so that $(\sigma, 1) \notin \Sigma$ then the same
  arguments ($M + 2 L + 2$ derivatives, $M + 2 L + 1$ to the test function, $M
  + 2 L$ to $m^{\varepsilon, \delta}$ or its covariance) lead to
  \[ \varepsilon^{M + 2 L} \left\| \int_0^{\cdummy} S_{\cdummy - s}
     P_{\leqslant R} \varphi_{\sigma} (\Delta u_s^{\varepsilon, \delta},
     m^{\varepsilon, \delta}_s) d s \right\|_{L^2 L^2 H^{- \beta}} \lesssim
     \varepsilon^{M + 2 L} \| P_{\leqslant R} \varphi_{\sigma} (\Delta
     u^{\varepsilon, \delta}, m^{\varepsilon, \delta}) \|_{L^2 L^2 H^{- 2 -
     \beta}} \]
  \[ \lesssim \varepsilon^{M + 2 L} (\log \varepsilon^{- 1})^{M / 2}
     (\delta^{- M - 2 L} + R^{M + 2 L + 1 - 2 - \beta} \vee 1), \]
  hence \eqref{cond:11.5b} follows after division by $\delta$.
\end{proof}

\subsection{Additional $z$-error}\label{s:addz}

\begin{lemma}
  \label{l:zerror}Let $\beta > 0$. For every $\sigma \in \Sigma \setminus \{
  (0) \}$ with $M$ odd and $(\sigma, 0) \in \Sigma$
  \[ \varepsilon^{M + 2 L - 1} \left\| \int_0^{\cdummy} S_{\cdummy - s} P_{\leqslant
     R} S_{\sigma} z_s d s \right\|_{L^2 L^2 H^{- \beta}} \lesssim
     \varepsilon^{M + 2 L - 1} \delta^{- M - 2 L + 1} . \]
\end{lemma}

\begin{proof}
  By \eqref{eq:sk} we have for all $\alpha \in \mathbb{R}$
  \[ \| S_{\sigma} v \|_{H^{\alpha}} \lesssim \delta^{- (M + 2 L - 1)} \| v
     \|_{H^{\alpha + 2}}, \]
  hence by maximal regularity
  \[ \varepsilon^{M + 2 L - 1} \left\| \int_0^{\cdummy} S_{\cdummy - s} P_{\leqslant
     R} S_{\sigma} z_s d s \right\|_{L^2 L^2 H^{- \beta}} \]
  \[ \lesssim \varepsilon^{M + 2 L - 1} \| S_{\sigma} z \|_{L^2 L^2 H^{- 2 -
     \beta}} \lesssim \varepsilon^{M + 2 L - 1} \delta^{- (M + 2 L - 1)} \| z
     \|_{L^2 L^2 H^{- \beta}} . \]
\end{proof}

\section{Final absorption argument and removal of cut-off}\label{s:concl}

Let $\beta \in (0, 1]$. Including the cut-off leads us to
\begin{equation}
  \| e^{- \lambda \cdummy} (z^{\varepsilon, \delta} - z) \|_{L^2 H^{- \beta}}
  \leqslant \| P_{> R} (z^{\varepsilon, \delta} - z) \|_{L^2 H^{- \beta}} + \|
  e^{- \lambda \cdummy} P_{\leqslant R} (z^{\varepsilon, \delta} - z) \|_{L^2
  H^{- \beta}} . \label{eq:bigest}
\end{equation}
For $\kappa \in (0, \beta)$, the high-frequency part is estimated by the
energy estimate \eqref{eq:energy}, the bound for $v^{\varepsilon, \delta}$
from Lemma~\ref{l:v}, and the bound for $z$ from Lemma~\ref{l:z} as
\[ \| P_{> R} (z^{\varepsilon, \delta} - z) \|_{L^2 H^{- \beta}} = \left\|
   P_{> R} \left( \frac{u^{\varepsilon, \delta} - v^{\varepsilon, \delta} -
   u}{\delta} - z \right) \right\|_{L^2 H^{- \beta}} \]
\[ \lesssim \left\| P_{> R} \frac{u^{\varepsilon, \delta} - v^{\varepsilon,
   \delta} - u}{\delta} \right\|_{L^2 H^{- \beta}} + \| P_{> R} z \|_{L^2 H^{-
   \beta}} \]
\[ \lesssim \delta^{- 1} R^{- (1 + \beta)} (\| u^{\varepsilon, \delta} \|_{L^2
   H^1} + \| v^{\varepsilon, \delta} \|_{L^2 H^1} + \| u \|_{L^2 H^1}) + R^{-
   (\beta - \kappa)} \| z \|_{L^2 H^{- \kappa}} \]
\[ \lesssim_{\omega} \delta^{- 1} R^{- (1 + \beta)} + R^{- (\beta - \kappa)} .
\]
Here the implicit constant in the last line depends on $\omega$ through the
estimate of $z$. Therefore, under the compatibility condition
\begin{equation}
  \delta^{- 1} R^{- (1 + \beta)} \rightarrow 0, \label{cond:large}
\end{equation}
the high-frequency part vanishes a.s.

Among the terms arising from the low-frequency part, it remains only to treat
the convective term stemming from $J_3$, which requires a separate absorption
argument. As discussed in Section~\ref{s:setup}, after subtracting the
equations for $v^{\varepsilon, \delta}$ and for $z$ we are led to
\eqref{eq:plo1}, \eqref{eq:42}, which read as
\begin{equation}
  e^{- \lambda \cdummy} (z^{\varepsilon, \delta} - z) \cdummy \nabla u +
  u^{\varepsilon, \delta} \cdummy \nabla e^{- \lambda \cdummy}
  (z^{\varepsilon, \delta} - z) + e^{- \lambda \cdummy} (z^{\varepsilon,
  \delta} - z) \cdummy \nabla v^{\varepsilon, \delta} \label{eq:plo}
\end{equation}
\begin{equation}
  + e^{- \lambda \cdummy} (u^{\varepsilon, \delta} - u) \cdummy \nabla z +
  e^{- \lambda \cdummy} z \cdummy \nabla v^{\varepsilon, \delta} .
  \label{eq:plo67}
\end{equation}
For the terms in \eqref{eq:plo}, we need to achieve smallness to be able to
absorb $e^{- \lambda \cdummy} (z^{\varepsilon, \delta} - z)$ into the left
hand side of \eqref{eq:bigest}. Recall that $u$, $u^{\varepsilon, \delta}$ as
well as $v^{\varepsilon, \delta}$ only belong to the energy space $L^{\infty}
(0, T ; H) \cap L^2 (0, T ; H^1)$ and so direct application of
Proposition~\ref{p:2} does not yield any smallness in time. However, by
Corollary~\ref{c:14}, $u^{\varepsilon, \delta} \rightarrow u$ in $L^2 (0, T ;
H^1)$ in probability and consequently by interpolation
\[ \| u^{\varepsilon, \delta} - u \|_{L^{2 / \beta} H^{\beta} \cap L^4 L^4}
   \ll 1 \]
provided $\varepsilon, \delta$ are small enough depending on $\omega$.
Additionally, the limit $u$ is close to its suitably chosen Galerkin
truncation $\bar{u}$ so that
\[ \| u - \bar{u} \|_{L^{2 / \beta} H^{\beta} \cap L^4 L^4} \ll 1. \]
Finally, the Galerkin truncation $\bar{u}$ belongs e.g. to $L^{\infty} (0, T ;
H^{\beta}) \cap L^{\infty} (0, T ; L^4)$ as required by
Proposition~\ref{p:22}, yielding smallness for large $\lambda$. To summarize,
\eqref{eq:plo} is bounded by Proposition~\ref{p:2}, Proposition~\ref{p:22} and
Lemma~\ref{l:v} by
\[ \left\| \int_0^{\cdummy} S_{\cdummy - s} P_{\leqslant R} [e^{- \lambda s}
   (z_s^{\varepsilon, \delta} - z_s) \cdummy \nabla u_s + u_s^{\varepsilon,
   \delta} \cdummy \nabla e^{- \lambda s} (z_s^{\varepsilon, \delta} - z_s) +
   e^{- \lambda s} (z_s^{\varepsilon, \delta} - z_s) \cdummy \nabla
   v^{\varepsilon, \delta}_s] d s \right\|_{L^2 H^{- \beta}} \]
\[ \lesssim \| e^{- \lambda \cdummy} (z^{\varepsilon, \delta} - z) \|_{L^2
   H^{- \beta}} (\| u^{\varepsilon, \delta} - u \|_{L^{2 / \beta} H^{\beta}
   \cap L^4 L^4} + \| u - \bar{u} \|_{L^{2 / \beta} H^{\beta} \cap L^4 L^4}
   \nobracket \]
\[ \nobracket + \lambda^{- \vartheta} \| \bar{u} \|_{L^{\infty} H^{\beta} \cap
   L^{\infty} L^4} + (\varepsilon \delta^{- 1})^2) . \]
First, we choose $\bar{u}$ so that the second term in the bracket above is
small enough. Second we choose $\lambda$ large enough so that the third term
in the bracket is small enough. Finally, we argue via the convergence in
probability: as $\varepsilon, \delta \rightarrow 0$, the probability that $\|
u^{\varepsilon, \delta} - u \|_{L^{2 / \beta} H^{\beta} \cap L^4 L^4}$ is
large vanishes. So with increasing probability also the first term can be
absorbed into the left hand side.

The expression \eqref{eq:plo67} is controlled by Proposition~\ref{p:2} and
Lemma~\ref{l:v} as
\[ \left\| \int_0^{\cdummy} S_{\cdummy - s} P_{\leqslant R} [e^{- \lambda s}
   (u_s^{\varepsilon, \delta} - u_s) \cdummy \nabla z_s + e^{- \lambda s}
   z_s \cdummy \nabla v_s^{\varepsilon, \delta}] d s \right\|_{L^2 H^{-
   \beta}} \]
\[ \lesssim \| z \|_{L^2 H^{- \beta}} (\| u^{\varepsilon, \delta} - u \|_{L^{2
   / \beta} H^{\beta} \cap L^4 L^4} + (\varepsilon \delta^{- 1})^2) . \]
The right hand side vanishes in probability as $\varepsilon, \delta
\rightarrow 0$: writing for any $a, \Lambda > 0$
\[ \mathbf{P} (\| z \|_{L^2 H^{- \beta}} \| u^{\varepsilon, \delta} - u
   \|_{L^{2 / \beta} H^{\beta} \cap L^4 L^4} > a) \leqslant \mathbf{P} (\|
   u^{\varepsilon, \delta} - u \|_{L^{2 / \beta} H^{\beta} \cap L^4 L^4} > a /
   \Lambda) + \mathbf{P} (\| z \|_{L^2 H^{- \beta}} > \Lambda), \]
the first term vanishes by the convergence in probability of $u^{\varepsilon,
\delta}$ and the second one by tightness of $z$ in $L^2 H^{- \beta}$.

\subsection{Choice of parameters and conclusion}\label{s:removal}

It now only remains to find $R = R (\varepsilon, \delta)$ such that the
compatibility condition \eqref{cond:large} holds and all bounds
\eqref{eq:newcond}, \eqref{cond:11.1}, \eqref{cond:11.2}, \eqref{cond:11.3},
\eqref{cond:12.4R}, \eqref{cond:11.4}, \eqref{cond:11.5}, \eqref{cond:11.5b}
vanish as $\varepsilon, \delta \rightarrow 0$.

Recall that our standing assumption is $\varepsilon = o (\delta^{1 + \iota})$
for some $\iota > 0$ and that the order of expansion $N$ was determined in
\eqref{eq:N} to guarantee $(\varepsilon \delta^{- 1})^N \delta^{- 1}
\rightarrow 0$ and to absorb a logarithmic factor. We take $R \assign
\delta^{- 1}$, which is the natural choice balancing spatial frequency
truncation against the homogenization scale. The logarithmic factors in
\eqref{cond:11.3}, \eqref{cond:11.4} are readily absorbed, since for every $p
> 0$ it holds $\varepsilon (\log \varepsilon^{- 1})^p = o (\delta)$ as a
consequence of $\varepsilon = o (\delta^{1 + \iota})$, $\iota > 0$. Indeed,
for every $\kappa \in (0, 1)$ we have $\varepsilon (\log \varepsilon^{- 1})^p
= \varepsilon^{1 - \kappa} \varepsilon^{\kappa} (\log \varepsilon^{- 1})^p = o
\left( {\delta^{(1 + \iota) (1 - \kappa)}}  \right)$ which implies
$\varepsilon (\log \varepsilon^{- 1})^p = o (\delta)$ if $\kappa$ is
sufficiently small. In the last two conditions \eqref{cond:11.5},
\eqref{cond:11.5b} we are precisely at the threshold and we make use of the
strict inequality in the definition of $N$ in \eqref{eq:N}. It implies for
every $p > 0$ that $\varepsilon (\log \varepsilon^{- 1})^p = o (\delta^{1 + d
/ (2 N)})$ and accordingly
\[ (\varepsilon \delta^{- 1} \log \varepsilon^{- 1})^N \delta^{- 1} =
   (\varepsilon \delta^{- 1 - d / (2 N)} \log \varepsilon^{- 1})^N \delta^{d /
   2} \delta^{- d / 2} \rightarrow 0. \]
Letting $R = \delta^{- 1}$ and choosing $\bar{u}$ and $\lambda > 0$ as
described above, all the error terms vanish and the convective term can be
absorbed into the left hand side. This establishes for every $\beta \in (0,
1]$ that
\[ \| z^{\varepsilon, \delta} - z \|_{L^2 H^{- \beta}} \rightarrow 0 \]
in probability as $\varepsilon, \delta \rightarrow 0$, completing the proof of
Theorem~\ref{thm:2} upon removing the exponential weight $e^{- \lambda t}$.

\begin{remark}
  \label{r:34}The restriction to $d = 2$ in Theorem~\ref{thm:2} reflects three
  independent obstructions  in $d = 3$. First, uniqueness of energy-level
  solutions to \eqref{eq:ulim} is not known in $d = 3$, making the background flow
  $u$ non-canonical. Second, the stochastic convolution is rougher in $d = 3$,
  forcing $z$ into a lower regularity class, and the convective estimate of
  Proposition~\ref{p:2} has no analogue there; these two obstructions are
  independent and cumulative. Third, the lattice sum $\sum_{| k| \leqslant
  \delta^{- 1}} | k |^{- 2} \sim \log  \delta^{- 1}$ in $d = 2$ is
  used critically in the Leray symbol analysis; in $d = 3$ the analogous sum
  grows as $\delta^{- 1}$, producing an uncontrollable divergence after the
  fluctuation rescaling.
\end{remark}

\appendix\section{Technical tools}\label{s:a}

\subsection{Gaussian process estimates}\label{s:a1}

The following estimates for suprema of stationary Gaussian processes are used
in Section~\ref{s:OU}.

\begin{lemma}
  \label{l:disc}Let $(Y_k)_{k = 1, \ldots, N}$ be random variables such that for some $\sigma>0$,
$\mathbf{P}(Y_{k}>\mathbb{E}[Y_{k}]+a)\leqslant e^{-a^{2}/2\sigma^{2}},
$
for all $a>0$ and $k\in\{1,\dots,N\}$.
  Then
  \[ \mathbb{E} [\max_{k = 1, \ldots, N} Y_k] \lesssim \max_{k = 1, \ldots, N} \mathbb{E}[Y_k]+\sigma \sqrt{1+\log N} .
  \]
\end{lemma}

\begin{proof}
Write $\bar{Y}_{k}:=Y_{k}-\mathbb{E}[Y_{k}]$, so that $\max_{k = 1, \ldots, N}Y_{k}\leqslant \max_{k = 1, \ldots, N}\mathbb{E}[Y_{k}]+M_{N}$ where
 $M_N \assign \max_{k = 1, \ldots, N} \bar{Y}_k$. The union bound gives
  $ \mathbf{P} (M_N > a) 
    \leqslant N e^{- a^2/2 \sigma^2}
     $.
 Then
  \[ \mathbb{E} [M^+_N] = \int_0^{\infty} \mathbf{P} (M_N > a) d a \leqslant
     \int_0^{\infty} \min \left\{ 1, N e^{- a^2/2 \sigma^2} \right\} d a. \]
  Split at $a_0 = \sqrt{2 \sigma^2 \log N}$. For $0 \leqslant a \leqslant a_0$
  the integrand is at most  $1$ so the contribution to the integral is at most $a_0$. For $a >
  a_0$ the integrand is $N e^{- {a^2}/{2 \sigma^2}}  $, so
  the contribution is $\lesssim \sigma$ because it is a Gaussian tail
  integral. Altogether
  \[ \mathbb{E} [M_N] \leqslant \mathbb{E} [M^+_N] \lesssim
  \sigma\sqrt{1+ \log N} .\]
\end{proof}

\begin{proposition}
  \label{p:cont}Let $T\geqslant 1$ and let $X_{t, x}$, $(t, x) \in [0, T] \times \mathbb{T}^d$, be a
  centered Gaussian field with continuous trajectories such that  $\mathbb{E} [X^2_{t, x}] \leqslant \sigma^2$ for all
  $(t, x)$ and  for some $\kappa\geqslant1$,
  \begin{equation}\label{eq:block}
  \mathbb{E}[\sup_{[j,j+1]\times\mathbb{T}^{d}}|X_{t,x}|]\leqslant \kappa \sigma
  \end{equation}
   for every $j=0,\dots,[T]$.
  Then for all $q \in [1, \infty)$,
  \begin{equation}
    (\mathbb{E} [\sup_{(t, x) \in [0, T] \times \mathbb{T}^d} | X_{t, x}
    |^q])^{1 / q} \lesssim_q \kappa\sigma \sqrt{1+\log T} . \label{eq:subG}
  \end{equation}
\end{proposition}

\begin{proof}
Writing $Z_j \assign \sup_{(t, x) \in ([j, j +
  1] \cap [0, T]) \times \mathbb{T}^d} | X_{t, x} |$ for $j = 0, \ldots, [T]$,
  where $[T]$ denotes the biggest integer smaller or equal to $T$, the blocks
  cover $[0, T]$ and therefore
  \[ S \assign \sup_{(t, x) \in [0, T] \times \mathbb{T}^d} | X_{t, x} | =
     \max_{j = 0, \ldots, [T]} Z_j . \]
Each $Z_j$ is the supremum of
  the centered Gaussian family
  \[ \{ \eta X_{t, x} ; \, \eta \in \{ \pm 1 \}, \, (t, x) \in ([j, j
     + 1] \cap [0, T]) \times \mathbb{T}^d \}, \]
  whose pointwise variances are bounded by $\sigma^2$ by assumption, and
  $\mathbb{E} [Z_j] \leqslant \kappa \sigma < \infty$ by \eqref{eq:block}.
  Hence, the Borell--TIS inequality, Theorem 2.1.1 in {\cite{MR2319516}},
  yields
  \[ \mathbf{P} (Z_j > \mathbb{E} [Z_j] + a) \leqslant e^{- a^2 / 2
     \sigma^2}, \qquad a > 0. \]
The random variables $(Z_j)_{j = 0,
  \ldots, [T]}$ satisfy the assumptions of Lemma~\ref{l:disc} with $N = [T] +
  1$, hence by \eqref{eq:block}
  \[ \mathbb{E} [S] =\mathbb{E} [\max_{j = 0, \ldots, [T]} Z_j] \lesssim
     \max_{j = 0, \ldots, [T]} \mathbb{E} [Z_j] + \sigma \sqrt{1 + \log ([T]
     + 1)} \lesssim \kappa \sigma \sqrt{1 + \log T} . \]

The supremum $S$ is itself the supremum of the
  centered Gaussian family $\{ \eta X_{t, x} ; \, \eta \in \{ \pm 1
  \}, \, (t, x) \in [0, T] \times \mathbb{T}^d \}$, with pointwise variances
  bounded by $\sigma^2$ and $\mathbb{E} [S] < \infty$ as shown above.
 The Borell--TIS inequality yields, for every $a > 0$,
  \[ \mathbf{P} (S > \mathbb{E} [S] + a) \leqslant e^{- a^2 / 2 \sigma^2},
  \]
  and therefore
  \[ (\mathbb{E} [(S -\mathbb{E} [S])_+^q])^{1 / q} = \left( \int_0^{\infty}
     q a^{q - 1} \mathbf{P} (S > \mathbb{E} [S] + a) d a \right)^{1 / q}
     \lesssim_q \sigma . \]
  Consequently, since $S \leqslant \mathbb{E} [S] + (S -\mathbb{E} [S])_+$,
  \[ (\mathbb{E} [S^q])^{1 / q} \leqslant \mathbb{E} [S] + (\mathbb{E} [(S
     -\mathbb{E} [S])_+^q])^{1 / q} \lesssim_q \kappa \sigma \sqrt{1 + \log
     T}, \]
  which is \eqref{eq:subG}.
\end{proof}

\subsection{Riemann sum approximations}\label{s:a2}

The following approximation results are used in Section~\ref{s:quali} and
Section~\ref{s:quanti}.

\begin{lemma}
  \label{l:riemann}Let $f : \mathbb{R}^d \setminus \{ 0 \} \rightarrow
  \mathbb{R}$ be continuous, bounded near zero, and satisfy $| f (k) |
  \lesssim | k |^{- d - \kappa}$ for $| k | \geqslant 1$ and some $\kappa >
  0$. Then
  \[ \delta^d \sum_{k \in \mathbb{Z}^d_0} f (\delta k) \rightarrow
     \int_{\mathbb{R}^d} f (k) d k \qquad \tmop{as} \qquad \delta \rightarrow
     0. \]
\end{lemma}

\begin{proof}
  We split into three contributions by choosing $\Lambda > 1$ large and $\rho
  \in (0, 1)$ small. Write
  \[ \left| \delta^d \sum_{k \in \mathbb{Z}^d_0} f (\delta k) -
     \int_{\mathbb{R}^d} f (k) d k \right| \leqslant \left| \delta^d \sum_{k
     \in \mathbb{Z}^d_0, \rho \leqslant | \delta k | \leqslant \Lambda} f
     (\delta k) - \int_{\rho \leqslant | k | \leqslant \Lambda} f (k) d k
     \right| \]
  \[ + \delta^d \sum_{k \in \mathbb{Z}^d_0, | \delta k | < \rho} | f (\delta
     k) | + \int_{| k | < \rho} | f (k) | d k \]
  \[ + \delta^d \sum_{k \in \mathbb{Z}^d_0, | \delta k | > \Lambda} | f
     (\delta k) | + \int_{| k | > \Lambda} | f (k) | d k \backassign
     A_{\delta, \rho, \Lambda} + B_{\delta, \rho} + C_{\delta, \Lambda} . \]
  For $A_{\delta, \rho, \Lambda}$: the function $f$ is continuous on the
  compact annulus $\{ \rho \leqslant | k | \leqslant \Lambda \}$, so the sum
  is a classical Riemann sum for the integral over this region, and
  $A_{\delta, \rho, \Lambda} \rightarrow 0$ as $\delta \rightarrow 0$ for
  fixed $\rho$, $\Lambda$.
  
  For $B_{\delta, \rho}$: by boundedness of $f$ near zero, both terms are
  $\lesssim \rho^d$, since the volume of $\{ | k | < \rho \}$ is $\lesssim
  \rho^d$ and $\delta^d$ times the number of lattice points with $| \delta k |
  < \rho$ is likewise $\lesssim \rho^d$.
  
  For $C_{\delta, \Lambda}$: by the decay assumption on $f$ and switching to
  polar coordinates we obtain
  \[ \int_{| k | > \Lambda} | f (k) | d k \lesssim \int_{| k | > \Lambda} | k
     |^{- d - \kappa} d k \lesssim \int_{\Lambda}^{\infty} r^{- d - \kappa}
     r^{d - 1} d r \lesssim \int_{\Lambda}^{\infty} r^{- 1 - \kappa} d r
     \lesssim \Lambda^{- \kappa}, \]
  and
  \[ \delta^d \sum_{k \in \mathbb{Z}^d_0, | \delta k | > \Lambda} | f (\delta
     k) | \lesssim \delta^{-\kappa} \sum_{k \in \mathbb{Z}^d_0, | \delta k | >
     \Lambda} | k |^{- d - \kappa} \lesssim \delta^{-\kappa} \sum_{k \in
     \mathbb{Z}^d_0, | \delta k | > \Lambda} \int_{k + [- 1 / 2, 1 / 2)^d} | x
     |^{- d - \kappa} d x \]
  \[ \lesssim \delta^{-\kappa} \int_{| k | > \delta^{- 1} \Lambda / 2} | k |^{- d -
     \kappa} d k \lesssim \delta^{-\kappa} \int_{\delta^{- 1} \Lambda / 2}^{\infty}
     r^{- 1 - \kappa} d r 
     \lesssim \Lambda^{- \kappa} . \]
  Choosing first $\Lambda$ large enough and $\rho$ small enough, and then
  $\delta$ small enough, gives the result.
\end{proof}

\begin{lemma}
  \label{l:13}Let $f \in C^2 (\mathbb{R}^2)$ with $D^2 f \in L^1
  (\mathbb{R}^2)$ and assume $f$ is even, i.e., $f (x) = f (- x)$ for all $x
  \in \mathbb{R}^2$. Then for every $\delta \in (0, 1]$
  \[ \left| \delta^2 \sum_{k \in \mathbb{Z}^2} f (\delta k) -
     \int_{\mathbb{R}^2} f (x) d x \right| \lesssim \delta^2 \| D^2 f \|_{L^1
     (\mathbb{R}^2)} . \]
\end{lemma}

\begin{proof}
  Partition $\mathbb{R}^2$ into cells $Q_k \assign \delta (k + [0, 1)^2)$, $k
  \in \mathbb{Z}^2$, so that $\mathbb{R}^2 = \cup_{k \in \mathbb{Z}^2} Q_k$
  and $| Q_k | = \delta^2$. Then
  \[ \delta^2 \sum_{k \in \mathbb{Z}^2} f (\delta k) - \int_{\mathbb{R}^2} f
     (x) d x = \sum_{k \in \mathbb{Z}^2} \int_{Q_k} (f (\delta k) - f (x)) d
     x. \]
  Fix $k \in \mathbb{Z}^2$ and write $x = \delta k + y$ with $y \in [0,
  \delta)^2$. Define $g (\theta) \assign f (\delta k + \theta y)$, $\theta \in
  [0, 1]$. Taylor's formula with integral remainder gives $g (1) = g (0) + g'
  (0) + \int_0^1 g'' (\theta) (1 - \theta) d \theta$, hence
  \[ f (\delta k) - f (\delta k + y) = - D f (\delta k) y - \int_0^1 D^2 f
     (\delta k + \theta y) (y, y) (1 - \theta) d \theta . \]
  Integrating over $y \in [0, \delta)^2$ yields
  \begin{equation}
    \int_{Q_k} (f (\delta k) - f (x)) d x = - D f (\delta k) \int_{[0,
    \delta)^2} y d y - \int_{[0, \delta)^2} \int_0^1 D^2 f (\delta k + \theta
    y) (y, y) (1 - \theta) d \theta d y. \label{eq:51}
  \end{equation}
  Step 1: The linear term cancels in the sum over $k$. Note that
  \[ \int_{[0, \delta)^2} y d y = \left( \frac{\delta^3}{2},
     \frac{\delta^3}{2} \right), \]
  and that since $f$ is even, $D f$ is odd, i.e., $D f (- x) = - D f (x)$.
  Therefore, the linear contribution to the global error vanishes
  \[ - \sum_{k \in \mathbb{Z}^2} D f (\delta k) \left( \frac{\delta^3}{2},
     \frac{\delta^3}{2} \right) = 0 \]
  by pairing each $k$ with $- k$.
  
  Step 2: Estimate of the quadratic remainder. We now bound the second term in
  \eqref{eq:51}. Since $| y |^2 \leqslant 2 \delta^2$,
  \[ | D^2 f (\delta k + \theta y) (y, y) | \lesssim \delta^2 | D^2 f (\delta
     k + \theta y) | . \]
  Hence
  \[ \left| \delta^2 \sum_{k \in \mathbb{Z}^2} f (\delta k) -
     \int_{\mathbb{R}^2} f (x) d x \right| \leqslant \delta^2 \sum_{k \in
     \mathbb{Z}^2} \int_{[0, \delta)^2} \int_0^1 | D^2 f (\delta k + \theta y)
     | (1 - \theta) d \theta d y. \]
  Now, observe the following covering fact: for each $\theta \in [0, 1]$, the
  sets
  \[ \delta k + \theta [0, \delta)^2 \subset \delta k + [0, \delta)^2 = Q_k \]
  are subsets of disjoint cells $Q_k$. Therefore
  \[ \sum_{k \in \mathbb{Z}^2} \int_{[0, \delta)^2} | D^2 f (\delta k + \theta
     y) | d y \leqslant \sum_{k \in \mathbb{Z}^2} \int_{Q_k} | D^2 f (z) | d z
     = \| D^2 f \|_{L^1 (\mathbb{R}^2)} . \]
  Thus, the claim follows.
\end{proof}

\subsection{Lattice sum estimates}\label{s:a3}

The following lattice sum estimate is used in Section~\ref{s:12.1}.

\begin{lemma}
  \label{lem:8}Let $\ell \in \mathbb{Z}^2 \setminus \{0\}$. For each $\delta
  \in (0, 1]$, let $w_{\delta} (\cdot, \ell) : \mathbb{Z}^2 \setminus \{0\}
  \to [0, \infty)$ be a weight satisfying:
  
  1. Evenness in $k$: $w_{\delta} (- k, \ell) = w_{\delta} (k, \ell)$ for all
  $k \neq 0$.
  
  2. Schwartz domination at scale $\delta$: there exists a Schwartz function
  $\Phi$ on $\mathbb{R}^2$, independent of $\ell$ and $\delta$, such that
  $w_{\delta} (k, \ell) \leqslant \Phi (\delta k)$ for all $k \neq 0$, $\delta
  \in (0, 1]$.
  
  Define
  \[ \Sigma_{\delta} (\ell) \assign \delta^2 \sum_{k \in \mathbb{Z}^2
     \setminus \{0\}} w_{\delta} (k, \ell) \hspace{0.17em} (\Pi (k - \ell) -
     \Pi (k)), \]
  where the sum is understood as a limit over symmetric boxes. Then
  \[ | \Sigma_{\delta} (\ell) | \lesssim | \ell |^2 \delta^2 (1 + \log
     \delta^{- 1}), \]
  where the implicit constant depends only on $\Phi$ and is independent of
  $\delta$ and $\ell$.
\end{lemma}

\begin{proof}
  Step 1: Basic bounds for $\Pi$, $D \Pi$, $D^2 \Pi$. Write componentwise $\Pi
  (p)_{i j} = \frac{p_i p_j}{|p|^2}$. Then $\Pi$ is smooth on $\mathbb{R}^2
  \setminus \{0\}$, even (i.e. $\Pi (- p) = \Pi (p)$), bounded $| \Pi (p) |
  \leqslant 1$, and homogeneous of degree $0$. Since $D \Pi$ and $D^2 \Pi$ are
  homogeneous of degrees $- 1$ and $- 2$, respectively, it follows that
  \begin{equation}
    | D \Pi | \lesssim | p |^{- 1}, \qquad | D^2 \Pi | \lesssim | p |^{- 2} .
    \label{eq:48}
  \end{equation}
  Finally, differentiating $\Pi (p) = \Pi (- p)$, it holds $D \Pi (- p) = - D
  \Pi (p)$.
  
  Step 2: Taylor expansion of $\Pi (k - \ell)$ around $k$. Let $k \in
  \mathbb{Z}^2 \setminus \{ 0 \}$. Consider the path $g (\theta) = \Pi (k -
  \theta \ell)$ for $\theta \in [0, 1]$. Then $g' (\theta) = - D \Pi (k -
  \theta \ell) \ell$, $g'' (\theta) = D^2 \Pi (k - \theta \ell) (\ell, \ell)$
  hence by Taylor's theorem with integral remainder
  \[ \Pi (k - \ell) - \Pi (k) = g (1) - g (0) = g' (0) + \int_0^1 g'' (\theta)
     (1 - \theta)  d \theta \]
  \[ = - D \Pi (k) \ell + \int_0^1 D^2 \Pi (k - \theta \ell) (\ell, \ell) (1 -
     \theta)  d \theta . \]
  If $| k | \geqslant 2 | \ell |$ then $| k - \theta \ell | \geqslant | k | -
  \theta | \ell | \geqslant | k | / 2$ for all $\theta \in [0, 1]$, and
  \eqref{eq:48} bounds the remainder as
  \[ \left| \int_0^1 D^2 \Pi (k - \theta \ell) (\ell, \ell) (1 - \theta)  d
     \theta \right| \lesssim \int_0^1 \frac{| \ell |^2}{| k - \theta \ell |^2}
     d \theta \lesssim \frac{| \ell |^2}{| k |^2} . \]
  In other words, on $\{ | k | \geqslant 2 | \ell | \}$,
  \begin{equation}
    \Pi (k - \ell) - \Pi (k) = - D \Pi (k) \ell + R (k, \ell), \qquad
    \tmop{with} | R (k, \ell) | \lesssim \frac{| \ell |^2}{| k |^2} .
    \label{eq:49}
  \end{equation}
  Step 3: Cancellation of the first order term in the symmetric sum. Decompose
  the sum $\Sigma_{\delta} (\ell) = \Sigma^{\tmop{small}}_{\delta} (\ell) +
  \Sigma^{\tmop{large}}_{\delta} (\ell)$ into small and large $k$ as
  \[ \Sigma^{\tmop{small}}_{\delta} (\ell) \assign \delta^2 \sum_{k \in
     \mathbb{Z}^2 \setminus \{0\}, | k | < 2 | \ell |} w_{\delta} (k, \ell)
     (\Pi (k - \ell) - \Pi (k)), \]
  \[ \Sigma^{\tmop{large}}_{\delta} (\ell) \assign \delta^2 \sum_{k \in
     \mathbb{Z}^2 \setminus \{0\}, | k | \geqslant 2 | \ell |} w_{\delta} (k,
     \ell) (\Pi (k - \ell) - \Pi (k)) . \]
  For the large part, use \eqref{eq:49} to get
  \[ \Sigma^{\tmop{large}}_{\delta} (\ell) = - \delta^2 \sum_{k \in
     \mathbb{Z}^2 \setminus \{0\}, | k | \geqslant 2 | \ell |} w_{\delta} (k,
     \ell) D \Pi (k) \ell + \delta^2 \sum_{k \in \mathbb{Z}^2 \setminus \{0\},
     | k | \geqslant 2 | \ell |} w_{\delta} (k, \ell) R (k, \ell) . \]
  Now, by the oddness of $D \Pi$ established in Step 1 we have $- D \Pi (k)
  \ell = D \Pi (- k) \ell$, and since $w_{\delta} (k, \ell) = w_{\delta} (- k,
  \ell)$ by assumption 1., the summand $k \mapsto w_{\delta} (k, \ell) D \Pi
  (k) \ell$ is odd. For every $k$ in the summation domain $\{ k \in
  \mathbb{Z}^2 \setminus \{0\}, | k | \geqslant 2 | \ell | \}$, there is also
  $- k$ in the domain. Hence, we conclude that the first sum above vanishes
  and
  \[ \Sigma^{\tmop{large}}_{\delta} (\ell) = \delta^2 \sum_{k \in \mathbb{Z}^2
     \setminus \{0\}, | k | \geqslant 2 | \ell |} w_{\delta} (k, \ell) R (k,
     \ell) . \]
  Step 4: Bound for the small $k$-region. Here we use $| \Pi (k) | \leqslant
  1$, hence $| \Pi (k - \ell) - \Pi (k) | \lesssim 1$ as well as the fact that
  $w_{\delta} (k, \ell) \leqslant \| \Phi \|_{L^{\infty}}$ by assumption 2. It
  follows
  \[ | \Sigma^{\tmop{small}}_{\delta} (\ell) | \lesssim \delta^2 \# \{ k \in
     \mathbb{Z}^2 \setminus \{ 0 \}, | k | < 2 | \ell | \} \lesssim \delta^2 |
     \ell |^2, \]
  since in 2 dimensions, $\# \{ | k | < R \} \lesssim R^2$.
  Step 5: Bound the large $k$-region. By domination $w_{\delta} (k, \ell)
  \leqslant \Phi (\delta k)$ and \eqref{eq:49}, we have
  \[ | \Sigma^{\tmop{large}}_{\delta} (\ell) | \lesssim \delta^2 \sum_{k \in
     \mathbb{Z}^2 \setminus \{0\}, | k | \geqslant 2 | \ell |} \Phi (\delta k)
     \frac{| \ell |^2}{| k |^2} \leqslant | \ell |^2 \delta^2 \sum_{k \in
     \mathbb{Z}^2 \setminus \{0\}} \frac{\Phi (\delta k)}{| k |^2} . \]
  Now, we estimate the last sum. Split $k$ into $| k | \leqslant \delta^{- 1}$
  and $| k | > \delta^{- 1}$. For $| k | > \delta^{- 1}$, since $\Phi$ is Schwartz it holds $\Phi(x)/|x|^{2}\lesssim (1+|x|)^{-3}$ for $|x|\geqslant 1$; as $(1+|x|)^{-3}$ varies at most by a constant factor over each cell $\delta k+[0,\delta)^{2}$, comparing the sum with the corresponding integral yields
  \[ \sum_{k \in \mathbb{Z}^2 \setminus \{0\}, | k | > \delta^{- 1}}
     \frac{\Phi (\delta k)}{| k |^2} = \delta^2 \sum_{k \in \mathbb{Z}^2
     \setminus \{0\}, | \delta k | > 1} \frac{\Phi (\delta k)}{| \delta k |^2}
     \]
     \[
     \lesssim \delta^2 \sum_{k \in \mathbb{Z}^2
     \setminus \{0\}, | \delta k | > 1} (1+|\delta k|)^{-3}
     \lesssim \int_{\mathbb{R}^{2}} (1+|x|)^{-3} d x \lesssim 1. \]
  For $1 \leqslant | k | \leqslant \delta^{- 1}$, we bound $\Phi (\delta k)
  \leqslant \| \Phi \|_{L^{\infty}}$ and the 2D lattice estimate
  \[ \sum_{1 \leqslant | k | \leqslant \delta^{- 1}} \frac{1}{| k |^2}
     \lesssim 1 + \log \delta^{- 1} . \]
  Altogether,
  \[ | \Sigma^{\tmop{large}}_{\delta} (\ell) | \lesssim | \ell |^2 \delta^2 (1
     + \log \delta^{- 1}), \]
  which completes the proof.
\end{proof}

\subsection{Convolution estimates in Lorentz spaces}\label{s:a4}

The following convolution inequality is used in Section~\ref{s:convective}.

\begin{lemma}
  \label{l:lor}Assume that $p, q \in (1, \infty)$, and that $q' \in (1,
  \infty)$ is the conjugate exponent associated to $q$. Then
  \[ \| F \ast (f g) \|_{L^p} \lesssim \| F \|_{L^{q, \infty}} \| f \|_{L^p}
     \| g \|_{{L^{q'}} } . \]
\end{lemma}

\begin{proof}
  Recall that $L^p = L^{p, p}$ and  O'Neil's convolution inequality
  {\cite{ONeil1963}} yields for
  \begin{equation}
    1 + 1 / p = 1 / q + 1 / r \label{eq:gg}
  \end{equation}
  \[ \| F \ast (f g) \|_{L^p} = \| F \ast (f g) \|_{L^{p, p}} \lesssim \| F
     \|_{L^{q, \infty}} \| f g \|_{L^{r, p}} . \]
  H{\"o}lder inequality in Lorentz spaces
  \[ \| f g \|_{L^{r, p}} \lesssim \| f \|_{L^{r_1, p_1}} \| g \|_{L^{r_2,
     p_2}}, \qquad 1 / r = 1 / r_1 + 1 / r_2, \qquad 1 / p = 1 / p_1 + 1 /
     p_2, \]
  hence in view of $L^{q'} \subset L^{q', \infty}$
  \[ \| f g \|_{L^{r, p}} \lesssim \| f \|_{{L^{p, p}} } \| g \|_{L^{q',
     \infty}} \lesssim \| f \|_{{L^p} } \| g \|_{L^{q'}} \]
  since the requirement $1 / r = 1 / p + 1 / q'$ combined with \eqref{eq:gg}
  gives
  \[ 1 / q' = 1 / r - 1 / p = 1 - 1 / q \]
  so $q'$ is indeed the conjugate exponent to $q$.
\end{proof}

\phantomsection

\

\end{document}